\documentclass[reqno, 11pt]{amsart}
\usepackage[left=2cm, right=2cm, top=2cm, bottom=2cm]{geometry}
\usepackage[T1]{fontenc}
\usepackage[utf8]{inputenc}	
\usepackage[english]{babel}
\usepackage[babel]{csquotes}
\usepackage{amssymb}
\usepackage{amsmath}
\usepackage{amsthm}
\usepackage{latexsym}
\usepackage{xcolor}
\usepackage{mathrsfs}
\usepackage{bbm}
\usepackage{verbatim}
\usepackage[style=numeric, sorting=nty, sortcites=true, backend=biber, maxbibnames=12, giveninits=true]{biblatex}
\usepackage{faktor}
\usepackage[shortlabels]{enumitem}
\usepackage[english]{babel}
\usepackage{cases} 
\usepackage[colorlinks=true, pdfstartview=FitV, linkcolor=blue, citecolor=blue, urlcolor=blue]{hyperref}
\usepackage{silence}
\numberwithin{equation}{section}
\newtheorem{thm}{Theorem}[section]

\newtheorem{lem}[thm]{Lemma}

\theoremstyle{definition}
\newtheorem{defin}[thm]{Definition}
\newtheorem{remark}[thm]{Remark}

\renewcommand{\t}{t}
\renewcommand{\d}{{\mathrm d}} % d for integrals
\renewcommand{\div}{\operatorname{div}} % divergence
\newcommand{\norm}[1]{\left\|#1\right\|} % norms
\newcommand{\ip}[2]{\left\langle #1,#2 \right\rangle} % duality
\def\laweq{\stackrel{\mathcal L}{=}}
\def\supp{\sup_{\tau \in [0,t]}}

\def\tom{\widetilde{\Omega}}
\def\tF{\widetilde{\mathscr{F}}}
\def\tP{\widetilde{\mathbb{P}}}
\def\tE{\widetilde{\mathbb{E}}}

\def\hom{\widehat{\Omega}}
\def\hF{\widehat{\mathscr{F}}}
\def\hP{\widehat{\mathbb{P}}}
\def\hE{\widehat{\mathbb{E}}}

\def\hphi{\widehat\varphi}
\def\hmu{\widehat\mu}

\def\bu{{\boldsymbol u}}

\def\bH{{\boldsymbol H}}
\def\bHs{{\boldsymbol H_\sigma}}

\def\b #1{{\boldsymbol #1}}
\def\enne{\mathbb{N}}

\def\erre{\mathbb{R}}
\def\R{\mathbb{R}}

\def\P{\mathbb{P}}

\def\E{\mathop{{}\mathbb{E}}}

\def\cL{\mathscr{L}}
\def\cF{\mathscr{F}}

\def\eps{\varepsilon}
\def\cP{\mathscr{P}}

\def\OO{\mathcal{O}}
\def\embed{\hookrightarrow}
\def\L{\mathcal{L}_{\text{Strat}}}
\def\hL{\widehat{\mathcal L}_{\text{Strat}}}
\def\hLd{\widetilde{\mathcal L}_{\text{Strat}}^{\delta}}
\def\hLdd{\widehat{\mathcal L}_{\text{Strat}}^{\delta}}
\def\Ln{\L^{\eps, n}}
\def\tLn{\widetilde{\mathcal{L}}_{\text{Strat}}^{\eps, n}}
\def\tLe{\widetilde{\mathcal{L}}_{\text{Strat}}^{\eps, \lambda}}
\def\cLe{\widecheck{\mathcal{L}}_{\text{Strat}}^{\eps}}
\def\W{\mathbb{W}_r}
\def\com{\widecheck{\Omega}}
\def\cf{\widecheck{\cF}}
\def\cp{\widecheck{\P}}
\def\bj{\b {j}}
\def\bsigma{\b\sigma}
\def\bSigma{\b\Sigma}
\def\cE{\widecheck{\E}}
\newcommand{\hOm}{\widehat{\Omega}}
\newcommand{\seq}[2]{\{{#1}_{#2}\}_{#2 \in \mathbb{N}}}

\DeclareFontFamily{U}{mathx}{}
\DeclareFontShape{U}{mathx}{m}{n}{<-> mathx10}{}
\DeclareSymbolFont{mathx}{U}{mathx}{m}{n}
\DeclareMathAccent{\widehat}{0}{mathx}{"70}
\DeclareMathAccent{\widecheck}{0}{mathx}{"71}

\def\luca #1{{\color{black} #1}}
\def\andrea #1{{\color{black} #1}}
\def\andreap #1{{\color{black} #1}}

\makeatletter
\DeclareFontFamily{OMX}{MnSymbolE}{}
\DeclareSymbolFont{MnLargeSymbols}{OMX}{MnSymbolE}{m}{n}
\SetSymbolFont{MnLargeSymbols}{bold}{OMX}{MnSymbolE}{b}{n}
\DeclareFontShape{OMX}{MnSymbolE}{m}{n}{
	<-6>  MnSymbolE5
	<6-7>  MnSymbolE6
	<7-8>  MnSymbolE7
	<8-9>  MnSymbolE8
	<9-10> MnSymbolE9
	<10-12> MnSymbolE10
	<12->   MnSymbolE12
}{}
\DeclareFontShape{OMX}{MnSymbolE}{b}{n}{
	<-6>  MnSymbolE-Bold5
	<6-7>  MnSymbolE-Bold6
	<7-8>  MnSymbolE-Bold7
	<8-9>  MnSymbolE-Bold8
	<9-10> MnSymbolE-Bold9
	<10-12> MnSymbolE-Bold10
	<12->   MnSymbolE-Bold12
}{}

\let\llangle\@undefined
\let\rrangle\@undefined
\DeclareMathDelimiter{\llangle}{\mathopen}%
{MnLargeSymbols}{'164}{MnLargeSymbols}{'164}
\DeclareMathDelimiter{\rrangle}{\mathclose}%
{MnLargeSymbols}{'171}{MnLargeSymbols}{'171}
\makeatother
\AtEveryBibitem{%
	\clearfield{doi}%
	\clearfield{url}%
	\clearfield{issn}%
	\clearfield{isbn}%
	\ifentrytype{article}
	{}
	{\clearfield{eprint}}
} % refining bibliography removing DOI (comment to reinsert)
\renewbibmacro*{publisher+location+date}{%
	\ifentrytype{book}
	{\printlist{publisher}%
		\iflistundef{location}
		{\setunit*{\addcomma\space}}
		{\setunit*{\addcomma\space}}%
		\printlist{location}}
	{\printlist{location}%
		\iflistundef{publisher}
		{\setunit*{\addcomma\space}}
		{\setunit*{\addcomma\space}}%
		\printlist{publisher}}
	\setunit*{\addcomma\space}%
	\usebibmacro{date}%
	\newunit} % change macro for books

\begin{document}
	\title[Stochastic phase separation driven by transport noise]
	{Stochastic phase separation driven by transport noise}
	\author{Andrea Di Primio, Andrea Papini \and Luca Scarpa}
	\address{Classe di Scienze, Scuola Normale Superiore, Piazza dei Cavalieri 7, 56126 Pisa, Italy}
	\email{andrea.diprimio@sns.it}
    \address{Department of Mathematical Sciences, Chalmers University of Technology and University of Gothenburg, 412 96 Gothenburg, Sweden} \email{andreapa@chalmers.se}
    \address{Dipartimento di Matematica,
		Politecnico di Milano, Via E.~Bonardi 9, 20133 Milano, Italy}
	\email{luca.scarpa@polimi.it}
	% Abstract
	\begin{abstract}
Phase separation phenomena occurring in complex fluids
are known to be sensitive to
hydrodynamic effects: when a binary fluid mixture undergoes spinodal decomposition, the emerging pattern of domains is no longer dictated by 
interfacial dynamics alone, 
but is reshaped by advection.
In several applications of interest,  
advection effects may even 
cause the onset of turbulence,
further affecting the segregation process
on a wide range of spatial scales.
Motivated by this, we propose a 
Cahn--Hilliard model driven by transport noise
able to capture the intrinsic stochasticity 
of turbulence.
The model is analyzed in its 
thermodynamically-relevant framework, 
namely employing 
a singular Flory--Huggins potential 
and a possibly degenerate mobility,
and the noise is considered both 
in It\^o and Stratonovich form.
In this work, as a first-step investigation, 
we establish well-posedness of the system
in two and three spatial dimensions
and its thermodynamical consistency.
	\end{abstract}
	\maketitle
	
	\section{Introduction} \label{sec:intro} \noindent
\subsection{A mathematical description of phase separation}
	When a multicomponent system (such as a metallic alloy or a complex fluid) is subject to a sufficiently high thermodynamical stress, usually in the form of a rapid thermal quench, it undergoes a physical process called phase separation. More specifically, the individual constituents of the system (hereafter also referred to as phases) tend to segregate rather than persist in a (quasi-)homogeneous state. This process is ubiquitous in many different disciplines, ranging from materials science to biology, where it has attracted significant interest in recent years due to the fascinating discovery that liquid-liquid phase separation is a fundamental organization mechanism in cells (cfr. \cite{Hyman}). Phase separation is typically characterized by two distinct regimes acting at different time and length scales. Initially, thermodynamical instabilities trigger a process called spinodal decomposition: at the mesoscopic spatial scale, local concentration gradients cause fast motions by uphill diffusion, that eventually are the main driver of phase separation. At a later stage, the dynamics of the system typically reaches a metastable regime called coarsening: on much longer time scales, single-phase-rich domains grow in size, until, eventually, their characteristic length may be comparable with that of the full domain $\OO$, at equilibrium. 
    
    A mathematical model capable of capturing this two-stage dynamics dates back to seminal works by J. W. Cahn and J. E. Hilliard, among which we refer to, for instance, \cite{cahn-hill, cahn-hill2, CH1961}. Therein, the following mathematical model for phase separation is proposed. Focusing without much loss of generality on binary systems, i.e., multicomponent mixtures consisting of only two different components, we may assume that a bounded, open and connected region $\OO \subset \mathbb R^{d}$, with $d \in \{2, 3\}$, is filled with two different substances, that we abstractly denote with $A$ and $B$, and that the resulting system evolves until a fixed final time $T > 0$. Introducing the rescaled concentrations
	\[
	\varphi_i: [0,T]\times\OO \to [0,1], \qquad i \in \{A, B\},
	\]
	we observe that formally
	\begin{equation} \label{eq:mass}
		\varphi_A(t,x)+ \varphi_B(t,x) = 1
	\end{equation}
	for every $x \in \OO$ and $t \in [0,T]$. Therefore, instead of deriving two coupled equations for the two phase concentrations, it is enough to  describe the dynamics of the rescaled concentration difference 
	\[
	\varphi : [0,T]\times\OO \to [-1,1] \qquad \varphi(t, x)=\varphi_A(t,x)- \varphi_B(t, x)
	\]
	observing that \eqref{eq:mass} implies that
	\[
	\varphi_A(t,x) = \dfrac{1-\varphi(t,x)}{2}, \qquad \varphi_B(t,x)= \dfrac{1+\varphi(t, x)}{2}
	\]
	for all $ x \in \OO$ and $t \in [0,T]$. The classical Cahn--Hilliard model proposed in \cite{cahn-hill} reads
	\begin{equation} \label{eq:dCHstrong}
		\begin{cases}
			\partial_t \varphi = 
            \operatorname{div}(m(\varphi)\nabla\mu) & \quad \text{in } (0,T)\times\OO, \\
			\mu = -\Delta \varphi + F'(\varphi)& \quad \text{in } (0,T)\times\OO, \\
			\partial_{\b n} \varphi = m(\varphi)\nabla\mu\cdot{\b n} = 0 & \quad \text{on } (0,T)\times\partial\OO, \\
			\varphi(0) = \varphi_0 & \quad \text{in }\OO,
		\end{cases}
	\end{equation}
    where $F:[-1,1]\to \R$ is a double-well potential
    and $m:[-1,1]\to\R$ is a mobility function.
    The thermodynamical consistency of the model 
    requires $F'$ to be singular at $\pm1$, with the 
    classical choice being the Flory-Huggins potential (cfr. \cite{Flory42, Huggins41}), namely
    \begin{equation}\label{F_log}
    F_{\text{log}}:[-1,1] \to \mathbb R,  \qquad F(x) = \dfrac{\theta}{2}[(1+x)\log(1+x)+(1-x)\log(1-x)] - \dfrac{\theta_0}{2}x^2,
    \end{equation}
    for $0 < \theta < \theta_0$, extended by continuity at the endpoints $\pm1$.
    As far as the mobility is concerned, 
    the most relevant choice 
    is given by 
    \begin{equation}
        \label{m_deg}
    m_{\text{deg}}:[-1,1] \to \mathbb R, \qquad m_{\text{deg}}(x) = 1-x^2,
    \end{equation}
    typically referred to as degenerate, since $m_{\text{deg}}(\pm1) = 0$. Non-degenerate,
    either variable or constant, mobility functions 
    are nonetheless employed in the literature.
    Equation \eqref{eq:dCHstrong} has two fundamental properties:
    \begin{enumerate}[(i)] \itemsep0.3em
        \item it is the gradient flow
    with respect to the $H^1(\OO)^*$-metric 
    of the free-energy functional 
    \[
    \mathcal E: \operatorname{dom} \mathcal E \subset H^1(\OO) \to \mathbb R, \qquad \mathcal E(v) = \dfrac{1}{2}\int_\OO |\nabla v( x)|^2 \: \d x + 
    \int_\OO F(v( x)) \: \d x,
    \]
    namely, the dynamics is driven by the minimization of the interfacial energy and the competition of mixing and demixing effects;
        \item mass conservation holds, namely, multiplying the first equation in \eqref{eq:dCHstrong} by the unitary function, we easily get by integration by parts that, for all $t \geq 0$,
        \[
        \int_\OO \varphi(t, x) \: \d x = \int_\OO \varphi_0( x) \: \d x.
        \]
    \end{enumerate}

\subsection{Accounting for convection and turbulence}
In several situations of interest, 
phase-separation processes occur 
in a moving fluid, e.g.~in the case of mixtures of immiscible liquids.
In this framework, the deterministic approach 
to phase-separation typically prescribes 
a forcing term of transport type in \eqref{eq:dCHstrong},
namely 
\begin{equation} \label{eq:dCHstrong_conv}
		\begin{cases}
			\partial_t \varphi = 
            \operatorname{div}(m(\varphi)\nabla\mu)
            +\nabla\varphi\cdot\bu
            & \quad \text{in } (0,T)\times\OO, \\
			\mu = -\Delta \varphi + F'(\varphi)& \quad \text{in } (0,T)\times\OO, \\
			\partial_{\b n} \varphi = m(\varphi)\nabla\mu\cdot{\b n}= 0 & \quad \text{on } (0,T)\times\partial\OO, \\
			\varphi(0) = \varphi_0 & \quad \text{in }\OO,
		\end{cases}
	\end{equation}
where $\bu$ represents the velocity of the fluid.
From the mathematical perspective, 
$\bu$ is usually required to be divergence-free
in order to model incompressibility, and may 
either be a fixed vector-field 
or possibly satisfy a
further equation such as Navier-Stokes or Euler,
as in the case of the celebrated H-model\andreap{, see for detail \cite[Chapter V]{RevModPhys.49.435}}.
In the former case phase-separation is usually referred to as {\em passive}, meaning that the effect of the phase-variable $\varphi$ on the fluid itself is neglected, 
whereas in the latter case the phase-separation process
is called {\em active}.

The presence of a moving fluid is known to heavily affect phase-separation dynamics: if the action of convection is sufficiently high, as in the case of strong stirring, 
the formation of the interface separating the two fluids may be delayed, or even prevented. This phenomenon is  referred to as coarsening arrest, 
and it has been extensively studied
in the recent literature, for which we refer to \cite{Berti2005,Perlekar2014,Perlekar2017, MFFIT-CH} 
and the references therein. In particular, 
phase-separation has been shown to  
be deeply influenced by turbulence.

\andreap{The mathematical modelling of turbulence requires 
the velocity field to have an intrinsic chaotic nature \cite{Tong1989,Chan1987, BCF1991, Kraichnan1968, FGV2001, MajdaKramer1999}.}
More precisely, if $\bu$ is represented in components with respect to a typical Fourier basis, roughness arises in the form of 
higher time-oscillations on the so-called 
high-modes (alternatively, on small-scales)
and lower time-oscillations on the low-modes
(i.e.~on large scales).
In order to capture the 
energy stored at small-scales in the
deterministic model \eqref{eq:dCHstrong_conv},
one of the most natural ways to 
track turbulence is to
employ a rescaling argument \cite{Berthier2014}: 
given a velocity field $\bu$, 
a sequence
of rescaled variables $(\bu_\eps)_\eps$ is introduced
to enforce high time-oscillations
and the 
dynamics of the deterministic system  
with $\bu_\eps$ are considered
when $\eps$ is small.

\andreap{Here we take a complementary point of view, motivated by the mathematical theory of turbulent transport \cite{MikuleviciusRozovskii2004,Holm2015,Memin2014, CGH2017, Flandoli2011}.} Rather than postulating a given, deterministic random velocity field acting on the phase variable, we model the effective action of a rapidly fluctuating, small-scale advecting field through a transport-type noise. 
This modeling choice is not merely postulated, but formally arises as a scaling limit of a divergence-free velocity field whose time-correlations degenerate at small scales, in the spirit of the rescaling argument as $\eps\to0$ outlined above: as the correlation time of the small-scale component of
$\bu$ vanishes, its action on the phase variable converges, in the Wong–Zakai sense \cite{WongZakai1965,WongZakai1965b}, to a white-in-time stochastic perturbation of transport type in Stratonovich form, namely
\[
\partial_t \varphi_\eps = 
            \operatorname{div}(m(\varphi_\eps)
            \nabla\mu_\eps) 
            +\nabla\varphi_\eps\cdot\bu_\eps
\qquad\leadsto\qquad
\partial_t\varphi = \div[m(\varphi)\nabla\mu] 
                + \nabla\varphi\cdot\bu\circ \dot{W},
\]
where $W$ is a given Wiener process.
\andreap{This perspective has been extensively developed in the context of fluid
dynamics, where transport noise arises as an effective description of the
action of turbulent, small-scale velocity fields on the large-scale
dynamics, both through formal arguments and rigorous two-scale model
reduction \cite{FP2020,FP2021,FP2022,DP2024}. In a suitable scaling
regime, the resulting stochastic equations converge to deterministic
parabolic equations featuring an additional eddy viscosity, for passive
scalars \cite{Galeati2020,FGL2022eddy,FGL2024} as well as for the fluid
velocity itself \cite{FGL2021scaling,FLL2024,EJP2025}. At the same time,
transport noise produces genuinely regularizing effects: enhanced and
anomalous dissipation \cite{GessYaroslavtsev2021,Agresti2025}, delayed
blow-up and global existence with large probability
\cite{FL2021,FGL2021delayed,Agresti2026,Agresti2024RD}, and
regularization-by-noise and well-posedness results ranging from the linear
transport equation \cite{FGP2010} to nonlinear models driven by rough,
Kraichnan-type noise with physically relevant spectra
\cite{BGM2025,AgrestiHussein2023,AgrestiVeraar2024}. Beyond the fluid
equations themselves, the same framework has been used to quantify
turbulence effects on coagulation processes \cite{PFH2023,FHP2024}, on the
Kelvin--Helmholtz instability \cite{FMP2024}, and in L\'evy-driven
generalisations of the averaged dissipation mechanism \cite{FPR2025}.}
%This perspective has been extensively developed in the context of fluid dynamics, where transport noise has been employed to model the effect of turbulent, small-scale velocity fields on the large-scale dynamics of both passive scalars and the fluid velocity itself \cite{FP2020, FP2021, FP2022, FGL2024, DP2024, Agresti2026, FGL2022eddy,FGL2021scaling, FLL2024, Galeati2020}, and, more broadly, to capture universal features of turbulent transport such as enhanced dissipation and regularisation \cite{PFH2023,FHP2024,FMP2024,FPR2025,EJP2025,FGL2021delayed,BGM2025, GessYaroslavtsev2021,FGP2010, FL2021, AgrestiHussein2023, AgrestiVeraar2024,Agresti2025, Agresti2024RD}.
In this direction, we refer to the monograph \cite{FL2023} for a general overview.

This motivates switching to a stochastic version of 
\eqref{eq:dCHstrong}
where the noise is structurally assumed to be in transport form, in order to model 
the effect of fluid turbulence on phase-separation dynamics. In this spirit, 
we propose two main
stochastic models for phase-separation, namely:
\begin{enumerate}[label=(\arabic*)]
		\item the stochastic Cahn--Hilliard equation with It\^o transport noise
		\begin{equation} \label{eq:ch_ito}
			\begin{cases}
				\d \varphi - \div[m(\varphi)\nabla\mu] \,\d t 
                = \nabla\varphi\cdot\b{\Sigma}\,\d W
                \quad&\text{in } (0,T)\times\OO, \\
				\mu = -\Delta \varphi + F'(\varphi) \quad&\text{in }(0,T)\times\OO, \\
                \partial_{\b n}\varphi=m(\varphi)\nabla\mu\cdot\b n=0
                \quad&\text{on }(0,T)\times\partial\OO, \\
				\varphi(\cdot \,, 0) = \varphi_0 \quad&\text{in } \OO;
			\end{cases}
		\end{equation}
		\item the stochastic Cahn--Hilliard equation with Stratonovich transport noise
		\begin{equation} \label{eq:ch_str}
			\begin{cases}
				\d \varphi - \div[m(\varphi)\nabla\mu] \,\d t 
                = \nabla\varphi\cdot\b{\Sigma}\circ\d W
                \quad&\text{in } (0,T)\times\OO, \\
				\mu = -\Delta \varphi + F'(\varphi) \quad&\text{in }(0,T)\times\OO, \\
                \partial_{\b n}\varphi=m(\varphi)\nabla\mu\cdot\b n=0
                \quad&\text{on }(0,T)\times\partial\OO, \\
				\varphi(\cdot \:, 0) = \varphi_0 \quad&\text{in } \OO.
			\end{cases}
		\end{equation}
\end{enumerate}
Here, $W$ is a cylindrical Wiener process on a fixed separable Hilbert space $U$, while $\b\Sigma$ is a linear Hilbert-Schmidt operator
from $U$ with values in a suitable subspace of the space of solenoidal vector fields with null normal component at the boundary $\partial\OO$.

Adopting this framework for the Cahn–Hilliard equation offers two main advantages. First, it provides a physically consistent derivation of the noise term, directly linked to the chaotic, multiscale nature of turbulence, rather than an ad hoc stochastic perturbation. Second, it preserves key structural features of the original, deterministic transport term: in particular, since the noise retains a divergence-free structure at the level of the driving vector fields, one formally expects 
the associated flow to respect conservation properties typical of Cahn-Hilliard dynamics.
This structural compatibility distinguishes transport noise from more general multiplicative noises considered in the literature, whose action on the equation need not respect such conservation properties, and further motivates its use as a natural stochastic counterpart of convective phase-separation.

\subsection{State of the art}
The mathematical literature on deterministic Cahn-Hilliard model has been extensively developed in the last decades:
in this regard, we refer to the 
monograph \cite{mir-CH} and the references therein for a general overview.
Above all, we point out the pioneering contribution \cite{elliott-garcke} on existence of solutions to the Cahn-Hilliard equation with degenerate mobility.
In the context of phase-separation in fluids, 
the deterministic Cahn-Hilliard equation with convection 
has been analyzed in terms of well-posedness 
\cite{col-gil-spr-conv,della-grass} and 
optimal velocity control problems
\cite{col-gil-spr-conv2, col-gil-spr-conv3, 
roc-spr, zhao-liu, zhao-liu2}, while
coupled systems have been treated e.g.~in 
\cite{abels-CHNS, frig-grass-spr, frig-grass-spr2}.

In the stochastic case, the original version of the Cahn-Hilliard equation was firstly proposed by Cook \cite{cook} and has then received considerable attention by the mathematical community.
Starting from the pioneering contribution 
\cite{daprato-deb}, several developments have been obtained in 
\cite{corn, elez-mike, scar-SCH, scar-SVCH, scar-OCSHC}.
For the thermodynamically-relevant Flory-Huggins potential, we refer to the works
\cite{deb-zamb, deb-goud, goud}
dealing with reflection measures, 
\cite{scarpa21} for the case of degenerate mobility with non-conservative noise, 
and \cite{DPGS2024} for the conservative noise case in divergence form. Stochastic Cahn-Hilliard equations with convection have been dealt with in \cite{scar-OVCSCH} in the case of passive phase-separation, while for the active case we refer to \cite{DGS, DPSZ} and \cite{TTM17, TTM19} on the Navier-Stokes equation coupled with a phase equation of Cahn-Hilliard or Allen-Cahn type.
In the context of stochastic phase-field modelling with singular potentials, we also mention the contributions \cite{bauz-bon-leb, Bertacco21, BOS21, SZ, SZ2026}.
In the recent years, considerable attention has been devoted also to the numerical approximation of stochastic Cahn-Hilliard-type equations. In this regard, we refer to 
\cite{Met2025, Met2026, BanMet2026}.

Let us stress that the available literature on the stochastic Cahn-Hilliard equation does not deal with the proposed models 
\eqref{eq:ch_ito}--\eqref{eq:ch_str}
with noise in transport form.
All the available results in the mathematical and physical community focus either on the 
classical Cahn-Hilliard-Cook model with additive noise or on specific multiplicative noise, possibly conservative, but not in transport form.
In this direction, the only available result that we are aware of is the contribution
\cite{FLZ2020}, where a Cahn-Hilliard equation 
with transport noise in Stratonovich form 
was considered as the natural counterpart of a
suitable Hele-Shaw equation perturbed by stochastic forcing: here the equation 
was analyzed from the numerical point of view, with constant mobility and polynomial potential,
suggesting the relevance of this class of perturbations, though without a systematic analytical treatment of well-posedness in the thermodynamically-relevant framework.

\subsection{Aims and novelties of this contribution} 
On account of the well-known connection between stochastic integration à la It\^{o} and à la Stratonovich, we study the models \eqref{eq:ch_ito} and \eqref{eq:ch_str} jointly, by considering a general system with It\^o-type noise and the classical 
Stratonovich correction term, namely 
\begin{equation} \label{eq:ch}
			\begin{cases}
				\d \varphi - \div[m(\varphi)\nabla\mu] \,\d t 
                +\kappa\L\varphi\,\d t
                = \nabla\varphi\cdot\b{\Sigma}\,\d W
                \quad&\text{in } (0,T)\times\OO, \\
				\mu = -\Delta \varphi + F'(\varphi) \quad&\text{in }(0,T)\times\OO, \\
                \partial_{\b n}\varphi=m(\varphi)\nabla\mu\cdot\b n=0
                \quad&\text{on }(0,T)\times\partial\OO, \\
				\varphi(\cdot \:, 0) = \varphi_0 \quad&\text{in } \OO.
			\end{cases}
		\end{equation}
In \eqref{eq:ch}, the operator $\L$
denotes the classical Stratonovich corrector, that will be rigorously introduced later as a linear operator in divergence form. The parameter $\kappa\in\{0,1\}$ encodes the type of noise that is considered in the model: if $\kappa=0$, then equation \eqref{eq:ch} comes down to \eqref{eq:ch_ito} (i.e., with It\^o noise), while if $\kappa=1$ then equation \eqref{eq:ch} 
is equivalent to \eqref{eq:ch_str} (i.e., with Stratonovich noise).

From the mathematical standpoint, the main novelty of this work is twofold. On the one hand, we show that the proposed transport noise is structurally compatible with the two fundamental features of the Cahn–Hilliard dynamics pointed out above, namely the energy balance (i) and
the mass conservation (ii).
Indeed, thanks to the divergence-free structure of the velocity field, it is shown that the transport term
does not affect the conservation of the spatial average of $\varphi$, 
exactly as in the deterministic convective case; moreover, despite the presence of noise, the natural energy functional associated with the system still satisfies a suitable (stochastic) dissipation estimate.
This is not automatic: transport noise, being multiplicative and of gradient type, generically interacts with both the conservative and dissipative structures of the equation. The fact that these are preserved is a genuine indication that transport noise constitutes a physically and mathematically consistent stochastic perturbation of the Cahn–Hilliard system, rather than an arbitrary one. 
This is a feature that, to our knowledge, is not addressed for other types of multiplicative noise considered in the literature in the relevant case of 
singular potential and degenerate mobility.
On the other hand, we carry out the well-posedness analysis in the thermodynamically relevant setting, namely with a singular (Flory–Huggins) potential $F$
 and a mobility $m$
 which may possibly degenerate at the pure phases $\varphi = \pm 1$. More precisely, we establish existence of probabilistically weak solutions both in the case of non-degenerate and degenerate mobility, and, when the mobility is constant, we further prove pathwise uniqueness of solutions, hence obtaining a fully well-posed model in this case. 
 
 The combination of the singular potential, which forces $\varphi$ to remain within the physical range $[-1,1]$, with the low regularity induced by the transport noise requires a careful choice of approximation scheme, 
 since standard techniques developed for either the deterministic convective Cahn–Hilliard equation or the stochastic Cahn–Hilliard equation with additive/multiplicative noise 
 do not directly transfer to the present setting. 
To the best of our knowledge, this is the first work providing a rigorous well-posedness analysis of the Cahn–Hilliard equation with transport noise, both in its Itô and Stratonovich formulations, in its thermodynamically relevant setting.

\subsection{Impact and further developments}
As a first step towards a rigorous stochastic theory of phase separation in the presence of turbulent transport, this work lays the mathematical foundations upon which several natural questions can be built, bridging several well-developed but so far largely disconnected theories.

A first, natural question concerns the effect of transport noise on the coarsening dynamics of the system. In the deterministic setting, the large-time coarsening behavior of the Cahn–Hilliard equation is by now well understood at the level of rigorous upper bounds: for constant mobility, the characteristic length scale $\ell(t)$ 
is known to grow no faster than $t^{1/3}$, corresponding to the Mullins–Sekerka evolution, 
while for degenerate mobility the corresponding rate is $t^{1/4}$, associated with motion by surface diffusion, as established via the energy-dissipation framework of Kohn and Otto \cite{KohnOtto2002}. 
A natural question is then whether, and how, the presence of transport noise, which, as shown in this work, is compatible with the energy dissipation structure of the system, modifies these coarsening rates: one may expect the noise to either accelerate coarsening, by enhancing effective mixing and interface merging, or to slow it down, as small-scale turbulent stirring keeps re-splitting the interface, in analogy with the coarsening arrest phenomena observed in convective Cahn–Hilliard models mentioned above. 
Investigating whether a Kohn-Otto-type upper bound persists, and whether its scaling exponent is affected by the noise, appears to be a natural and challenging extension of the present analysis.

A second, and perhaps more striking, perspective concerns the possible regularising effect of transport noise in the degenerate mobility case. In the deterministic setting, while existence of weak solutions with degenerate mobility is classical \cite{ell-gar}, uniqueness has remained a longstanding open problem for several decades, and is expected to fail in general, as is typical for this class of fourth-order degenerate parabolic equations. In light of the well-known phenomenon of regularisation by transport-noise for nonlinear equations \cite{FGP2010, GessMaurelli2018, FL2021}, 
it is natural to ask whether the presence of transport noise may restore uniqueness, or at least improve the regularity of solutions, in the degenerate mobility setting, where the deterministic theory is known to be insufficient. Establishing such a result would provide a genuinely new instance of regularisation by noise in the context of degenerate fourth-order equations, complementing the existing literature, which has so far primarily addressed regularisation phenomena for first- and second-order equations.

A further, distinctive feature of our approach lies in the generality of the vector field $\bSigma$ entering the transport noise, which is not fixed a priori to a specific structure, but can be chosen depending on the phenomenon one wishes to capture.
This flexibility opens several directions of investigation beyond the present analysis, and reflects, at the stochastic level, the classical dichotomy between passive and active phase separation recalled above.

On the one hand, choosing the components of $\bSigma$ as a divergence-free basis of $L^2(\OO)$
is the natural choice to capture the effective action of a genuinely turbulent, small-scale velocity field, and, as discussed above, is the setting in which one may expect regularisation-by-noise phenomena to take place.
In this case, the noise plays a purely passive role: it models the action of the fluid on the phase variable $\varphi$, without the phase variable itself feeding back into the dynamics of the underlying flow.

On the other hand, a genuinely different choice consists in taking the vector field $\bSigma$ as the fluid velocity field $\bu$,
which is required to satisfy some natural 
fluid-dynamics equations, e.g.~of Navier–Stokes or Euler type.
Such a choice would naturally lead to a coupled system, in which the Cahn–Hilliard equation for $\varphi$ 
perturbed by transport noise driven by $\bu$
is complemented by a stochastic evolution equation for
$\bu$ itself,
in which the phase variable $\varphi$ enters in turn as a forcing term (e.g., through a Korteweg-type stress), mirroring the classical deterministic H-model. 
From a physical standpoint, this would amount to a stochastic model of active phase separation, in which the fluctuating component of the velocity field is not merely postulated externally, but is genuinely generated by the coupled dynamics, allowing the phase-separation process and the small-scale turbulent transport to influence one another. 
From a mathematical standpoint, such a coupling raises substantial new difficulties, since the noise would no longer be a fixed, external perturbation, but would depend on the solution itself, turning the problem into one of transport noise with solution-dependent coefficients. This is a class of stochastic PDEs which, to the best of our knowledge, remains essentially unexplored, and whose analysis lies beyond the scope of the present paper.
We regard the investigation of this active counterpart of the model, together with the coarsening and regularisation questions discussed above, as one of the most promising directions for future research stemming from the present work.

\subsection{Plan of the paper} 
Section \ref{sec:main} illustrates the necessary mathematical preliminaries, such as notation and basic assumptions, as well as the main results of this work. Section \ref{sec:nd} presents the well-posedness analysis in the case of nondegenerate mobility, while Section \ref{sec:d} deals with the degenerate case.

%%%%%%%%%%%%%%%%%%%%%%%%%%%%%%%%%%%%%%%

\section{Preliminaries, assumptions and main results} 
\label{sec:main}
Throughout the paper, $\OO\subset\R^d$ is a Lipschitz bounded domain, 
with $d\in\{1,2,3\}$, and $T>0$ is a fixed final reference time. Moreover, 
$(\Omega,\cF,(\cF_t)_{t\in[0,T]},\P)$ is a
probability space endowed with a filtration
satisfying the usual conditions,
$\cP$ is the progressive $\sigma$-algebra on $\Omega\times[0,T]$, and 
$W$ is a cylindrical Wiener process on a separable Hilbert space $U$, endowed with a fixed orthonormal system $(u_k)_{k\in\enne}$.
	
\subsection{Notation and preliminaries} 
\label{ssec:notation}
For any given Banach space $E$, we denote by the bold 
symbol $\boldsymbol E$ 
the product space $E^d$. The dual space and the respective duality pairing are denoted by
$E^*$ and $\ip{\cdot}{\cdot}_{E^*,E}$.
If $E$ is a Hilbert space, its scalar product 
is indicated by $(\cdot,\cdot)_E$.
Given two separable Hilbert spaces $E_1$ and $E_2$,
the space of Hilbert-Schmidt operators 
$E_1$ to $E_2$ is denoted by $\cL^2(E,F)$,
and is endowed with the natural norm $\norm{\cdot}_{\cL^2(E_1,E_2)}$.

For every Banach space $E$ and $s\in[1,+\infty]$,
we employ the classical symbols $L^s(0,T; E)$
and $C^0([0,T]; E)$
for the usual spaces of strongly measurable Bochner-integrable functions
from $(0,T)$ to $E$, and of strongly continuous functions 
from $[0,T]$ to $E$.
If $E$ is omitted, it is understood that $E = \mathbb{R}$.
For real Sobolev spaces, we use the classical notation
$W^{s,p}(\OO)$, where $s\in\erre$ and $p\in[1,+\infty]$
and we denote by $\norm{\cdot}_{W^{s,p}(\OO)}$ their natural norms. We set $H^s(\OO):=W^{s,2}(\OO)$ for all $s\in\erre$.

We employ the symbol $\laweq$ for equality of laws
of random variables. Moreover, for any Banach space $E$
and $p \geq 1$,
we denote by $L^p(\Omega; E)$ the space of strongly measurable $E$-valued random variables 
with finite $p$th-moment. 
When $E$ is a Banach space of functions depending on time, 
we employ the symbol $L^p_\cP(\Omega;E)$
to stress that measurability is required also 
with respect to $\cP$.
We recall that if $E$ is a separable and reflexive Banach space, for all $s\in(1,+\infty)$ it holds
by \cite[Thm.~8.20.3]{edwards} that 
\[
	\left(L^{\frac{s}{s-1}}(\Omega; L^1(0,T; E))\right)^*
    =L^s_w(\Omega; L^\infty(0,T; E^*)),
\]
where
	\[
	L^s_w(\Omega; L^\infty(0,T; E^*)):=
	\left\{v:\Omega\to L^\infty(0,T; E^*) \text{ weakly*-measurable and }
	\norm{v}_{L^\infty(0,T; E^*)}\in L^s(\Omega)
	\right\}.
	\]

We recall that for the cylinrical Winer process $W$ we have the formal representation
\[
		W = \sum_{k=0}^{\infty} \beta_k u_k,
\]
where $(\beta_k)_{k \in \enne}$ is a sequence of real  independent Brownian motions. 
In order to make the sum rigorous, there always exists 
a Hilbert space $U_0$ such that the inclusion $U \embed U_0$ is Hilbert-Schmidt, hence $W$ amkes rigorous sense as
a stochastic process with trajectories in $C^0([0,T]; U_0)$. Moreover, for every separable Hilbert space $E$
and any $B\in L^2_\cP(\Omega; L^2(0,T;\cL^2(U,H)))$,
the stochastic It\^o integral 
    \[
	\int_0^\cdot B(s)\,\d W(s)
	\]
is well-defined. For details on stochastic integration wee refer to \cite[Subsec.~2.5.2]{LiuRo}.

\subsection{Variational setting}
We define
\[
H:=L^2(\OO), \qquad V_1:=H^1(\OO), \qquad
V_2:=\left\{\psi\in H^2(\OO):\;\partial_{\b n}\psi=0 \text{ a.e.~on } \partial\OO\right\},
\]
endowed with their norms $\norm{\cdot}_H$,
	$\norm{\cdot}_{V_1}$, and $\norm{\cdot}_{V_2}$ respectively.
	The Hilbert space $H$ is identified with its dual, so that we have the variational triplet
	\[
	V_2 \embed V_1\embed H \embed V_1^* \embed V_2^*,
	\]
	where the embeddings are dense and compact.
    Analogously, we introduce the
    zero-mean spaces
	\[
    V_{1,0}^*:=\{\psi\in V_1^*:\:\overline\psi:=\frac1{|\OO|}\ip{\psi}{1}_{V_1^*,V_1}=0\}
    \]
    and 
    \[ 
	H_0 := H\cap V_{1,0}= \left\{ \psi \in H : \overline{\psi} = \dfrac{1}{|\OO|}\int_{\OO} \psi \: \d x = 0\right\}, \qquad V_{1,0} := V_1 \cap H_0,
	\]
	endowed with the norms induced by $V_1^*$, $H$ and $V_1$, respectively.
    hence we still carry over the same notation. 

\subsection{The double-well potential and the mobility function} 
\label{ssec:ass_F_m}
In the following, we shall specify the main assumptions required on the double-well potential $F$ and on the mobility function $m$, starting with the former.
\begin{enumerate}[label = \textbf{(A\arabic*)}, ref = \textbf{(A\arabic*)}]
		\item\label{hyp:potential} 
        The potential $F:(-1,1)\to[0,+\infty)$ 
        admits the decomposition $F= \Psi+ R$,
		where:\vspace{0.3\baselineskip}
        \begin{itemize}
        \item the function $\Psi:(-1, 1) \to [0,+\infty)$, called singular part of $F$, is such that $\Psi\in C^2(-1,1)$. Moreover, it is a strongly convex function such that 
        \[
        \Psi'(0)=0, \qquad
        \lim_{r\to(-1)^+}\Psi'(r)=-\infty, \qquad 
        \lim_{r\to1^-}\Psi'(r)=+\infty;
        \]\vspace{0.15\baselineskip}
        \item the function $R: \erre \to \erre$, called regular part of $F$, is such that $R\in C^2(\erre)$ with bounded second derivative. Moreover, we assume that $R'(0)=0$ and we set
        \[C_R:=\|R''\|_{C^0(\erre)}.\]
        \end{itemize}
        The singular part $\Psi$ uniquely extends to 
        a proper, convex, 
        lower semicontinuous function $\Psi:\erre\to[0,+\infty]$,
        denoted with the same symbol, by simply setting
        \[
        \Psi(r):=+\infty \quad\text{if } r\in\mathbb R\setminus[-1,1], \qquad
        \Psi(-1):=\lim_{r\to(-1)^+}\Psi(r), \qquad
        \Psi(1):=\lim_{r\to1^-}\Psi(r).
        \]
        The potential $F$ is then accordingly extended, and, with no loss of generality,
        we also assume that it is nonnegative, namely $F:\erre\to[0,+\infty]$.
\end{enumerate}\noindent
Concerning the mobility function, we will assume the following hypotheses, depending on whether we are dealing with the non-degenerate or degenerate case. Here and in the following, to distinguish between the two, we employ the labels ``nd'' and ``d'' to refer to the non-degenerate and degenerate cases, respectively. Precisely, 
in the non-degenerate case we introduce the assumption
\begin{enumerate}[start=2,
label = \textbf{(A\arabic*)$_{\mathrm{nd}}$}, 
ref = \textbf{(A\arabic*)$_{\mathrm{nd}}$}]
		\item\label{hyp:mobility_nd} The mobility
        $m:[-1,1]\to[0,+\infty)$ is continuous, and there exist two constants $c_m,\,C_m>0$ such that 
        \[
        0<c_m\leq m(r)\leq C_m \quad\text{for all } r\in[-1,1];
        \]
\end{enumerate}
while, in the degenerate case, we assume
\begin{enumerate}[start=2,
label = \textbf{(A\arabic*)$_{\mathrm{d}}$}, 
ref = \textbf{(A\arabic*)$_{\mathrm{d}}$}]
		\item\label{hyp:mobility_d} The mobility
        $m:[-1,1]\to[0,+\infty)$ is Lipschitz-continuous,
        with $m(1)=m(-1)=0$,
        and there exist constants $C_m >0$ and $k_m>0$ such that 
        \[
        0< m(r)\leq C_m \quad\text{for all } r\in(-1,1),
        \qquad
        m(r)\Psi''(r)\geq k_m  \quad\text{for all } r\in(-1,1).
        \]
        Furthermore, we assume that $mF''\in L^\infty(-1,1)$ and extends to a continuous function on $[-1,1]$.
\end{enumerate}
As usual when dealing with variable mobility, we introduce the 
functional $M:(-1,1)\to[0,+\infty)$ as
\[
  M(r):=\int_0^r\int_0^s\frac1{m(t)}\,\d t\,\d s, \quad r\in(-1,1),
\]
so that $M\in C^2(-1,1)$ with $M''=\frac1{m}$, and trivially $mM''\in L^\infty(-1,1)$. Note that the second inequality in 
\ref{hyp:mobility_d} yields that the growth of
$M''=\frac{1}{m}$ at $\pm1$
is controlled from above by $\Psi''$.
\begin{remark}
    Let us point out that the assumptions above
    are classical in the context of Cahn-Hilliard equations with degenerate mobility. The relevant logarithmic potential $F_{\text{log}}$ and the degenerate mobility $m_{\text{deg}}$ defined in \eqref{F_log}
    clearly satisfy \ref{hyp:potential} and \ref{hyp:mobility_d}. Indeed, an immediate computation shows that 
    \[
    F_{\log}''(x)=\frac{\theta}{1-x^2}-\theta_0=
    \frac{\theta}{m_{\text{deg}}(x)} -\theta_0
    \quad\forall\,x\in(-1,1),
    \]
    so that in the specific example one has exactly
    $m_{\text{deg}}\Psi''_{\text{log}}\equiv\theta$.
\end{remark}

\subsection{The transport noise} 
\label{ssec:ass_noise}
In this subsection, we precise the main requirements assumed on the noise.
For $d \in \{2, 3\}$, let us define the critical Sobolev exponent $r_* := \frac{d}{2}$,
and, accordingly, for any $r > r_*$ let us define the spaces
\[
\W:=\left\{\b{u}\in \bH^\andrea{r}(\OO):\; \div\b{u}=0 \text{ in } \OO,\,
\b{u}\cdot\b{n}=0 \text{ on } \partial\OO\right\}\!,\]
endowed with the structure of normed space inherited by $\b H^r(\OO)$, and denote by $C_\infty = C_\infty(r)$ the norm of the continuous inclusion $\b H^r(\OO)\embed \b L^\infty(\OO)$. The operator $\b\Sigma: \Omega \times[0,T] \to \cL^2(U, \W)$ is completely determined by its pointwise actions on the complete orthonormal system $\seq{u}{k} \subset U$, i.e., by the family of random process $\seq{\b \sigma}{k}$ such that
    \[
  \b\Sigma(\omega,t)[u_k]:=\b{\sigma}_k(\omega,t),
  \qquad \forall \: (\omega,t)\in\Omega\times[0,T]
    \]
for all $k \in \mathbb N$. It is therefore natural to formulate the needed assumptions in terms of $\seq{\b \sigma}{k}$. In particular, in the non-degenerate case, we require
\begin{enumerate}[start=3,label = \textbf{(A\arabic*)$_\mathrm{nd}$}, 
ref = \textbf{(A\arabic*)$_\mathrm{nd}$}]
		\item\label{hyp:noise} 
The family of random vector fields $\seq{\b \sigma}{k}$ is a sequence of $\W$-valued progressively measurable processes for some $r > r_*$ and satisfies
    \[
    \sup_{t\in[0,T]}
    \sum_{k\in\enne}
    \|\b\sigma_k(t)\|^2_\bH \in L^{\frac q2}(\Omega), \qquad
    \exp\left(\int_0^T\left(
    \sum_{k\in\enne}
    \|\b\sigma_k(t)\|^2_{\W}
    \right)^{2}\,\d t\right)
    \in L^\ell(\Omega),
    \]
    for some fixed $q>2$ and for every $\ell\in[1,+\infty\mathclose)$;
\end{enumerate}
while in the degenerate case we assume
\begin{enumerate}[start=3,label = \textbf{(A\arabic*)$_\mathrm{d}$}, 
ref = \textbf{(A\arabic*)$_\mathrm{d}$}]
    \item\label{hyp:noise_d} 
The family of random vector fields $\seq{\b \sigma}{k}$ is a sequence of $\W$-valued progressively measurable processes for some $r > r_*$ and there exists a constant
$C_{\b\Sigma}>0$ such that
    \[
    \sum_{k\in\enne}
    \|\b\sigma_k(t)\|^2_{\W}
    \leq C_{\b\Sigma}^2\quad\text{a.e.~in } \Omega\times(0,T),
    \]
and $C_{\b\Sigma}<\frac{2k_m}{C_\infty^2}$ if $\kappa=0$. 
\end{enumerate}
On account of the previous assumptions, the process $\b \Sigma$ enjoys some regularity properties. More precisely, under Assumption \ref{hyp:noise}, it holds that
\[
\b\Sigma\in L^q_\cP(\Omega; 
L^\infty(0,T; \cL^2(U,\bH))\cap L^4(0,T; \cL^2(U,\W)))
\]
and
\[
    \exp\left(\norm{\b\Sigma}^4_{
    L^4(0,T;\cL^2(U,\W))}\right)
    \in L^\ell(\Omega) \qquad\forall\,\ell\in[1,+\infty\mathclose),
\]
while under Assumption \ref{hyp:noise_d} we also have 
\[
\b\Sigma\in L^\infty_\cP(\Omega\times(0,T);\cL^2(U,\W)).
\]
It is also convenient to define the multiplicative noise coefficient in transport form as an operator, i.e., 
	\[
	G:\Omega\times[0,T]\times V_1\to \cL^2(U,H)
	\]
    given by
    \[
    G(\omega,t,\psi)[u_k]:=\nabla\psi\cdot\b\sigma_k(\omega,t)
    \qquad \forall \: k\in\enne, \qquad \forall \: (\omega,t,\psi)\in\Omega\times[0,T]\times V_1.
    \]
    More compactly, we will employ the contracted notation 
    \[
    G(\omega,t,\psi)=\nabla\psi\cdot\b\Sigma(\omega,t)
    \qquad \forall \: (\omega,t,\psi)\in\Omega\times[0,T]\times V_1.
    \]
\begin{remark}[Deterministic time-independent velocities]
    A possible choice of vector fields satisfying 
    \ref{hyp:noise} is given by a family 
    $\seq{\b \sigma}{k}\subseteq\W$, non-random and independent of time,
    such that 
    \[
    C_{\b\Sigma}^2:=
    \sum_{k\in\enne}\|\b\sigma_k\|^2_{\W}<+\infty.
    \]
    If such series satisfies also the smallness condition
    $C_{\b\Sigma}\in(0, \frac{2k_m}{C_\infty^2})$ then also \ref{hyp:noise_d} holds.
    Note that the regularity $\W$
    is natural in the context of valocity fields since, for example,
    one has that $D(\b A)\subseteq \W$,
    where $D(\b A)$ is the effective domain of the Stokes operator on $\OO$.
    A particular case of this choice is given by 
    $\b\sigma_k=\alpha_k\b u$,
    where $\seq{\alpha}{k}\in\ell^2$ and 
    $\b u\in D(\b A)$ (or $\b u\in \W$).
\end{remark}
\begin{remark}[Deterministic time-dependent velocities]
    Assumption \ref{hyp:noise} is also satisfied 
    by a family of non-random vector fields,
    possibly dependent of time 
    $\seq{\b \sigma}{k} \subset L^\infty(0,T;\bH)\cap L^4(0,T; \W)$, such that 
    \[
    \sup_{t\in[0,T]}\sum_{k\in\enne}\|\b\sigma_k(t)\|^2_{\bH}
    +\int_0^T \left(
    \sum_{k\in\enne}\|\b\sigma_k(t)\|^2_{\W}
    \right)^2\,\d t
    <+\infty.
    \]
    A particular case is given by
    $\b\sigma_k=\alpha_k\b u$, where $(\alpha_k)_k\in\ell^2$
    and $\b u\in L^\infty(0,T;\bH)\cap L^4(0,T;\W)$. 
    Let us spend a few words on the regularity 
    required on $\b u$ and compare it 
    to the typical parabolic regularity generally achieved in PDEs arising from fluid dynamics. Let $\bHs \subset \b L^2(\OO)$ denote the space of square-integrable solenoidal vector fields and set
    \[
    \mathbb W := \bigcup_{r >r_*} \W.
    \]
    In the two-dimensional case we have $\mathbb W = \b H^{1+}(\OO) \cap \bHs$ and the regularity $\b u\in L^4(0,T; \b H^{1+}(\OO))$ is achieved, for example by interpolation, by any function $\b u$ solving the deterministic Navier-Stokes equation with initial datum $\b u_0\in \bH^{\frac12+}(\OO)$, since  
    \[
    L^\infty(0,T; \bH^{\frac12+}(\OO))\cap L^2(0,T; \bH^{\frac32}(\OO))
    \subseteq L^4(0,T; \bH^{1+}(\OO)).
    \]
    Therefore, in two dimensions, the regularity required by Assumption \ref{hyp:noise} on the velocity fields $\seq{\b \sigma}{k}$ 
    is then somehow in between the typical  parabolic regularities 
    of weak and strong solutions to \andreap{the Navier-Stokes \cite{koch_tataru_2001_wellposedness}}.
    As one may expect, \andreap{in the three-dimensional case the situation is worse \cite{koch_tataru_2001_wellposedness,planchon_1996_global_strong_sobolev, kato_1984_strong_lp},}  since
    one now has that $\mathbb W = \b H^{\frac 32+}(\OO) \cap \bHs$, and
    the regularity $\b u\in L^4(0,T; H^{\frac32+}(\OO))$
    is achieved by a solution $\b u$ to a deterministic Navier-Stokes equation
    with initial datum $\b u_0\in \bH^{1+}(\OO)$, since  
    \[
    L^\infty(0,T; \bH^{1+}(\OO))\cap L^2(0,T; \bH^{2}(\OO))
    \subseteq L^4(0,T; \bH^{\frac32+}(\OO)).
    \]
    The regularity required by Assumption \ref{hyp:noise} on the velocity fields $\seq{\b \sigma}{k}$ is then slightly stronger than the 
    typical parabolic regularity of strong solutions to the Navier--Stokes equations. In the case of degenerate mobility, Assumption \ref{hyp:noise_d} forces a smallness condition on the $\ell^2$-norm of $\seq{\alpha}{k}$ and the regularity 
    $\b u\in L^\infty(0,T; \W)$. 
    In two dimensions, this regularity requirement is met for example by a solution $\b u$ to a deterministic Navier--Stokes equation with initial datum in $\bH^{1+}(\OO)$, hence with a slightly stronger regularity 
    than a \andreap{strong solution \cite{fujita_kato_1964_initial_value,koch_tataru_2001_wellposedness}}. In three dimensions, 
    the initial datum is required to lie in 
    $\bH^{\frac32+}(\OO)$ and the needed regularity of $\b u$ is significantly stronger than the one of \andreap{strong solutions
    to the Navier--Stokes equation \cite{planchon_1996_global_strong_sobolev, koch_tataru_2001_wellposedness}}. For these reasons, the most natural setting for the degenerate mobility case is actually the one of non-random, time-independent velocities described in the previous remark.
%    \luca{to do: Controllare tutte le regolarità di NS con dati iniziali Sobolev belli e aggiungere referenze per regolarità.Cercare se ci sono analoghi risultati per Eulero e commentare.}
\end{remark}
	
\begin{remark}[Properties of the multiplicative noise operator]
\label{rmk:prop_noise}
    Let us note straightaway that the noise coefficient is 
    pathwise mass-conserving, meaning that
    \[
    G:\Omega\times[0,T]\times V_1\to \cL^2(U,H_0).
    \]
    This follows from the fact that, for every $k\in\enne$
    and $\psi\in V_1$, it holds that 
    \[
    \int_\OO\nabla\psi(x)\cdot\b\sigma_k(x)\,\d x
    =\int_\OO\div[\psi\b\sigma_k](x)\,\d x
    =\int_{\partial\OO}\psi(x)\b\sigma_k(x)\cdot\b{n}(x)\,\d S(x)=0,
    \]
    thanks to the the definition of the space $\W$.
    Furthermore, we point out that the restriction of $G$ to 
    $\Omega\times[0,T]\times V_2$ enjoys more regularity, namely it holds that 
    \[
    G:\Omega\times[0,T]\times V_2\to \cL^2(U,V_{1,0}).
    \]
    Indeed, for every $k\in\enne$ and $\psi\in V_2$ one has that 
    \[
    \nabla[\nabla\psi\cdot\b\sigma_k]
    =(D^2\psi)\b\sigma_k + (D\b\sigma_k)\nabla\psi
    \]
    so that, by recalling the embedding $V_1\embed L^s(\OO)$
    for all $s\geq1$ if $d=2$ and for $s=6$ if $d=3$, 
    the H\"older inequality yields
    \[
    \norm{\nabla[\nabla\psi\cdot\b\sigma_k]}_{\bH}
    \leq \|\psi\|_{V_2}\|\b\sigma_k\|_{\b{L}^\infty(\OO)}+
    \|D\b\sigma_k\|_{\b{L}^\andrea{p}(\OO)}\|\psi\|_{V_2},
    \]
    where $p>2$ if $d=2$ and $p=3$ if $d=3$. It follows from
    \ref{hyp:noise} that  
    \[
    \norm{G(\omega,t,\psi)}_{\cL^2(U,V_1)}\leq L_{\b\Sigma}(\omega,t)\|\psi\|_{V_2}
    \quad\forall\,(\omega,t,\psi)\in\Omega\times[0,T]\times V_2,
    \]
    where the process $L_{\b\Sigma}\in L^q_\cP(\Omega; L^2(0,T))$ under 
    \ref{hyp:noise} and $L_{\b\Sigma}\in L^\infty_\cP(\Omega\times(0,T))$ under \ref{hyp:noise_d}.
\end{remark}

\subsection{The Stratonovich correction}
The Stratonovich correction operator 
$\L:\Omega\times[0,T]\times V_1\to V_1^*$ is defined in variational form as 
\[
\ip{\L(\omega,t,\psi_1)}{\psi_2}_{V_1}:=
\frac12\sum_{k\in\enne}\int_\OO
[\nabla\psi_1(x)\cdot\b{\sigma}_k(\omega,t,x)]
[\nabla\psi_2(x)\cdot\b{\sigma}_k(\omega,t,x)]\:\d x,
\]
for all $(\omega,t)\in\Omega\times[0,T]$ and all $\psi_1,\,\psi_2\in V_1$.
Let us note that $\L$ is actually well defined 
both under
\ref{hyp:noise} and \ref{hyp:noise_d}.
In particular, it is immediate to check that  
\[
  \|\L(\omega,t,\psi)\|_{V_1^*} \leq L^2_{\b\Sigma}(\omega,t)\|\nabla\psi\|_{\bH}
  \qquad\forall\,(\omega,t,\psi)\in \Omega\times[0,T]\times V_1
\]
and 
\[
  \ip{\L(\omega,t,\psi)}{\psi}_{V_1}=
\frac12\sum_{k\in\enne}
\norm{\nabla\psi\cdot\b{\sigma}_k(\omega,t,x)}_H^2\geq0
\qquad \forall\,(\omega,t,\psi)\in \Omega\times[0,T]\times V_1.
\]
Furthermore, it is immediate to see that the operator $\L$ is also pathwise mass-conservative, 
i.e., that we have $\L:\Omega\times[0,T]\times V_1\to V_{1,0}^*$. Analogously to the transport noise operator discussed above, 
it is possible to see via integration by parts that the restriction of the 
Stratonovich correction $\L$ to $V_2$ enjoys more regularity, 
i.e.,~one has the strong formulation 
$\L:\Omega\times[0,T]\times V_2\to H_0$ given by
\[
\L(\omega,t,\psi)=-\frac12\sum_{k\in\enne}
[\b{\sigma}_k(\omega,t)\cdot\nabla][
\b{\sigma}_k(\omega,t)\cdot\nabla\psi],
\qquad \forall \: (\omega,t,\psi)\in \Omega\times[0,T]\times V_2.
\]
Finally, we point out a notational remark. If 
$\widehat{\b\Sigma}$ is a stochastic process, possibly defined on another 
filtered probability space $(\widehat\Omega, \widehat\cF, (\widehat\cF_t)_{t\in[0,T]}, \widehat\P)$, such that 
$\b\Sigma\laweq\widehat{\b\Sigma}$, then
we will denote the corresponding induced Stratonovich correction operator by the symbol
$\hL:\widehat\Omega\times[0,T]\times V_1\to V_1^*$. Clearly, all the considerations made above still hold for $\hL$ by equality of laws.

    \subsection{Main results: the case of non-degenerate mobility}
	First, we state the concepts of probabilistically strong and weak solutions
    for problem \eqref{eq:ch} in the case of non-degenerate mobility.
	\begin{defin} 
    \label{def:sol-nd}
    Let $\kappa\in\{0,1\}$ and let Assumptions \ref{hyp:potential}, \ref{hyp:mobility_nd} and \ref{hyp:noise} hold. Let further
		\begin{equation}
		    \label{init}
		\varphi_0 \in V_1, \qquad
		F(\varphi_0) \in L^1(\OO), \qquad
		\overline{\varphi_0}\in(-1,1).
		\end{equation}
    \begin{enumerate}[(i)] \itemsep0.5em
    \item \label{def:strong-nd}A probabilistically-strong solution to the
		Cahn--Hilliard equation \eqref{eq:ch} 
        with non-degenerate mobility starting from the initial datum $\varphi_0$ is a stochastic process $\varphi$ such that
		\begin{align}
			\label{phi}
			&\varphi \in L^p_\cP(\Omega; C^0([0,T]; H))\cap
			L^p_w(\Omega; L^\infty(0,T; V_1)) \cap
			L^p_\cP(\Omega; L^2(0,T; V_2),\\
			& |{\varphi}(\omega, t,x)| < 1
			\text{ for a.a.~}(\omega,t,x) 
            \in \Omega \times (0,T)\times \OO, \\
			\label{mu}
			&\mu:=-\Delta\varphi+F'(\varphi)
			\in L^{\frac p2}_\cP(\Omega; L^2(0,T; V_1)),
			\\
			\label{initial}
			&\varphi(0) = \varphi_0,
		\end{align}
		for every $p\in[2,q)$, and
		\begin{multline} \label{eq:ch-str-nd}
			( \varphi(t),\psi)_H +
			\int_0^t\left[\int_{\OO}
			m(\varphi(s))\nabla  \mu(s)\cdot \nabla \psi
            \,\d x
            +\kappa\ip{\L(s,\varphi(s))}{\psi}_{V_1}
            \right]\,\d s\\
			= ( \varphi(0),\psi)_{H} +
			\left(\int_0^t \nabla\varphi(s)\cdot\b\Sigma(s)\,\d  W(s), \psi\right)_{H}
		\end{multline}
        for every $\psi\in V_1$, $t \in [0,T]$ and ${\P}$-almost surely.
		\item A probabilistically-weak solution to the
		Cahn--Hilliard equation \eqref{eq:ch} 
        with non-degenerate mobility starting from the initial datum $\varphi_0$ is a a family $(\widehat\Omega,\, \widehat\cF,\,(\widehat\cF_t)_{t\in[0,T]},\,
    \widehat\P,\,
	\widehat W,\, \widehat{\b\Sigma},\, \widehat\varphi)$,
	where $(\widehat\Omega, \widehat\cF, (\widehat\cF_t)_{t\in[0,T]}, \widehat\P)$
	is a filtered
	probability space satisfying the usual conditions, $\widehat W$
	is a cylindrical Wiener processes on $U$, 
    the processes $\widehat\varphi$ and $\widehat{\b\Sigma}$ satisfy       
		\begin{align}
			\label{phi_hat}
			&\widehat\varphi \in L^p_\cP(\widehat\Omega; C^0([0,T]; H))\cap
			L^p_w(\widehat\Omega; L^\infty(0,T; V_1)) \cap
			L^p_\cP(\widehat\Omega; L^2(0,T; V_2)\andreap{)},\\
			& |\widehat{\varphi}(\omega, t,x)| < 1
			\text{ for a.a.~}(\omega,t,x) 
            \in \widehat\Omega \times (0,T)\times \OO, \\
			\label{mu_hat}
			&\widehat\mu:=-\Delta\widehat\varphi+F'(\widehat\varphi)
			\in L^{\frac p2}_\cP(\widehat\Omega; L^2(0,T; V_1)),
			\\
			\label{initial_hat}
			&\widehat\varphi(0) = \varphi_0,\\
            \label{sigma_hat}
        &\widehat{\b\Sigma} \in L^q_\cP(\widehat\Omega; L^2(0,T; \cL^2(U,\W))), \qquad
        \widehat{\b\Sigma}\laweq \b\Sigma,
		\end{align}
		for every $p\in[2,q)$, and
		\begin{multline} \label{eq:ch-wea-nd}
			( \widehat\varphi(t),\psi)_H +
			\int_0^t\left[\int_{\OO}
			m(\widehat\varphi(s))\nabla \widehat\mu(s)\cdot \nabla \psi\,\d x
            +\kappa\ip{\hL(s,\widehat\varphi(s))}{\psi}_{V_1}\right]\,\d s\\
			= ( \widehat\varphi(0),\psi)_{H} +
			\left(\int_0^t \nabla\widehat\varphi(s)\cdot
            \widehat{\b\Sigma}(s)\,\d  \widehat W(s), \psi\right)_{H}
		\end{multline}
		for every $\psi\in V_1$, $t \in [0,T]$ and $\widehat\P$-almost surely.
        \end{enumerate}
	\end{defin} \noindent
    Next, we present the main existence and uniqueness result in the case of non-degenerate mobility.
    \begin{thm}
        \label{th:nd}
        Let $\kappa\in\{0,1\}$ and let Assumptions \ref{hyp:potential}, \ref{hyp:mobility_nd} and \ref{hyp:noise} hold.
        Then, the following hold:
        \begin{enumerate}[(i)] \itemsep0.3em
        \item for every $\varphi_0$ satisfying \eqref{init},
        the Cahn--Hilliard equation \eqref{eq:ch} 
        with non-degenerate mobility 
        admits a probabilistically-weak solution starting from the initial datum $\varphi_0$,
        in the sense of Definition~\ref{def:sol-nd};
        \item if the mobility $m$ is constant, the
        for every $\varphi_0$ satisfying \eqref{init},
        the Cahn--Hilliard equation \eqref{eq:ch} 
        with non-degenerate mobility 
        admits a unique probabilistically-strong solution,
        starting from the initial datum $\varphi_0$,
        in the sense of Definition~\ref{def:sol-nd};
        \item if the mobility $m$ is constant, then
        for every $\varphi_{0,1},\,\varphi_{0,2}$ satisfying \eqref{init} with $\overline{\varphi_{0,1}}=\overline{\varphi_{0,2}}$
        and for every $\b\Sigma_1,\,\b\Sigma_2$ satisfying \ref{hyp:noise}, 
        for every $p\in[2,q)$ there exists a constant 
        $C_p>0$ such that 
        the respective probabilistically-strong solutions 
        $\varphi_1$ and $\varphi_2$
        of the Cahn--Hilliard equation \eqref{eq:ch} 
        with non-degenerate mobility 
        satisfy 
    \begin{multline}
        \label{cont_dep}
    \qquad \qquad \norm{\varphi_1-\varphi_2}_{L^p_\cP(\Omega;C^0([0,T]; V_1^*)\cap L^2(0,T; V_1))}\leq C_p\left[
    \norm{\varphi_{0,1}-\varphi_{0,2}}_{V_1^*}\right.\\
    \left.+\kappa
    \norm{\b\Sigma_1-\b\Sigma_2}_{L^q_\cP(\Omega; L^2(0,T; \cL^2(U,\W)))}+
    \norm{\b\Sigma_1-\b\Sigma_2}_{L^q_\cP(\Omega; L^2(0,T; \cL^2(U,\bH)))}
    \right].
    \end{multline}
    \end{enumerate}
    \end{thm}

	\subsection{Main results: the case of degenerate mobility}
	Analogously, we state here the concept of probabilistically strong and weak solutions
    for problem \eqref{eq:ch} in the case of degenerate mobility. The main idea to handle the degeneracy of the mobility consists in formally plugging the equation for $\mu$ in \eqref{eq:ch} directly into the equation for $\varphi$. The resulting formulation 
    does not involve the variable $\mu$ anymore, as one may expect, and the resulting terms arising from such substitution (see equation \eqref{eq:ch2-str-d} below) are under control thanks to assumption
    \ref{hyp:mobility_d}. For further details on 
    the concept of solution in case of degenerate mobility we refer the reader to \cite{ell-gar, scarpa21}.
\begin{defin}
    \label{def:sol-d}
    Let $\kappa\in\{0,1\}$ and let Assumptions \ref{hyp:potential}, \ref{hyp:mobility_d} and \ref{hyp:noise_d} hold. Let further
		\begin{equation}
		    \label{init_d}
		\varphi_0 \in V_1, \qquad
		F(\varphi_0) \in L^1(\OO), \qquad
        M(\varphi_0) \in L^1(\OO).
		\end{equation}
    \begin{enumerate}[(i)] \itemsep0.5em
    \item A probabilistically-strong solution to the
		Cahn--Hilliard equation \eqref{eq:ch} 
        with degenerate mobility starting from the initial datum $\varphi_0$ is a pair of stochastic processes $(\varphi,\,\b j)$ such that
		\begin{align}
			\label{phi_d}
			&\varphi \in L^p_\cP(\Omega; C^0([0,T]; H))\cap
			\luca{L^p_w(\Omega;L^\infty(0,T; V_1))} \cap
			L^p_\cP(\Omega; L^2(0,T; V_2)),\\
			& |{\varphi}(\omega, t,x)| \leq 1
			\text{ for a.a.~}(\omega,t,x) 
            \in \Omega \times (0,T)\times \OO, \\
			\label{j_d}
			&\b j
			\in L^{p}_\cP(\Omega; L^2(0,T; \b H)),
			\\
			\label{initial_d}
			&\varphi(0) = \varphi_0,
		\end{align}
		for every $p\in[2,+\infty)$, where
		\begin{multline} \label{eq:ch-str-d}
			( \varphi(t),\psi)_H +
			\int_0^t\left[\int_{\OO}
			\b j(s)\cdot \nabla \psi\,\d x
            +\kappa\ip{\L(s,\varphi(s))}{\psi}_{V_1}
            \right]\,\d s\\
    		  = ( \varphi(0),\psi)_{H} +
			\left(\int_0^t \nabla\varphi(s)\cdot\b\Sigma(s)\,\d  W(s), \psi\right)_{H}
        \end{multline}
        for every $\psi\in V_1$, $t \in [0,T]$, $\P$-almost surely, and 
        \begin{align} 
        \label{eq:ch2-str-d}
            \int_{\OO}
			\b j\cdot \b\zeta\,\d {x}
            &=\int_{\OO}
            m(\varphi)\Delta\varphi\div\b\zeta \: \d x
            +\int_{\OO}\left[
            m'(\varphi)\Delta\varphi
            +m(\varphi)F''(\varphi)\right]
            \nabla\varphi\cdot\b\zeta \: \d x
		\end{align}
        for every $\b\zeta\in \b{V}_1$, almost everywhere in $(0,T)$, $\P$-almost surely.
		\item A probabilistically-weak solution to the
		Cahn--Hilliard equation \eqref{eq:ch} 
        with degenerate mobility starting from the initial datum $\varphi_0$ is a a family $(\widehat\Omega,\, \widehat\cF,\, (\widehat\cF_t)_{t\in[0,T]},\,
    \widehat\P,\,
	\widehat W,\, \widehat{\b\Sigma},\, \widehat\varphi,\, \widehat{\b j})$,
	where $(\widehat\Omega,\, \widehat\cF,\,(\widehat\cF_t)_{t\in[0,T]},\, \widehat\P)$
	is a filtered probability space satisfying the usual conditions, $\widehat W$
	is a cylindrical Wiener processes on $U$, 
    the processes $\widehat\varphi$, $\widehat{\b j}$ and $\widehat{\b\Sigma}$ satisfy       
		\begin{align}
			\label{phi_hat_d}
			&\widehat\varphi \in L^p_\cP(\widehat\Omega; C^0([0,T]; H))\cap
			\luca{L^p_w(\widehat\Omega;L^\infty(0,T; V_1))} \cap
			L^p_\cP(\widehat\Omega; L^2(0,T; V_2)),\\
			& |\widehat{\varphi}(\omega, t,x)| \leq 1
			\text{ for a.a.~}(\omega,t,x) 
            \in \widehat\Omega \times (0,T)\times \OO, \\
			\label{j_d_hat}
			&\widehat{\b j}
			\in L^{p}_\cP(\widehat\Omega; L^2(0,T; \b H)),\\
			\label{initial_hat_d}
			&\widehat\varphi(0) = \varphi_0,\\
            \label{sigma_hat_d}
        &\widehat{\b\Sigma} \in L^q_\cP(\widehat\Omega; L^2(0,T; \cL^2(U,\W))), \qquad
        \widehat{\b\Sigma}\laweq \b\Sigma,
		\end{align}
		for every $p\in[2,+\infty)$, where
		\begin{multline} \label{eq:ch-wea-d}
			( \widehat\varphi(t),\psi)_H +
			\int_0^t\left[\int_{\OO}
			\widehat{\b j}(s)\cdot \nabla \psi\, \d x
            +\kappa\ip{\hL(s,\widehat\varphi(s))}{\psi}_{V_1}
            \right]\,\d s\\
			= ( \widehat\varphi(0),\psi)_{H} +
			\left(\int_0^t \nabla\widehat\varphi(s)\cdot
            \widehat{\b\Sigma}(s)\,\d  \widehat W(s), \psi\right)_{H}
		\end{multline}
		for every $\psi\in V_1$, $t \in [0,T]$, $\widehat\P$-almost surely,
        and 
        \begin{align} 
        \label{eq:ch2-wea-d}
            \int_{\OO}
			\widehat{\b j}\cdot \b\zeta\,\d {x}
            &=\int_{\OO}
            m(\widehat{\varphi})\Delta\widehat{\varphi}\div\b\zeta \: \d x
            +\int_{\OO}\left[
            m'(\widehat{\varphi})\Delta\widehat{\varphi}
            +m(\widehat{\varphi})F''(\widehat{\varphi})\right]
            \nabla\widehat{\varphi}\cdot\b\zeta \: \d x
		\end{align}
        for every $\b\zeta\in \b{V}_1$, almost everywhere in $(0,T)$, $\widehat{\P}$-almost surely.
        \end{enumerate}
	\end{defin} \noindent
    Then, the well-posedness result in the case of degenerate mobility reads as follows.
    \begin{thm}
        \label{th:d}
        Let $\kappa\in\{0,1\}$ and let Assumptions \ref{hyp:potential}, \ref{hyp:mobility_d} and \ref{hyp:noise_d} hold. Then, for every $\varphi_0$ satisfying \eqref{init_d},
        the Cahn--Hilliard equation \eqref{eq:ch} 
        with degenerate mobility 
        admits a probabilistically-weak solution
        starting from the initial datum $\varphi_0$,
        in the sense of Definition~\ref{def:sol-d},
        that also satisfies, for every $p\in[2,+\infty)$,
        \[
        \luca{F(\widehat\varphi)
        \in L^p_w(\widehat\Omega;L^\infty(0,T; L^1(\OO))),
        \qquad
        M(\widehat\varphi)\in \luca{L^p_w(\Omega;L^\infty(0,T; V_1))}.}
        \]
        In particular, if either 
        \[\lim_{|r|\to1^-}F(r)=+\infty \text{   or   }
        \lim_{|r|\to1^-}M(r)=+\infty,\] 
        then it also holds that
        \[
            |\widehat{\varphi}(\omega, t,x)| < 1
			\text{ for a.a.~}(\omega,t,x) 
            \in \widehat\Omega \times (0,T)\times \OO.
        \]
    \end{thm}
    \begin{remark}
        Let us point out that in the classical case of logarithmic Flory-Huggins potential \eqref{F_log}
        and degenerate mobility \eqref{m_deg},
        the last condition is not satisfied and one can only infer that $|\varphi|\leq1$ almost everywhere.
        Nonetheless, it is worthwhile to note that there are other choices of relevant potentials and degenerate mobilities that ensure instead such condition. For example, one can take either Lennard-Jones-type potentials in the form 
        \[
        \Psi_{\beta}(x):=\frac1{(1-x^2)^\beta}, \quad x\in(-1,1),
        \]
        for some $\beta>0$,
        or more degenerate mobilities in the form 
        \[
        m_{\alpha}(x):=(1-x^2)^\alpha, \quad 
        x\in[-1,1],
        \]
        for some $\alpha>1$. It is immediate to check also that assumption \ref{hyp:mobility_d} is satisfied by choosing $\alpha=\beta-2$.
    \end{remark}
    
\section{Proof of Theorem \ref{th:nd}}
\label{sec:nd}
This section is devoted to showing the well-posedness of problem \eqref{eq:ch} in the case of non-degenerate mobilities, i.e., to proving Theorem~\ref{th:nd}.
The proof of existence is performed in several steps throughout Subsections~\ref{ssec:approx}--\ref{ssec:limit_eps}, by employing a three-level approximation procedure and stochastic compactness arguments.
The uniqueness and continuous dependence part of the the theorem is eventually presented in Subsection~\ref{ssec:uniq}.

\subsection{The approximated problem}
\label{ssec:approx}
In this Subsection, we present the three-level approximation scheme to set up a compactness argument.

\subsubsection{Regularization of the transport field}
First, we introduce a regularization of $\b\Sigma$. More precisely, our aim is to introduce a family of stochastic processes $\{\b \Sigma^\eps\}_{\eps \in (0,1)}$ in such a way that the following properties hold:
\begin{enumerate}[(i)] \itemsep0.3em
    \item the process $\b \Sigma^\eps \in L^\infty_\cP(\Omega;L^\infty(0,T; \cL^2(U,\bH^2(\OO) \cap \W)))$ for all $\eps \in (0,1)$,
    \item the convergence $\b \Sigma ^\eps \to \b \Sigma$ holds in $L^q_\cP(\Omega;L^\infty(0,T;\cL^2(U;H)) \cap L^4(0,T;\cL^2(U;\W)))$ as $\eps \to 0^+$;
    \item letting $\b \sigma^\eps_k := \b \Sigma^\eps u_k$ for all $k \in \mathbb N$, it holds $\b\sigma^\eps_k \in 
L^\infty_\cP(\Omega; L^\infty(0,T; \bH^2(\OO)\cap \W))$ and
    \[
  \norm{\b\sigma^\eps_k}_{\bH}\leq
    \norm{\b\sigma_k}_{\bH},
    \quad
\norm{\b\sigma^\eps_k}_{\W}\leq
    \norm{\b\sigma_k}_{\W},
    \quad
\norm{\b\sigma^\eps_k}_{\b{L}^\infty(\OO)}\leq
    \norm{\b\sigma_k}_{\b{L}^\infty(\OO)}
  \quad\text{a.e.~in } \Omega\times(0,T),
\]
for all $k\in\enne$ and $\eps\in(0,1)$, 
as well as, in the limit $\eps\to0^+$,
\[
\b\sigma^\eps_k\to\b\sigma_k
\quad\text{in } \W,
\quad\text{a.e.~in } \Omega\times(0,T)
\]
for all $k \in \mathbb N$.
\end{enumerate}
Let us point out that such an approximation scheme exists and can be constructed by combining a truncation argument in $\Omega \times (0,T)$ with
classical singular perturbation techniques and elliptic regularity.
%for example, for every $\eps\in(0,1)$ and $k\in\enne$,
%one can set $\b\sigma_k^\eps$ as the unique solution 
%to the elliptic problem
%\[
%  \b\sigma_k^\eps + \eps \b A\b\sigma^\eps_k = \b\sigma_k,
%\]
%where $\b A$ is the Stokes operator on $\bH_\sigma$. \andrea{If $\b I$ is the identity operator on $\bHs$, then} this corresponds to setting $\b\sigma_k^\eps:=(\b{I}+\eps \b A)^{-1}\b\sigma_k$,
%and it is immediate to check that the properties \andrea{listed above} are satisfied.
%With this notation, for every $\eps\in(0,1)$, \andrea{we can define accordingly} the operator \andrea{$\b\Sigma^\eps
%\in L^\infty_{\cP}(\Omega; L^\infty(0,T; \cL^2(U,\W)))$} as
%\[
%  \b\Sigma^\eps u_k:=\b\sigma^\eps_{k}, \quad k\in\enne,
%\]
Accordingly, we also define the regularized Stratonovich correction
operator $\L^\eps:\Omega\times[0,T]\times V_1\to V_1^*$ as
\[
\ip{\L^\eps(\omega,t,\psi_1)}{\psi_2}:=
\frac12\sum_{k\in\enne}\int_\OO
[\nabla\psi_1(x)\cdot\b{\sigma}^\eps_k(\omega,t,x)]
[\nabla\psi_2(x)\cdot\b{\sigma}^\eps_k(\omega,t,x)]\,\d x,
\]
for $(\omega,t)\in\Omega\times[0,T]$ and $\psi_1,\,\psi_2\in V_1$.

\subsubsection{Regularization of the singular potential}
The second approximation scheme concerns a regularization of the singular part $\Psi$. Under Assumption \ref{hyp:potential}, for every $\lambda\in(0,1)$, we denote by 
$\Psi_\lambda:\erre\to[0,+\infty)$ the Moreau--Yosida regularization of the 
convex function $\Psi:\erre\to[0,+\infty]$. Let us recall that $\Psi_\lambda\in C^{1,1}(\erre)$ and $\Psi_\lambda'$
is the Yosida approximation of $\Psi'$ as a maximal monotone graph 
in $\erre\times\erre$. Indeed, letting $J_\lambda:\erre\to(-1,1)$ denote
the resolvent of $\Psi'$, namely $J_\lambda:=(I+\lambda\Psi')^{-1}$, it holds $\Psi_\lambda'(r)=\Psi'(J_\lambda(r))$ for all $r\in\erre$. 
Accordingly, we define the approximated potential
$F_\lambda:\erre\to[0,+\infty)$ as
$F_\lambda(r):=\Psi_\lambda(r) + R(r)$, $r\in\erre$. 
For classical results of monotone analysis we refer e.g.~to \cite{barbu-monot}.

\subsubsection{A discretization scheme}
Finally, we set up a Galerkin discretization scheme, exploiting the spectral properties of the negative Neumann--Laplace operator. Indeed, let $\seq{e}{n}$ denote a complete orthonormal system of $H$ made of eigenvectors of the negative Laplacian with homogeneous 
Neumann boundary conditions, and let $\seq{\lambda}{n}$ be the sequence of the respective eigenvalues.
For every $n\in\enne$, let 
$H_n:=\operatorname{span}\{e_1,\ldots,e_n\}\subset V_2$,
endowed with the $\norm{\,\cdot\,}_H$-norm, and let $P_n\in\cL(H,H_n)$
be the orthogonal projection on $H_n$. 
Analogously, let $\b{H}_n:=H_n^d$ the vector-valued Galerkin space
and let $\b{P}_n\in\cL(\b H, \b{H}_n)$ be the corresponding orthogonal projection. For every $\eps\in(0,1)$ and 
$n\in\enne$, we define the operator 
$\b\Sigma^\eps_n:=\b{P}_n\circ\b\Sigma^\eps
\in L^\infty_{\cP}(\Omega; L^\infty(0,T; \cL^2(U,\bH_n)))$ as
\[
  \b\Sigma_n^\eps u_k:=\b{P}_n\b\sigma^\eps_{k} \qquad \forall \: k\in\enne.
\]
On the same line, for every
$\eps\in(0,1)$ and $n\in\enne$,
we define the 
finite-dimensional Stratonovich correction operator 
$\Ln:\Omega\times[0,T]\times H_n\to H_n$ 
as 
\[
(\Ln(\omega,t,\psi_1),\,\psi_2)_{H_n}:=
\frac12\sum_{k=1}^n\int_\OO
[\b{\sigma}^\eps_k(\omega,t,x)\cdot\nabla \psi_1(x)]
[\b{\sigma}^\eps_k(\omega,t,x)\cdot\nabla\psi_2(x)]\:\d x,
\]
for $(\omega,t)\in\Omega\times[0,T]$ and $\psi_1,\,\psi_2\in H_n$. 
Note that $\Ln$ can be written also in strong 
formulation as 
\[
\Ln(\omega,t,\psi)=
-\frac12\sum_{k=1}^n [\b\sigma^\eps_k(\omega,t,x)\cdot \nabla] 
[\b{\sigma}^\eps_k(\omega,t,x)\cdot\nabla \psi(x)]
\qquad\forall\,
(\omega,t,\psi)\in\Omega\times[0,T]\times H_n.
\]
Finally, let $\rho \in C^\infty_0(\mathbb R)$ be a fixed standard mollifier, i.e., so that
\[
\rho(r) \geq 0 \quad \forall\: r \in \mathbb R, \qquad \int_\mathbb R \rho(r) \: \d r = 1.
\]
Defining, as customary, the sequence of mollifiers $\rho_n:\mathbb R\to \mathbb R$ as
\[
\rho_n(r) = n\rho(nr)
\]
for all $r \in \mathbb R$, we can consider then the approximated mobility $m_n:\mathbb R\to \mathbb R$ defined by
\[
m_n(r) = (\rho_n*m)(r)
\]
for all $r \in \mathbb R$. It is standard matter to show that $m_n$ is Lipschitz-continuous (albeit not uniformly in $n$), and that $m_n \to m$ in $L^p(\mathbb R)$ for all finite $p \geq 1$ and even uniformly on compact sets of $\mathbb R$. Moreover, observe that by positivity and unitary mass of $\rho_n$ it also holds that
\[
0 < c_m \leq m_n(r) \leq C_m
\]
for all $r \in \mathbb R$, i.e., the approximated mobility enjoys the bound prescribed by Assumption \ref{hyp:mobility_nd}. Accordingly, for all $n \in \mathbb N$, we define the second-order primitive
\[
M_n: \mathbb R\to \mathbb R \qquad M_n(r) = \int_0^r\int_0^s \dfrac{1}{m_n(t)} \: \d t \ \d s
\]
for all $r \in \mathbb R$.

\subsubsection{The approximated problem} 
In light of the previous schemes, we are in a position to consider the following approximated problem, 
depending on the three parameters $\eps$, $\lambda$, and $n$. In particular, the velocity fields are smoothed out 
through the $\eps$-regularization, the nonlinearity $F'$ is replaced by its $\lambda$-Yosida approximation, and the equation is projected on the finite-dimensional space $H_n$. The resulting problem reads
\begin{numcases}{}
  \label{eq1_app}
  \d\varphi^\eps_{\lambda,n} 
  - \div P_n\left[m_n(\varphi^\eps_{\lambda,n})
  \nabla\mu^\eps_{\lambda,n}\right]\,\d t + 
  \kappa\Ln\varphi^\eps_{\lambda,n}\,\d t
  =\nabla \varphi^\eps_{\lambda,n}\cdot \b\Sigma^\eps_n\,\d W
  \quad&\text{in }$(0,T)\times\OO$,\\
  \label{eq2_app}
  \mu^\eps_{\lambda,n}=-\Delta\varphi^\eps_{\lambda,n} + P_nF_{\lambda}'(\varphi^\eps_{\lambda,n})
  \quad&\text{in }$(0,T)\times\OO,$\\
  \label{eq3_app}
  \partial_{\b n}\varphi^\eps_{\lambda,n} = 
  m(\varphi^\eps_{\lambda,n})\partial_{\b n}
  \mu^\eps_{\lambda,n} = 0
  \quad&\text{in }$(0,T)\times\partial\OO$,\\
  \label{eq4_app}
  \varphi^\eps_{\lambda,n}(0)=P_n\varphi_0
  \quad&\text{in }$\OO$.
\end{numcases}
For every $\eps,\,\lambda\in(0,1)$ and $n\in\enne$, 
we look for a solution to 
\eqref{eq1_app}-\eqref{eq4_app} in the form
\[
  \varphi^\eps_{\lambda,n}=
  \sum_{j=1}^n a_j^{\eps,\lambda,n} e_j, \qquad
  \mu^\eps_{\lambda,n}=\sum_{j=1}^n b_j^{\eps,\lambda,n} e_j,
\]
for some vector-valued processes 
\[
\b a^{\eps}_{\lambda,n}:=(a_{1}^{\eps,\lambda,n},\ldots, a_n^{\eps,\lambda,n}):\Omega\times[0,T]\to\erre^{n}, \qquad
\b b^{\eps}_{\lambda,n}:=(b_\andrea{1}^{\eps,\lambda,n},\ldots,b_n^{\eps,\lambda,n}):\Omega\times[0,T]\to\erre^n.
\]
By exploiting the Lipschitz-continuity of 
$F'_\lambda$ and $m_n$,
it is a standard matter to check that the approximated system \eqref{eq1_app}--\eqref{eq4_app} comes down to a system of SDEs for $\b a^{\eps}_{\lambda,n}$ and 
$\b b^{\eps}_{\lambda,n}$ with Lipschitz-continuous coefficients. Therefore, we infer that \eqref{eq1_app}--\eqref{eq4_app} admits a unique
probabilistically and analitically strong solution 
\[
  \varphi^\eps_{\lambda,n},\,\mu^\eps_{\lambda,n} \in L^\ell_\cP(\Omega; C^0([0,T]; H_n))
  \qquad\forall\,\ell\in\mathopen[2,+\infty\mathclose).
\]

\subsection{Uniform estimates with respect to $\boldsymbol{n}$}
\label{ssec:est_n}
Throughout this subsection, we present in due detail the uniform estimates needed to retrieve a solution to problem \eqref{eq:ch} in the case of nondegenerate mobility. Here and in the following, the symbol $C$ is reserved for a positive constant depending on the structural parameters of the problem. As usual, its value may change within the same argument without relabeling and relevant dependencies will be highlighted if necessary. Let $\lambda \in (0,1)$ and $\eps \in (0,1)$ be fixed.

\subsubsection{The energy estimate} 
Leveraging the gradient flow structure of the deterministic part of the system, let us define the regularized free energy functional
\[
\mathcal E_\lambda: \operatorname{dom} \mathcal E_\lambda \subset V_1 \to \mathbb R \qquad \mathcal E_\lambda(v) = \int_\OO \dfrac 12 |\nabla v|^2 + F_\lambda(v) \: \d x.
\]
As customary in the theory of diffuse interface models, the crucial estimate is given by an energy inequality for the system. Applying the It\^o formula for the regularized free energy functional evaluated at $\varphi^\eps_{\lambda,n}(t)$ yields 
\begin{multline}
\label{ito:en_app}
\frac12\norm{\nabla\varphi^\eps_{\lambda,n}(t)}_{\bH}^2+\norm{F_\lambda(\varphi^\eps_{\lambda,n}(t))}_{L^1(\OO)}
    +\int_0^t\int_\OO
    m_n(\varphi^\eps_{\lambda,n}(s))
    |\nabla\mu^\eps_{\lambda,n}(s)|^2\,\d x\,\d s \\+\frac\kappa2\int_0^t\sum_{k=0}^n
    \int_\OO
    [\nabla \varphi^\eps_{\lambda,n}(s)\cdot \b{\sigma}^\eps_k(s)]
[\nabla\mu^\eps_{\lambda,n}(s)\cdot \b{\sigma}^\eps_k(s)]\,\d x\,\d s\\
=\frac12\norm{\nabla P_n\varphi_0}_{\bH}^2
    +\norm{F_\lambda(P_n\varphi_0)}_{L^1(\OO)}
    +\int_0^t\left(\mu^\eps_{\lambda,n}(s), \nabla \varphi^\eps_{\lambda,n}(s)
    \cdot\b\Sigma^\eps_n(s)\,\d W(s)\right)_{H}\\
    +\frac12\int_0^t\sum_{k=0}^n
    \int_\OO\left[
    |\nabla[\nabla \varphi^\eps_{\lambda,n}(s)\cdot\b\sigma^\eps_k(s)]|^2+F_\lambda''(\varphi^\eps_{\lambda,n}(s))
    |\nabla \varphi^\eps_{\lambda,n}(s)\cdot\b\sigma^\eps_k(s)|^2
    \right]\,\d x\,\d s
\end{multline}
for every $t\in[0,T]$, $\P$-almost surely. 
On the left-hand side, by the properties of the approximated mobility one has that 
\[
\int_0^t\int_\OO
    m_n(\varphi^\eps_{\lambda,n}(s))
    |\nabla\mu^\eps_{\lambda,n}(s)|^2\,\d x\,\d s
    \geq c_m\norm{\nabla\mu^\eps_{\lambda,n}}^2_{L^2(0,t;\bH)},
\]
while the H\"older and Young inequalities and Assumption \ref{hyp:noise} imply that 
\begin{align*}
    &\frac\kappa2\int_0^t\sum_{k=0}^n
    \int_\OO
    |\nabla \varphi^\eps_{\lambda,n}(s)\cdot \b{\sigma}^\eps_k(s)|
|\nabla\mu^\eps_{\lambda,n}(s)\cdot \b{\sigma}^\eps_k(s)|\,\d x\, \d s\\
&\hspace{2cm}\leq\frac\kappa2\int_0^t
\sum_{k\in\enne}\norm{\b\sigma_k^\eps(s)}_{\b{L}^\infty(\OO)}^2
\norm{\nabla\varphi^\eps_{\lambda,n}(s)}_\bH
\norm{\nabla\mu^\eps_{\lambda, n}}_\bH\,\d s\\
&\hspace{2cm}\leq\frac{c_m}4\norm{\nabla\mu^\eps_{\lambda,n}}_{L^2(0,t;\bH)}^2
+\frac{\kappa^2C_\infty^4}{4c_m}\int_0^t\left(
\sum_{k\in\enne}\norm{\b\sigma_k^\eps(s)}_{\W}^2
\right)^2
\norm{\nabla\varphi^\eps_{\lambda,n}(s)}_\bH^2\,\d s.
\end{align*}
On the right-hand side, by the contraction properties of $P_n$ it holds that 
\[
  \frac12\norm{\nabla P_n\varphi_0}_{\bH}^2\leq
  \frac12\norm{\nabla \varphi_0}_{\bH}^2.
\]
Furthermore, as far as the stochastic integral is concerned, we have 
by Assumption \ref{hyp:noise} that 
\[
\int_0^t\left(\mu^\eps_{\lambda,n}(s), 
\nabla \varphi^\eps_{\lambda,n}(s))
    \cdot\b\Sigma^\eps_n(s)\,\d W(s)\right)_{H}
    =-\int_0^t\left(\nabla\mu^\eps_{\lambda,n}(s), \varphi^\eps_{\lambda,n}(s)
    \cdot\b\Sigma^\eps_n(s)\,\d W(s)\right)_{\bH},
\]
so that 
the Burkholder--Davis--Gundy, H\"older, and Young
inequalities
imply that, for every $\ell\geq2$ and $\delta>0$,
\begin{align*}
    &\E\left| \supp
    \int_0^\tau\left(\mu^\eps_{\lambda,n}(s), \nabla \varphi^\eps_{\lambda,n}(s)
    \cdot\b\Sigma^\eps_n(s)\,\d W(s)\right)_{H}
    \right|^{\frac\ell2}\\
    &\hspace{2cm}\leq C_\ell
    \E\left(\int_0^t
    \norm{\nabla\mu^\eps_{\lambda,n}(s)}_{\bH}^2
    \norm{\varphi^\eps_{\lambda,n}(s)\b\Sigma^\eps_n(s)}_{\cL^2(U,\bH)}^2
    \,\d s\right)^{\frac\ell4}\\
    &\hspace{2cm}\leq C_{\ell}
    \E\left(\int_0^t
    \norm{\nabla\mu_{\lambda,n}(s)}_{\bH}^2
    \norm{\varphi^\eps_{\lambda,n}(s)}_H^2
    \sum_{k=0}^n
    \norm{\b\sigma^\eps_k(s)}_{\b{L}^\infty(\OO)}^2
    \,\d s\right)^{\frac\ell4}\\
    &\hspace{2cm}\leq C_\ell\E\left[
    \norm{\nabla\mu_{\lambda,n}}_{L^2(0,t;\bH)}^{\frac\ell2}
    \norm{\varphi^\eps_{\lambda,n}}_{L^\infty(0,t;H)}^{\frac\ell2}
    \left(\sup_{s\in[0,T]}\sum_{k\in\enne}
\norm{\b\sigma^\eps_k(s)}^{2}_{\b L^\infty(\OO)}\right)^{\frac\ell4}
\right]\\
    &\hspace{2cm}\leq \delta\E\norm{\nabla\mu_{\lambda,n}}^\ell_{L^2(0,t;\bH)}
    +\frac{C_{\eps,\ell}}\delta
    \E\norm{\varphi^\eps_{\lambda,n}}^\ell_{L^\infty(0,t;H)},
\end{align*}
where $C_\ell,\, C_{\eps,\ell}>0$ are constants independent of both $\lambda$ and $n$.
Eventually, recalling the properties of the noise coefficient
collected in Remark~\ref{rmk:prop_noise}, one has that 
\begin{multline*}
    \frac12\int_0^t\sum_{k=0}^n
    \int_\OO\left[
    |\nabla[\nabla \varphi^\eps_{\lambda,n}(s)\cdot\b\sigma^\eps_k(s)]|^2
    +F_\lambda''(\varphi^\eps_{\lambda,n}(s))
    |\nabla \varphi^\eps_{\lambda,n}(s)\cdot\b\sigma^\eps_k(s)|^2
    \right]\,\d x\,\d s\\
    \leq C\int_0^t\left(\sum_{k\in\enne}
    \norm{\b\sigma^\eps_k(s)}_{\W}^2\right)
    \left[\norm{\varphi^\eps_{\lambda,n}(s)}_{V_2}^2
    +\norm{|F''_\lambda(\varphi^\eps_{\lambda,n}(s))|^{\frac12}
    \nabla\varphi^\eps_{\lambda,n}(s)}_{\bH}^2
    \right]\,\d s.
\end{multline*}
Taking $\frac \ell2$-powers, essential suprema in time and $\P$-expectations in \eqref{ito:en_app} and collecting the above estimates, then choosing $\delta$ sufficiently small
and suitably rearranging the terms, we infer that, for every $\ell\geq2$ and for every $t\in[0,T]$,
\begin{multline}
    \label{ito:en_app2}
    \E\sup_{s\in[0,t]}
    \norm{\nabla\varphi^\eps_{\lambda,n}(s)}_\bH^\ell
    +\E\sup_{s\in[0,t]}
    \norm{F_\lambda(\varphi^\eps_{\lambda,n}(s))}_{L^1(\OO)}^{\frac\ell2}
    +\E\norm{\nabla\mu^\eps_{\lambda,n}}_{L^2(0,t; \bH)}^{\ell}\\
    \leq C_{\ell}\left(\norm{\nabla\varphi_0}_\bH^\ell
    +\norm{F_\lambda(P_n\varphi_0)}_{L^1(\OO)}^\ell\right)
    +C_{\eps,\ell}
    \E\norm{\varphi^\eps_{\lambda,n}}^\ell_{L^\infty(0,t;H)}
    +C_{\eps,\ell}\int_0^t
    \E\norm{\nabla\varphi^\eps_{\lambda,n}(s)}_\bH^\ell\,\d s\\
    +
    C_{\eps,\ell}
    \E\norm{\varphi^\eps_{\lambda,n}}^\ell_{L^2(0,t;V_2)}
    +C_{\eps,\ell}
    \E\norm{|F''_\lambda(\varphi^\eps_{\lambda,n})|^{\frac12}
    \nabla\varphi^\eps_{\lambda,n}}^\ell_{L^2(0,t;\bH)}.
\end{multline}
Note that the implicit constants appearing on the right-hand side
have been updated, as usual, and are independent of both $n$ and 
$\lambda$. In order to refine the previous estimate, we test \eqref{eq2_app} by 
$-\Delta\varphi^\eps_{\lambda,n}$ and integrate the result in time, deducing that, for every $K,\,\delta>0$,
\begin{align*}
&K\int_0^t
\left[
\norm{\Delta\varphi^\eps_{\lambda,n}(s)}_H^2
+\int_\OO \Psi_\lambda''(\varphi^\eps_{\lambda,n}(s))
|\nabla\varphi^\eps_{\lambda,n}(s)|^2\,\d x
\right]\,\d s\\
&\hspace{2cm}=K\int_0^t\int_\OO\nabla\mu^\eps_{\lambda,n}(s)\cdot
\nabla\varphi^\eps_{\lambda,n}(s)\,\d x\,\d s
-K\int_0^t\int_\OO R''(\varphi^\eps_{\lambda,n}(s))|
\nabla\varphi^\eps_{\lambda,n}(s)|^2\,\d x\,\d s\\
&\hspace{2cm}\leq\delta\norm{\nabla\mu^\eps_{\lambda,n}}^2_{L^2(0,t; \bH)}
+\left(\frac{K^2}{4\delta}+C_R K\right)
\norm{\nabla\varphi^\eps_{\lambda,n}}^2_{L^2(0,t;\bH)}.
\end{align*}
Hence, we can choose $K$ large enough (e.g., $K=2C_{\eps,\ell}$)
and $\delta$ small enough in order 
to plug the last estimate 
into \eqref{ito:en_app2}. Moreover, the pathwise mass-conservation properties of the system entails that 
\[|\overline{\varphi^\eps_{\lambda,n}}|=
|\overline{P_n\varphi_0}|\leq\norm{\varphi_0}_{L^1(\OO)}\]
for all times, $\P$-almost surely, and by using again that $|R''|\leq C_R$,
we deduce that 
\begin{multline}
    \label{ito:en_app3} \E\sup_{s\in[0,t]}
    \norm{\varphi^\eps_{\lambda,n}(s)}_{V_1}^\ell
    +\E\sup_{s\in[0,t]}
    \norm{F_\lambda(\varphi^\eps_{\lambda,n}(s))}_{L^1(\OO)}^{\frac\ell2}\\
    +\E\norm{\nabla\mu^\eps_{\lambda,n}}_{L^2(0,t; \bH)}^{\ell}
    +\E\norm{\varphi^\eps_{\lambda,n}}_{L^2(0,t; V_2)}^{\ell}
    +\E\norm{|\Psi''_\lambda(\varphi^\eps_{\lambda,n})|^{\frac12}
    \nabla\varphi^\eps_{\lambda,n}}^\ell_{L^2(0,t;\bH)}\\
    \leq C_{\eps,\ell}\left(\norm{\nabla\varphi_0}_\bH^\ell
    +\norm{F_\lambda(P_n\varphi_0)}_{L^1(\OO)}^\ell\right)
    +C_{\eps,\ell}
    \E\norm{\varphi^\eps_{\lambda,n}}^\ell_{L^\infty(0,t;H)}
    +C_{\eps,\ell}\int_0^t
    \E\norm{\nabla\varphi^\eps_{\lambda,n}(s)}_\bH^\ell\,\d s.
\end{multline}

\subsubsection{The mobility estimate} 
In order to close the energy estimate \eqref{ito:en_app3}, a lower-order estimate is necessary. Eventually, noting that under Assumption \ref{hyp:mobility_nd}
one has that $M_{\andrea{n}}\in C^2(\erre)$, by the It\^o formula we get 
\begin{multline}\label{ito:M_n}
\norm{M_{n}(\varphi^\eps_{\lambda,n}(t))}_{L^1(\OO)}
  +\int_0^t\int_\OO
  \nabla\mu^\eps_{\lambda,n}(s)\cdot
  \nabla\varphi^\eps_{\lambda,n}(s)\,\d x\,\d s\\
  =\norm{M_{n}(P_n\varphi_0)}_{L^1(\OO)}
  +\int_0^t\left(M'_{n}(\varphi^\eps_{\lambda,n}(s)),
  \nabla\varphi^\eps_{\lambda,n}(s)\cdot\b\Sigma^\eps_k(s) \:
  \d W(s)\right)_H\\
  +\frac{1-\kappa}2\int_0^t\sum_{k=0}^n\int_\OO
  M''_{n}(\varphi^\eps_{\lambda,n}(s))
  |\nabla\varphi^\eps_{\lambda,n}(s)
  \cdot\b\sigma^\eps_k(s)|^2\,\d x\,\d s.
\end{multline}
Now, on the left-hand side, thanks to Assumption \ref{hyp:mobility_nd}
it holds that 
\[
\norm{M_{n}(\varphi^\eps_{\lambda,n}(t))}_{L^1(\OO)}
\geq\frac1{2C_m}\norm{\varphi^\eps_{\lambda,n}(t)}_H^2,
\]
while the identity \eqref{eq2_app} yields 
\[
  \int_0^t\int_\OO
  \nabla\mu^\eps_{\lambda,n}(s)\cdot
  \nabla\varphi^\eps_{\lambda,n}(s)\,\d x\,\d s
  \geq \int_0^t\int_\OO\left[
  |\Delta\varphi^\eps_{\lambda,n}(s)|^2
  +\Psi_\lambda''(\varphi^\eps_{\lambda,n}(s))
  |\nabla\varphi^\eps_{\lambda,n}(s)|^2
  -C_R|\varphi^\eps_{\lambda,n}(s)|^2
  \right]\,\d x\,\d s.
\]
On the right-hand side, Assumption \ref{hyp:mobility_nd} yields
\[
\norm{M_{n}(P_n\varphi_0)}_{L^1(\OO)}
\leq \frac1{2c_m}\norm{P_n\varphi_0}_H^2
\leq \frac1{2c_m}\norm{\varphi_0}_H^2.
\]
Furthermore,
note that the stochastic integral vanishes, since for every $k\in\enne$ and $s\in[0,t]$ one has that 
\[
  \left(M'_{n}(\varphi^\eps_{\lambda,n}(s)),
  \nabla\varphi^\eps_{\lambda,n}(s)\cdot\b\sigma^\eps_k(s)
  \right)_H
  =\left(\nabla M_{n}(\varphi^\eps_{\lambda,n}(s)), \b\sigma^\eps_k(s)\right)_{\bH}=0.
\]
Eventually, again by Assumption \ref{hyp:mobility_nd} we have
\begin{multline*}
    \frac{1-\kappa}2\int_0^t\sum_{k=0}^n\int_\OO
  M''_{n}(\varphi^\eps_{\lambda,n}(s))
  |\nabla\varphi^\eps_{\lambda,n}(s)
  \cdot\b\sigma^\eps_k(s)|^2\,\d x\,\d s\\
  \leq \frac{(1-\kappa)C_\infty^2}{2c_m}
  \int_0^t\left(\sum_{k\in\enne}
  \norm{\b\sigma^\eps_k(s)}_{\W}^2
  \right)\norm{\nabla\varphi^\eps_{\lambda,n}(s)}_\bH^2\,\d s.
\end{multline*}
Taking this information into account in \eqref{ito:M_n}, after raising it to the power $\frac \ell 2$, taking essential suprema in time and $\P$-expectations, we infer that, for every $\ell\geq2$,
\begin{multline}
    \label{ito:M_n2}
    \E\sup_{s\in[0,t]}\norm{\varphi^\eps_{\lambda,n}(s)}_H^\ell
    +\E\norm{\varphi^\eps_{\lambda,n}}_{L^2(0,t; V_2)}^\ell
    +\E\norm{|\Psi_\lambda''(\varphi^\eps_{\lambda,n})|^{\frac12}
    \nabla\varphi^\eps_{\lambda,n}}_{L^2(0,t;\bH)}^\ell\\
    \leq C\left(1+\norm{\varphi_0}_H^2\right)
    +C_{\eps,\ell}\int_0^t\E\norm{\varphi^\eps_{\lambda,n}(s)}_{V_1}^\ell \: \d s.
\end{multline}
\subsubsection{Closing the energy estimate}
Multiplying \eqref{ito:M_n2} by a sufficiently large
constant and summing it to \eqref{ito:en_app3} to 
absorb all the terms, we obtain that 
\begin{multline}
    \label{ito:en_app4}   \E\sup_{s\in[0,t]}
    \norm{\varphi^\eps_{\lambda,n}(s)}_{V_1}^\ell
    +\E\sup_{s\in[0,t]}
\norm{F_\lambda(\varphi^\eps_{\lambda,n}(s))}_{L^1(\OO)}^{\frac\ell2}\\
+\E\norm{\nabla\mu^\eps_{\lambda,n}}_{L^2(0,t; \bH)}^{\ell}
+\E\norm{\varphi^\eps_{\lambda,n}}_{L^2(0,t; V_2)}^{\ell}
+\E\norm{|\Psi''_\lambda(\varphi^\eps_{\lambda,n})|^{\frac12}
\nabla\varphi^\eps_{\lambda,n}}^\ell_{L^2(0,t;\bH)}\\
    \leq C_{\eps,\ell}\left(\norm{\nabla\varphi_0}_\bH^\ell
+\norm{F_\lambda(P_n\varphi_0)}_{L^1(\OO)}^\ell\right)
    +C_{\eps,\ell}\int_0^t
    \E\norm{\varphi^\eps_{\lambda,n}(s)}_{V_1}^\ell\,\d s,
\end{multline}
where the implicit constants on the right-hand side
are independent of both $\lambda$ and $n$. Thanks to the quadratic growth of $F_\lambda$ and again the contraction of $P_n$ it holds that 
\[
\norm{F_\lambda(P_n\varphi_0)}_{L^1(\OO)}
\leq C_\lambda\left(1+\norm{P_n\varphi_0}_H^2\right)
\leq C_\lambda\left(1+\norm{\varphi_0}_H^2\right),
\]
where $C_\lambda>0$ is a positive constant independent of $n$.
Hence, we can apply 
the Gronwall lemma and
deduce that for every $\ell\geq2$, there exists a constant 
$C_{\eps,\lambda,\ell}$ independent of $n$, such that 
\begin{align}
    \label{est1_n}
    \norm{\varphi^\eps_{\lambda,n}}_{L^\ell_\cP(\Omega;
    C^0([0,T]; V_1)\cap L^2(0,T; V_2))}
    &\leq C_{\eps,\lambda,\ell},\\
    \label{est2_n}
    \norm{\nabla\mu^\eps_{\lambda,n}}_{L^\ell_\cP(\Omega; L^2(0,T;\bH))}&\leq C_{\eps,\lambda,\ell},\\
    \label{est3_n}
    \norm{|\Psi_\lambda''(\varphi^\eps_{\lambda,n})|^{\frac12}
    \nabla\varphi^\eps_{\lambda,n}}_{L^\ell_\cP(\Omega;L^2(0,T;\bH))}
    &\leq C_{\eps,\lambda,\ell}.
\end{align}

\subsubsection{Further estimates}
\label{ssec:further_n}
As a straightforward consequence of the previous estimates, since $F_\lambda'$ is Lipschitz-continuous, by comparison in 
\eqref{eq2_app} we also get 
\begin{align}
    \label{est4_n}
    \norm{\mu^\eps_{\lambda,n}}_{L^\ell_\cP(\Omega; L^2(0,T;V_1))}&\leq C_{\eps,\lambda,\ell},\\
    \label{est5_n}
    \norm{F_\lambda'(\varphi^\eps_{\lambda,n})}_{
    L^\ell_\cP(\Omega; L^2(0,T;H))}
    &\leq C_{\eps,\lambda,\ell}.
\end{align}
For the sake of convenience, let
\[
G^{\eps}_n: \Omega \times[0,T] \times V_1 \to \cL^2(U,H)
\]
be defined by 
\[
G^{\eps}_n(\omega, t,\psi)[u_k] = \nabla \psi\cdot \b \Sigma^\eps_n(\omega, t)[u_k] = \nabla \psi \cdot \b P_n \b \sigma^\eps_k(\omega, t)
\]
for all $k \in \mathbb N$.
Then, we have
\begin{equation*}
    \begin{split}
        \|G^{\eps}_n(\varphi^\eps_{\lambda,n}(t))\|^2_{\cL^2(U,H)} & \leq \sum_{k\in\mathbb N} \|\nabla \varphi^\eps_{\lambda,n}(t) \cdot \b P_n\b \sigma^\eps_{k}(t)\|^2_H \\
        & \leq C\sum_{k\in\mathbb N} \|\varphi^\eps_{\lambda,n}(t)\|^2_{V_1}\|\b \sigma^\eps_{k}(t)\|^2_{\W},
    \end{split}
\end{equation*}
and, in turn,
\[
\begin{split}
    \E \sup_{t \in [0,T]}  \|G^{\eps}_n(\varphi^\eps_{\lambda,n}(t))\|^\ell_{\cL^2(U,H)}&  \leq C\E\left[ \sup_{t \in [0,T]} \|\varphi^\eps_{\lambda,n}(t)\|^\ell_{V_1} \left( \sum_{k \in \mathbb N} \|\b \sigma^\eps_{k}(t)\|^2_{\W}\right)^\frac \ell 2 \right] \\
    & \leq C_\eps \E\left[ \sup_{t \in [0,T]} \|\varphi^\eps_{\lambda,n}(t)\|^\ell_{V_1} \right].
\end{split}
\]
Before moving on, observe that by testing the equation \eqref{eq2_app} by $-\Delta \varphi_{\lambda,n}^\eps$ we obtain
\[
\|\Delta \varphi_{\lambda,n}^\eps(t)\|^2_H \leq C\left( \|\varphi_{\lambda,n}^\eps(t)\|^2_{V_1} + \|\nabla \mu_{\lambda,n}^\eps(t)\|_{\b H}\|\varphi_{\lambda,n}^\eps(t)\|_{V_1}\right)
\]
where $C >0$ is independent of $n, \, \lambda$ and $\eps$. By squaring, integrating in time, and taking $\frac \ell 4$-powers and $\P$-expectations we have
\[
\begin{split}
    & \E \left| \int_0^t\|\Delta \varphi_{\lambda,n}^\eps(s)\|^4_H \: \d s \right|^\frac \ell 4 \\
    & \hspace{2cm}\leq C\left[ \E\left|\int_0^t\|\varphi_{\lambda,n}^\eps(s)\|^4_{V_1} \: \d s\right|^\frac \ell 4+ \E\left|\int_0^t\|\nabla \mu_{\lambda,n}^\eps(t)\|^2_{\b H}\|\varphi_{\lambda,n}^\eps(t)\|^2_{V_1}\: \d s\right|^\frac \ell 4\right] \\
    & \hspace{2cm}\leq C\left[ \E\left|\int_0^t\|\varphi_{\lambda,n}^\eps(s)\|^4_{V_1} \: \d s\right|^\frac \ell 4+ \E\sup_{t \in [0,T]} \|\varphi_{\lambda,n}^\eps(t)\|^\frac \ell 2_{V_1}\left|\int_0^t\|\nabla \mu_{\lambda,n}^\eps(t)\|^2_{\b H}\: \d s\right|^\frac \ell 4\right] \\
    & \hspace{2cm}\leq C\left[ \|\varphi_{\lambda,n}^\eps\|^\ell_{L^\ell_\cP(\Omega;L^4(0,T;V_1))} + \|\varphi_{\lambda,n}^\eps(t)\|^\ell_{L^\ell_\cP(\Omega;C^0([0,T];V_1))} + \|\nabla \mu_{\lambda,n}^\eps\|^\ell_{L^\ell_\cP(\Omega;L^2(0,T;\b H))}\right],
\end{split}
\]
implying in turn
\[
\|\varphi_{\lambda,n}^\eps\|_{L^\ell_\cP(\Omega;L^4(0,T;V_2))} \leq C_{\eps, \lambda,\ell},
\]
slightly improving the high-order time regularity of the previous estimate.
Next, we have
\begin{equation*}
    \begin{split}
        \|G^{\eps}_n(\varphi^\eps_{\lambda,n}(t))\|^2_{\cL^2(U,V_1)} & = \sum_{k\in\mathbb N} \|\nabla \varphi^\eps_{\lambda,n}(t) \cdot \b P_n\b \sigma^\eps_{k}(t)\|^2_{V_1} \\
        & \leq C\sum_{k\in\mathbb N} \|\varphi^\eps_{\lambda,n}(t)\|^2_{V_2}\|\b \sigma^\eps_{k}(t)\|^2_{\W},
    \end{split}
\end{equation*}
and, in turn,
\[
\E\left| \int_0^t \|G^{\eps}_n(\varphi^\eps_{\lambda,n}(s))\|^4_{\cL^2(U,V_1)} \: \d s\right|^\frac \ell 4 \leq C_\eps \E \left| \int_0^t \|\varphi^\eps_{\lambda,n}(s)\|^4_{V_2} \: \d s\right|^\frac \ell 4.
\]
The above estimates show that
\[
\|G^{\eps}_n(\varphi^\eps_{\lambda,n})\|_{L^\ell_\cP(\Omega;L^\infty(0,T;\cL^2(U,H)) \cap L^4(0,T;\cL^2(U,V_1)))} \leq C_{\eps,\lambda,\ell},
\]
and applying \cite[Lemma 2.1]{fland-gat}, we immediately get the bounds on the It\^{o} stochastic integrals
\[
\left\|\int_0^\cdot G^\eps_{\lambda, n}(\varphi_{\lambda,n}^\eps(s)) \: \d W_s\right\|_{L^\ell_\cP(\Omega; W^{\alpha, \ell}(0,T;H)) \cap W^{\beta, 4}(0,T;V_1))} \leq C_{\eps, \lambda, \ell}
\]
for all $\ell \geq 4$, all $\alpha \in (0,\frac 12)$ and all $\beta \in (0,\frac 12)$. Finally, we are left to handle the Stratonovich corrector, whose control is straightforward given the high regularity of the approximated transport field. Indeed, a simple application of the H\"{o}lder inequality yields
\[
\|\Ln(\varphi_{\lambda,n}^\eps)\|_{L^\ell_\cP(\Omega;C^0([0,T]; V_1^*))} \leq C_{\eps,\lambda,\ell}.
\]
By comparison, we deduce at last that
\[
\|\varphi_{\lambda,n}^\eps\|_{L^\ell(\Omega;W^{\alpha, \ell}(0,T;V_1^*))} \leq C_{\eps, \lambda, \ell}
\]
for all $\alpha \in (0,\frac 12)$ and $\ell \geq 2$.

\subsection{Passage to the limit as $\boldsymbol{n\to+\infty}$}
\label{ssec:limit_as_n}
The previously shown uniform estimates enable a stochastic compactness argument via the Prokhorov and Skorokhod theorems. 

\subsubsection{Tightness of the laws of Galerkin solutions}
In order to show the needed tightness properties, the procedure combines the previous uniform estimates with suitable compact embeddings of Bochner spaces through the Markov inequality. Collectively, we get the following result.
\begin{lem} \label{lem:tight}
    Let $W_n \equiv W$ be a constant sequence of Wiener processes. Then, the family of laws of 
    \[
    \{\varphi_{\lambda,n}^\eps,\, G^\eps_{\lambda,n}(\varphi_{\lambda,n}^\eps) \cdot W_n,\, W_n,\,\b \Sigma^\eps_{n}\}_{n \in\mathbb N}
    \]
    is tight in the space
    \[
    \Xi \times C^0([0,T]; H) \times C^0([0,T]; U_0) \times L^2(0,T;\cL^2(U,\b H)))
    \]
    where $\Xi = C^0([0,T];H) \cap L^2(0,T;V_1)$.
\end{lem}
\begin{proof}
    The tightness property of the sequence of laws of Wiener processes is trivial as such family of laws consists of a single probability measure on $C^0([0,T];U)$, which is a separable Banach space. For the remaining cases, let us recall that the embeddings (see also \cite[Corollary 5, p. 86]{simon})
    \begin{align*}
        L^\infty(0,T;V_1) \cap W^{\alpha, \ell}(0,T;V_1^*) & \embed C^0([0,T]; H) \\
        L^2(0,T;H^2) \cap W^{\alpha, \ell}(0,T;V_1^*) & \embed L^2(0,T;V_1)
    \end{align*}
    are compact whenever $\alpha \ell > 1$, i.e., given that $\alpha$ is an arbitrary value in $(0, \frac 12)$, whenever $\ell > 2$. By the same token, we have
    \begin{align*}
        W^{\beta, 4}(0,T;V_1) \embed C^0([0,T], H)
    \end{align*}
    compactly for any $\beta > \frac 14$. As for the transport fields, observe that by the contractive properties of $P_n$
    \[
    \|\b \Sigma^\eps_n\|_{\cL^2(U, \b V_1)}  \leq \|\b \Sigma^\eps\|_{\cL^2(U, \b V_1)}
    \]
    for all $n \in \mathbb N$, and the embedding $\b V_1 \embed \b H$ is compact. The claims all follow by the Markov inequality and the previously proven uniform estimates through a standard argument.
\end{proof} \noindent
Owing to the Prokhorov and Skorokhod theorems, there exists a probability space $(\tom,\, \tF,\, \tP)$, possibly depending on $\eps$ and $\lambda$, and a sequence of random variables $\{X_{\lambda,n}^\eps :(\tom, \tF) \to (\Omega, \cF)\}_{n \in \mathbb N}$ such that the law of $X_{\lambda,n}^\eps$ equals $\P$ for all $n \in \mathbb N$ and, up to subsequences that we do not relabel for readability purposes, all of the following convergences hold in the limit $n \to +\infty$:
\begin{align*}
    \widetilde{\varphi}^\eps_{\lambda, n} := \varphi_{\lambda,n}^\eps \circ X_{\lambda,n}^\eps \to \widetilde{\varphi}^\eps_\lambda & \qquad \text{in }C^0([0,T];H) \cap L^2(0,T;V_1),\, \tP\text{-a.s.;} \\
    \widetilde{I}^\eps_{\lambda, n} := (G^\eps_n(\varphi_{\lambda,n}^\eps) \cdot W_n) \circ X_{\lambda,n}^\eps  \to \widetilde{I}^\eps_\lambda & \qquad \text{in }C^0([0,T];H),\, \tP\text{-a.s.;} \\
    \widetilde{W}^\eps_{\lambda,n} := W_n \circ X_{\lambda,n}^\eps \to \widetilde W^\eps_\lambda & \qquad \text{in }C^0([0,T];U_0),\, \tP\text{-a.s.};\\
    \widetilde{\b \Sigma}^\eps_n := \b \Sigma^\eps_n \circ X_{\lambda,n}^\eps \to \widetilde{\b \Sigma}^\eps_\lambda & \qquad \text{in }L^2(0,T;\cL^2(U, \b H)),\, \tP\text{-a.s.};
\end{align*}
where the limit processes belong to the specified spaces. In turn, using the Vitali convergence theorem, reflexivity and the Banach--Alaoglu theorem, we have also
\begin{align*}
    \widetilde{\varphi}^\eps_{\lambda, n} := \varphi_{\lambda,n}^\eps \circ X_{\lambda,n}^\eps \to \widetilde{\varphi}^\eps_\lambda & \qquad \text{in }L^p(\tom; C^0([0,T];H) \cap L^2(0,T;V_1)),\,\quad \forall \, p < \ell; \\
    \widetilde{\varphi}^\eps_{\lambda, n} \overset{*}{\rightharpoonup} \widetilde{\varphi}^\eps_\lambda & \qquad \text{in }L^\ell_w(\tom; L^\infty([0,T];V_1)) \cap L^\ell(\tom; W^{\alpha, \ell}(0,T;V_1^*)); \\
    \widetilde{\varphi}^\eps_{\lambda, n} \rightharpoonup \widetilde{\varphi}^\eps_\lambda & \qquad \text{in }L^\ell(\tom; L^4([0,T];V_2)); \\
    \widetilde{\mu}^\eps_{\lambda, n} := {\mu}^\eps_{\lambda, n} \circ X_{\lambda,n}^\eps \rightharpoonup\widetilde{\mu}^\eps_\lambda & \qquad \text{in }L^\ell(\tom; L^2([0,T];V_1)); \\
    \widetilde{I}^\eps_{\lambda, n} := (G^\eps_n(\varphi_{\lambda,n}^\eps) \cdot W_n) \circ X_{\lambda,n}^\eps  \to \widetilde{I}^\eps_\lambda & \qquad \text{in }L^p(\tom; C^0([0,T];H)),\, \quad \forall \, p < \ell; \\
    \widetilde{W}^\eps_{\lambda,n} := W_n \circ X_{\lambda,n}^\eps \to \widetilde{W}^\eps_\lambda & \qquad \text{in }L^p(\tom;C^0([0,T];U_0)),\,\quad \forall \, p < \ell; \\
    \widetilde{\b \Sigma}^\eps_n := \b \Sigma^\eps_n \circ X_{\lambda,n}^\eps \to \widetilde{\b \Sigma}^\eps_\lambda & \qquad \text{in }L^p(\tom; L^2([0,T];\cL^2(U, \b H)),\,\qquad \forall \, p < \ell.
\end{align*}
Owing to all the convergence above, let us stress that the limiting processes satisfy the following regularity properties:
\begin{align*}
    \widetilde{\varphi}^\eps_\lambda & \in L^\ell_\cP(\tom; C^0([0,T];H)) \cap L^\ell_w(\tom; L^\infty(0,T;V_1)) \cap L^\ell_\cP(\tom; L^4(0,T;V_2)) \cap L^\ell_\cP(\tom; W^{\alpha, \ell}(0,T;V_1^*));\\
    \widetilde{\mu}^\eps_\lambda & \in L^\ell_\cP(\tom; L^2(0,T;V_1)); \\
    \widetilde{I}^\eps_\lambda & \in L^\ell_\cP(\tom; C^0([0,T]; H)); \\
    \widetilde{W}^\eps_\lambda & \in L^\ell_\cP(\tom; C^0([0,T];U)).
\end{align*}
The properties of the limit transport field will readily follow from the following identification argument.
\subsubsection{Identification of the limit solution} \label{ssec:limit_n}
It is then left to show that the limiting processes solve a regularized version of problem \eqref{eq:ch}. First, we identify the limits of several terms of the sequence leveraging the strong convergence of the approximated order parameters. Let $\ell \geq 2$ be arbitrary but fixed. Given the Lipschitz continuity of $F'_\lambda$, it is straightforward to show that
\[
F'_\lambda(\widetilde{\varphi}^\eps_{\lambda,n}) \to F'_\lambda(\widetilde{\varphi}^\eps_\lambda) \qquad \text{in }L^p(\tom;L^2(0,T;H))
\]
for all $p < \ell$, and similarly, defining
\[
\widetilde{G}^\eps_\lambda: \Omega \times(0,T) \times V_1 \to \cL^2(U, H)
\]
given by
\[
\widetilde{G}^\eps_\lambda(\omega, t, \psi)[u_k] = \nabla \psi \cdot \widetilde{\b \Sigma}^\eps_\lambda u_k = \nabla \psi \cdot \widetilde{\b \sigma}_{\lambda, k}^\eps(\omega, t)
\]
for all $k \in \mathbb N$, 
we have
\[
\begin{split}
    & \|\widetilde{G}^\eps_{n}(\widetilde {\varphi}^\eps_{\lambda, n}(t))-\widetilde{G}^\eps_\lambda(\widetilde{\varphi}^\eps_\lambda(t))\|^2_{\cL^2(U,H)} \\
    & \hspace{1.3cm} \leq C\sum_{k \in \mathbb N} \|\nabla \widetilde {\varphi}^\eps_{\lambda, n}(t)\cdot \b P_n \widetilde{\b \sigma}_{\lambda, k}^\eps(t) - \nabla \widetilde {\varphi}^\eps_{\lambda}(t)\cdot \widetilde{\b \sigma}_{\lambda, k}^\eps(t)\|^2_H \\
    & \hspace{1.3cm} \leq C\sum_{k \in \mathbb N} \|\nabla \widetilde {\varphi}^\eps_{\lambda, n}(t)\|_{\b L^4(\OO)}^2\|\b P_n\widetilde{\b \sigma}_{\lambda, k}^\eps(t)-\widetilde{\b \sigma}_{\lambda, k}^\eps(t)\|^2_{\b L^4(\OO)}+\|\nabla \widetilde {\varphi}^\eps_{\lambda, n}(t) - \nabla \widetilde {\varphi}^\eps_{\lambda}(t) \|_{\b H}^2\|\widetilde{\b \sigma}_{\lambda, k}^\eps(t)\|^2_{\b L^\infty(\OO)}
\end{split}
\]
and therefore, integrating in time and taking $\tP$-expectations, we finally get
\[
\widetilde{G}^\eps_{n}(\widetilde {\varphi}^\eps_{\lambda, n}) \to \widetilde{G}^\eps_\lambda(\widetilde {\varphi}^\eps_{\lambda}) \qquad \text{in }L^p(\tom;L^2(0,T;\cL^2(U,H)))
\]
for all $p < \ell$. Finally, we deal with the Stratonovich corrector. Defining the limit corrector $\tLe : \tom \times [0,T] \times V_1 \to V_1^*$ so that
\[
\langle \tLe(\omega,t,\psi_1),\,\psi_2\rangle_{V_1}:=
\frac12\sum_{k=1}^n\int_\OO
[\widetilde{\b \sigma}_{\lambda, k}^\eps(\omega,t,x)\cdot\nabla \psi_1(x)]
[\widetilde{\b \sigma}_{\lambda, k}^\eps(\omega,t,x)\cdot\nabla\psi_2(x)]\:\d x,
\]
for all $\omega \in \Omega$, $t \in [0,T]$ and $\psi_1, \,\psi_2 \in V_1$, we observe that
\[
\begin{split}
    & \|\tLn\widetilde {\varphi}^\eps_{\lambda, n}(t) - \tLe\widetilde {\varphi}^\eps_{\lambda}\|_{V_1^*} \\
    & \hspace{0.7cm} = \sup_{\substack{v \in V_1 \\ \|v\|_{V_1} = 1}}\left[ \left|\sum_{k=1}^n [ \widetilde{\b \sigma}_{\lambda, k}^\eps(t)\cdot(\nabla\widetilde {\varphi}^\eps_{\lambda, n}(t) - \nabla\widetilde {\varphi}^\eps_{\lambda}(t))][\widetilde{\b \sigma}_{\lambda, k}^\eps(t)\cdot \nabla v]\right| + \left|\sum_{k > n}[\widetilde{\b \sigma}_{\lambda, k}^\eps(t)\cdot\nabla\widetilde {\varphi}^\eps_{\lambda}(t)][\widetilde{\b \sigma}_{\lambda, k}^\eps(t)\cdot \nabla v] \right|\right], 
\end{split}
\]
yielding immediately by interpolation
\[
\tLn\widetilde {\varphi}^\eps_{\lambda, n} \to \tLe\widetilde {\varphi}^\eps_{\lambda} \qquad \text{in }L^p(\tom;L^2(0,T;V_1^*))
\]
for all $p < \ell$. As far as the stochastic integral is concerned, it can be shown by standard arguments that (see also \cite{scarpa21, DGS})
\[
\widetilde{I}^\eps_\lambda = \int_0^\cdot \widetilde G^\eps_\lambda(\widetilde{\varphi}^\eps_\lambda(s)) \: \d \widetilde W^\eps_\lambda(s),
\]
namely, an $H$-valued $\tF^\eps_{\lambda, t}$-martingale, where the $(\tF^\eps_{\lambda, t})_{t \in [0,T]} \subset \tF$ denotes the natural filtration generated by the limit processes, i.e., 
\[
\tF^\eps_{\lambda, t} = \sigma\left\{ \widetilde{\varphi}^\eps_\lambda(s),\, \widetilde{I}^\eps_\lambda(s),\, \widetilde W^\eps_\lambda(s),\,s \in [0,t] \right\}.
\]
Eventually, it is possible to pass to the limit in the weak formulations of both \eqref{eq1_app} and \eqref{eq2_app}. Indeed, testing both equations for some $\psi \in V_1$ and letting $n\to+\infty$, thanks to preservation of laws through $X_{\lambda,n}^\eps$ and the dominated convergence theorem we conclude that
\begin{multline*}
			(\widetilde {\varphi}^\eps_{\lambda}(t),\psi)_H +
			\int_0^t\left[\int_{\OO}
			m(\widetilde {\varphi}^\eps_{\lambda}(t))\nabla  \widetilde \mu^\eps_\lambda(s)\cdot \nabla \psi\,\d x
            +\kappa\ip{\tLe(s,\widetilde {\varphi}^\eps_{\lambda}(s))}{\psi}_{V_1^*, V_1}
            \right]\,\d s\\
			= ( \widetilde {\varphi}^\eps_{\lambda}(0),\psi)_{H} +
			\left(\int_0^t \nabla\widetilde {\varphi}^\eps_{\lambda}(s)\cdot\widetilde{\b \Sigma}^\eps_\lambda(s)\,\d  \widetilde W^\eps_\lambda(s), \psi\right)_{H}
\end{multline*}
$\tP$-almost surely, and, in turn, the regularized chemical potential $\widetilde \mu^\eps_\lambda$ satisfies
\[
\int_\OO \widetilde \mu^\eps_\lambda(t)\psi \: \d x = \int_\OO \nabla \widetilde {\varphi}^\eps_{\lambda}(t) \cdot \nabla \psi \: \d x + \int_\OO F'_\lambda(\widetilde {\varphi}^\eps_{\lambda}(t)) \psi \: \d x
\]
for almost every $t \in [0,T]$ and $\tP$-almost surely. Therefore, the procedure results in the existence of a martingale solution
\[
\left(\widetilde{\Omega},\,
\widetilde{\cF}, \,
(\widetilde{\cF}^\eps_{\lambda,t})_{t\in[0,T]},\,
\widetilde{\P}, \,
\widetilde{W}^\eps_\lambda, \,
\widetilde{\b\Sigma}^\eps_\lambda,\,
\widetilde{\varphi}^\eps_\lambda
\right)
\]
of the regularized problem 
\begin{numcases}{}
  \label{eq1_app_lam}
  \d\widetilde{\varphi}^\eps_{\lambda} 
  - \div \left[m(\widetilde{\varphi}^\eps_{\lambda})
  \nabla\widetilde{\mu}^\eps_{\lambda}\right]\,\d t + 
  \kappa\tLe
  \widetilde{\varphi}^\eps_{\lambda}\,\d t
  =\nabla \widetilde{\varphi}^\eps_{\lambda}\cdot 
  \widetilde{\b\Sigma}^\eps_\lambda\,\d \widetilde W^\eps_\lambda
  \qquad&\text{in }$(0,T)\times\OO$,\\
  \label{eq2_app_lam}
  \widetilde{\mu}^\eps_{\lambda}=
  -\Delta\widetilde{\varphi}^\eps_{\lambda} +F_{\lambda}'(\widetilde{\varphi}^\eps_{\lambda})
  \qquad&\text{in }$(0,T)\times\OO$,\\
  \label{eq3_app_lam}
  \partial_{\b n}\widetilde{\varphi}^\eps_{\lambda} = 
  m(\widetilde{\varphi}^\eps_{\lambda})\partial_{\b n}
  \widetilde{\mu}^\eps_{\lambda} = 0
  \qquad&\text{in }$(0,T)\times\partial\OO$,\\
  \label{eq4_app_lam}
  \widetilde{\varphi}^\eps_{\lambda}(0)=\varphi_0
  \qquad&\text{in }$\OO$.
\end{numcases}
\begin{remark}
    In principle, the probability space may also depend on the parameters $\lambda$ and $\eps$. However, it is a standard matter to see that this dependence is uninfluential (and can even be get rid of) for the sake of the remainder of the proof. Therefore, with the aim of better readability, we omit to specify these dependencies explicitly.
\end{remark}

\subsection{Uniform estimates with respect to $\boldsymbol{\lambda}$}
\label{ssec:est_lam}
This second set of uniform estimates recasts the computations of Subsection \ref{ssec:est_n} in order to setup a compactness argument with respect to the Yosida parameter $\lambda \in (0,1)$. In the following, the symbol $\tE$ denotes expectations with respect to the probability $\tP$, while the value of $\eps \in (0,1)$ is still fixed.

\subsubsection{The energy estimate} First, we establish an energy inequality for the regularized problem by passing to the limit as $n\to+\infty$ in the corresponding estimate. Indeed, by using the equivalence of laws in \eqref{ito:en_app4} and
by letting $n\to+\infty$, 
the weak lower semicontinuity of norms as well as
the fact that 
\[
  \lim_{n\to+\infty}
  \norm{F_\lambda(P_n\varphi_0)- F_\lambda(\varphi_0)}_{L^1(\OO)}=0
  \qquad\forall\,\lambda\in(0,1)
\]
yield that
\begin{multline}
    \label{ito:en_app5}
    \tE\sup_{s\in[0,t]}
    \norm{\widetilde{\varphi}^\eps_{\lambda}(s)}_{V_1}^\ell
    +\tE\sup_{s\in[0,t]}
    \norm{F_\lambda(\widetilde{\varphi}^\eps_{\lambda}(s))}_{L^1(\OO)}^{\frac\ell2}\\
    +\tE
    \norm{\nabla\widetilde{\mu}^\eps_{\lambda}}_{L^2(0,t; \bH)}^{\ell}
    +\tE
    \norm{\widetilde{\varphi}^\eps_{\lambda}
    }_{L^2(0,t; V_2)}^{\ell}
    +\tE
    \norm{|\Psi''_\lambda(\widetilde{\varphi}^\eps_{\lambda})|^{\frac12}
    \nabla\widetilde{\varphi}^\eps_{\lambda}}^\ell_{L^2(0,t;\bH)}\\
    \leq C_{\eps,\ell}(\norm{\nabla\varphi_0}_\bH^\ell
    +\norm{F_\lambda(\varphi_0)}_{L^1(\OO)}^\ell)
    +C_{\eps,\ell}\int_0^t
    \tE
    \norm{\widetilde{\varphi}^\eps_{\lambda}(s)}_{V_1}^\ell\,\d s,
\end{multline}
where the implicit constants on the right-hand side
are independent of $\lambda$.
Since by the properties of $F_\lambda$ we have
\[
  \norm{F_\lambda(\varphi_0)}_{L^1(\OO)}
  \leq \norm{F(\varphi_0)}_{L^1(\OO)}\!,
\]
we infer by the Gronwall lemma that 
for every $\ell\geq2$, there exists a constant 
$C_{\eps,\ell}$, independent of $\lambda$, such that 
\begin{align}
    \label{est1_lam}
    \norm{\widetilde{\varphi}^\eps_{\lambda}}_{
    L^\ell_\cP(\tom;
    C^0([0,T]; V_1)\cap L^2(0,T; V_2))}
    &\leq C_{\eps,\ell},\\
    \label{est2_lam}
    \norm{\nabla\widetilde{\mu}^\eps_{\lambda}}_{
    L^\ell_\cP(\tom; L^2(0,T;\bH))}&\leq C_{\eps,\ell},\\
    \label{est3_lam}
    \norm{|\Psi_\lambda''(\widetilde{\varphi}^\eps_{\lambda})|^{\frac12}
    \nabla\widetilde{\varphi}^\eps_{\lambda}}_{
    L^\ell_\cP(\tom;L^2(0,T;\bH))}
    &\leq C_{\eps,\ell}.
\end{align}

\subsubsection{Further estimates}
As $F_\lambda'$ is not uniformly Lipschitz-continuous with respect to $\lambda$, we now test \eqref{eq2_app_lam}
by $\widetilde \varphi_\lambda^\eps-\overline{\varphi_0}$.
Since $\overline{\widetilde\varphi^\eps_\lambda(t)}=
\overline{\varphi_0}\in(-1,1)$ for every $t\in[0,T]$,
$\tP$-almost surely, by the Poincar\`e-Wirtinger
inequality and Assumption \ref{hyp:potential} we infer that,
for almost every $t\in(0,T)$,
\begin{align*}
\int_\OO |\nabla\widetilde{\varphi}^\eps_\lambda(t)|^2\,
\d x
+\int_\OO \Psi_\lambda'(\widetilde{\varphi}^\eps_\lambda(t))
(\widetilde{\varphi}^\eps_\lambda(t)-\overline{\varphi_0})
\,\d x
&=\int_\OO\left[\widetilde{\mu}^\eps_\lambda(t)
-R'(\widetilde{\varphi}^\eps_\lambda(t))\right]
(\widetilde{\varphi}^\eps_\lambda(t)-\overline{\varphi_0})\,\d x\\
&=\int_\OO\left[\widetilde{\mu}^\eps_\lambda(t)-\overline{\widetilde{\mu}^\eps_\lambda(t)}
-R'(\widetilde{\varphi}^\eps_\lambda(t))\right]
(\widetilde{\varphi}^\eps_\lambda(t)-\overline{\varphi_0})\,\d x\\
&\leq C\left[
\norm{\nabla\widetilde{\mu}^\eps_\lambda(t)}_\bH+
\norm{\widetilde{\varphi}^\eps_\lambda(t)}_H\right]
\norm{\nabla\widetilde{\varphi}^\eps_\lambda(t)}_\bH,
\end{align*}
where the constant $C$ is independent of 
$t$, $\lambda$, and $\eps$.
By also recalling the classical inequality 
(see e.g.~\cite{MZ04})
\[
\int_\OO\Psi_\lambda'(\widetilde{\varphi}_{\lambda}^\eps(t)) (\widetilde{\varphi}_{\lambda}^\eps(t)) - \overline{\varphi_0})\,\d x \geq C\|\Psi'_\lambda(\widetilde{\varphi}_{\lambda}^\eps(t))\|_{L^1(\OO)} - C,
\]
the estimate above together with \eqref{est1_lam}--\eqref{est3_lam} implies that 
\begin{equation}
    \label{est4_lam}
    \norm{\Psi_\lambda'(\widetilde{\varphi}_{\lambda}^\eps)}_{
    L^\ell_\cP(\widetilde{\Omega}; L^2(0,T;L^1(\OO)))}\leq C_{\eps,\ell}.
\end{equation}
By comparison in \eqref{eq2_app_lam},
the estimates \eqref{est2_lam} and \eqref{est4_lam}
imply, thanks again to the 
Poincar\`e-Wirtinger
inequality, that 
\begin{equation}
    \label{est5_lam}
    \norm{\widetilde{\mu}_{\lambda}^\eps}_{
    L^\ell_\cP(\widetilde{\Omega}; L^2(0,T;V_1))}\leq C_{\eps,\ell}.
\end{equation}
Testing now \eqref{eq2_app_lam} by $F_\lambda'(\widetilde{\varphi}_\lambda^\eps)$, 
it is immediate to check that 
\begin{align*}
\norm{F_\lambda'(\widetilde{\varphi}_\lambda^\eps)}_{L^2(0,T; H)}^2
&\leq
\int_0^T\int_\OO\left[
\Psi_\lambda'(\widetilde{\varphi}_\lambda^\eps(s))
|\nabla\widetilde{\varphi}_\lambda^\eps(s)|^2+
|F_\lambda'(\widetilde{\varphi}_\lambda^\eps(s))|
\right]\,\d x\,\d s\\
&=\int_0^T\int_\OO\left[
\widetilde{\mu}_\lambda^\eps(s)-
R'(\widetilde{\varphi}_\lambda^\eps(s))
\right]
F_\lambda'(\widetilde{\varphi}_\lambda^\eps(s))
\,\d x\,\d s\\
&\leq 
\frac12\norm{F_\lambda'(\widetilde{\varphi}_\lambda^\eps)}_{L^2(0,T; H)}^2
+\norm{\widetilde{\mu}_\lambda^\eps}_{L^2(0,T; H)}^2
+C_R^2\norm{\widetilde{\varphi}_\lambda^\eps}_{L^2(0,T; H)}^2,
\end{align*}
so that \eqref{est1_lam} and \eqref{est5_lam} yield
\begin{equation}
    \label{est6_lam}
    \norm{F_\lambda'(\widetilde{\varphi}_{\lambda}^\eps)}_{
    L^\ell_\cP(\widetilde{\Omega}; L^2(0,T;H))}\leq C_{\eps,\ell}.
\end{equation}
By iterating the same computations illustrated in Subsubsection \ref{ssec:further_n}, the following additional estimates are readily achieved:
\begin{align*}
    \|\widetilde \varphi^\eps_\lambda\|_{L^\ell_\cP(\tom;L^4(0,T;V_2) \cap W^{\alpha, \ell}(0,T;V_1^*))} & \leq C_{\eps, \ell}, \\
    \|\widetilde{G}^\eps_\lambda(\widetilde{\varphi}^\eps_\lambda)\|_{L^\ell_\cP(\tom;L^\infty(0,T;\cL^2(U,H)) \cap L^4(0,T;\cL^2(U,V_1))} & \leq C_{\eps, \ell}, \\
    \left\|\int_0^\cdot \widetilde{G}^\eps(\widetilde \varphi_{\lambda}^\eps(s)) \: \d \widetilde W^\eps_\lambda(s)\right\|_{L^\ell_\cP(\Omega; W^{\alpha, \ell}(0,T;H)) \cap W^{\beta, 4}(0,T;V_1))} & \leq C_{\eps, \ell}, \\
    \|\tLe\widetilde \varphi^\eps_\lambda\|_{L^\ell_\cP(\tom;C^0([0,T];V_1^*))} & \leq C_{\eps, \ell},
\end{align*}
for all $\alpha \in (0, \frac 12)$ and all $\ell \geq 4$.

\subsection{Passage to the limit as $\boldsymbol{\lambda\to0^+}$}
\label{ssec:limit_lam}
A second compactness argument is enabled by the previous uniform estimates. As many technicalities are dealt with following closely the argument of Subsection \ref{ssec:limit_as_n}, we shall omit some details, focusing instead on the main differences. Without relabeling, any passage to the limit as $\lambda \to 0^+$ is performed along an arbitrary vanishing sequence $\{\lambda_k\}_{k \in \enne} \subset (0,1)$.

\subsubsection{Tightness of the laws of Yosida solutions}
The analogue result to Lemma \ref{lem:tight} is given hereafter. Its proof is carried out following the same strategy and is therefore omitted.
\begin{lem} \label{lem:tight_lam}
    The family of laws of 
    \[
    \{\widetilde\varphi_{\lambda}^\eps,\, \widetilde G^\eps_{\lambda}(\widetilde\varphi_{\lambda}^\eps) \cdot \widetilde W_\lambda^\eps,\, \widetilde W_\lambda^\eps,\, \widetilde{\b \Sigma}^\eps_\lambda\}_{\lambda \in (0,1)}
    \]
    is tight in the space
    \[
    \Xi \times C^0([0,T]; H) \times C^0([0,T]; U_0) \times L^2(0,T;\cL^2(U,\b H)))
    \]
    where $\Xi = C^0([0,T];H) \cap L^2(0,T;V_1)$.
\end{lem} \noindent
Let us note that laws of the transport fields $\widetilde{\b \Sigma}^\eps$ and the Wiener processes $\widetilde W_\lambda^\eps$ are independent of $\lambda$, and therefore the tightness of their laws follows trivially in any separable Banach space where their paths are well defined.
Again, owing to the Prokhorov and Skorokhod theorems, there exists a probability space $(\com,\, \cf,\, \cp)$ and a sequence of random variables $\{Y_\lambda^\eps :(\com, \cf) \to (\tom, \tF)\}_{\lambda \in (0,1)}$ such that the law of $Y^\eps_\lambda$ equals $\tP$ for all $\lambda \in (0,1)$ and, up to subsequences, all of the following convergences hold in the limit $\lambda \to 0^+$:
\begin{align*}
    \widecheck{\varphi}^\eps_{\lambda} := \widetilde \varphi_{\lambda}^\eps \circ Y^\eps_\lambda \to \widecheck{\varphi}^\eps & \qquad \text{in }L^p(\com; C^0([0,T];H) \cap L^2(0,T;V_1)),\,\quad \forall \, p < \ell; \\
    \widecheck{\varphi}^\eps_{\lambda} \overset{*}{\rightharpoonup} \widecheck{\varphi}^\eps & \qquad \text{in }L^\ell_w(\com; L^\infty([0,T];V_1)) \cap L^\ell(\tom; W^{\alpha, \ell}(0,T;V_1^*)); \\
    \widecheck{\varphi}^\eps_{\lambda} \rightharpoonup \widecheck{\varphi}^\eps & \qquad \text{in }L^\ell(\com; L^4([0,T];V_2)); \\
    \widecheck{\mu}^\eps_{\lambda} := \widetilde {\mu}^\eps_{\lambda} \circ Y^\eps_\lambda \rightharpoonup\widecheck{\mu}^\eps & \qquad \text{in }L^\ell(\com; L^2([0,T];V_1)); \\
    \widecheck{I}^\eps_{\lambda} := (\widetilde G^\eps_\lambda(\widetilde \varphi_{\lambda}^\eps) \cdot \widetilde W_\lambda^\eps) \circ Y^\eps_\lambda  \to \widecheck{I}^\eps & \qquad \text{in }L^p(\com; C^0([0,T];H)),\, \quad \forall \, p < \ell; \\
    \widecheck{W}_\lambda^\eps := \widetilde W_\lambda^\eps \circ Y^\eps_\lambda \to \widecheck W^\eps & \qquad \text{in }L^p(\com;C^0([0,T];U_0)),\,\quad \forall \, p < \ell; \\
    \widecheck{\b \Sigma}^\eps_\lambda := \widetilde{\b \Sigma}^\eps_\lambda \circ Y^\eps_\lambda \to \widecheck{\b \Sigma}^\eps & \qquad \text{in }L^p(\com; L^2([0,T];\cL^2(U, \b H)),\,\qquad \forall \, p < \ell,
\end{align*}
where the limiting processes satisfy
\begin{align*}
    \widecheck{\varphi}^\eps & \in L^\ell_\cP(\com; C^0([0,T];H)) \cap L^\ell_w(\tom; L^\infty(0,T;V_1)) \cap L^\ell_\cP(\com; L^4(0,T;V_2)) \cap L^\ell_\cP(\com; W^{\alpha, \ell}(0,T;V_1^*));\\
    \widecheck{\mu}^\eps & \in L^\ell_\cP(\com; L^2(0,T;V_1)); \\
    \widecheck{I}^\eps & \in L^\ell_\cP(\com; C^0([0,T]; H)); \\
    \widecheck{W} & \in L^\ell_\cP(\com; C^0([0,T];U_0));\\
    \widecheck{\b \Sigma}^\eps & \in L^\ell_\cP(\com; L^2(0,T; \cL^2(U,H))).
\end{align*}

\subsubsection{Identification of the limit solution} \label{ssec:id_lam}
As
\[
\widecheck{\varphi}^\eps_{\lambda} \to \widecheck{\varphi}^\eps  
\]
in the Hilbert space $L^2(\com;L^2(0,T;H))$, and given that the equivalence of laws entails that
\[
\Psi'_\lambda(\widecheck{\varphi}^\eps_\lambda) \text{ is uniformly bounded in }L^2(\com;L^2(0,T;H))
\]
we have that, up to subsequences,
\[
\Psi'_\lambda(\widecheck{\varphi}^\eps_\lambda) \rightharpoonup \xi = \Psi'(\widecheck{\varphi}^\eps) \qquad \text{in }L^\ell(\com;L^2(0,T;H))
\]
by the strong-weak closure of maximal monotone operators. Then, since $R'$ is Lipschitz continuous by Assumption \ref{hyp:potential}, it is immediate to show that
\[
F'_\lambda(\widecheck{\varphi}^\eps_\lambda) \rightharpoonup F'(\widecheck{\varphi}^\eps) \qquad \text{in }L^\ell(\com;L^2(0,T;H))
\]
for all $\ell \geq 2$. Arguing as in Subsubsection \ref{ssec:limit_n}, we define
\[
\widecheck{G}^\eps: \Omega \times(0,T) \times V_1 \to \cL^2(U, H)
\]
by
\[
\widecheck{G}^\eps(\omega, t, \psi)[u_k] = \nabla \psi \cdot \widecheck{\b \Sigma}^\eps u_k = \nabla \psi \cdot \widecheck{\b \sigma}_{ k}^\eps(\omega, t)
\]
for all $k \in \mathbb N$. It is straightforward to show that
\[
\widecheck{G}^\eps_\lambda(\widecheck {\varphi}^\eps_{\lambda}) \to \widecheck{G}^\eps(\widecheck {\varphi}^\eps) \qquad \text{in }L^p(\tom;L^2(0,T;\cL^2(U,H)))
\]
for all $p < \ell$, and, analogously,
\[
{\widecheck{\mathcal{L}}_{\text{Strat}}^{\eps,\lambda}} \to \cLe\widecheck {\varphi}^\eps \qquad \text{in }L^p(\tom;L^2(0,T;V_1^*))
\]
for all $p < \ell$. Moreover, the limit stochastic integral $\widecheck I^\eps$ satisfies
\[
\widecheck I^\eps = \int_0^\cdot \widecheck G^\eps(\widecheck \varphi^\eps(s)) \: \d \widecheck W^\eps(s),
\]
that is, an $H$-valued $\widecheck \cF^\eps_{t}$-adapted martingale, where
\[
\widecheck \cF^\eps_{t} := \sigma\left\{ \widecheck \varphi^\eps(s), \, \widecheck I^\eps(s),\,\widecheck W^\eps(s),\,s \in [0,t] \right\}.
\]Owing to all of the above, in the limit as $\lambda \to 0^+$, testing \eqref{eq1_app_lam} and \eqref{eq2_app_lam} against any $\psi \in V_1$ and letting $\lambda \to 0^+$, thanks once more to preservation of laws through $Y_\lambda^\eps$ and the dominated convergence theorem we have
\begin{multline*}
			(\widecheck {\varphi}^\eps(t),\psi)_H +
			\int_0^t\left[\int_\OO
			m(\widecheck {\varphi}^\eps(t))\nabla  \widecheck \mu^\eps(s)\cdot \nabla \psi\,\d x
            +\kappa\ip{\cLe(s,\widecheck {\varphi}^\eps(s))}{\psi}_{V_1^*, V_1}
            \right]\,\d s\\
			= ( \widecheck {\varphi}^\eps(0),\psi)_{H} +
			\left(\int_0^t \nabla\widecheck {\varphi}^\eps(s)\cdot\widecheck{\b \Sigma}^\eps(s)\,\d  \widecheck W(s), \psi\right)_{H}
\end{multline*}
$\cp$-almost surely. The equation for the chemical potential $\widecheck \mu^\eps$ reads
\[
\int_\OO \widecheck \mu^\eps(t)\psi \: \d x = \int_\OO \nabla \widecheck {\varphi}^\eps(t) \cdot \nabla \psi \: \d x + \int_\OO F'(\widecheck {\varphi}^\eps(t)) \psi \: \d x
\]
for almost every $t \in [0,T]$ and $\cp$-almost surely. More precisely, we obtain a martingale solution
\[
\left(\widecheck{\Omega},\,
\widecheck{\cF}, \,
(\widecheck{\cF}^\eps_{t})_{t\in[0,T]},\,
\widecheck{\P}, \,
\widecheck{W}^\eps, \,
\widecheck{\b\Sigma}^\eps,\,
\widecheck{\varphi}^\eps
\right)
\]
of the regularized problem 
\begin{numcases}{}
   \label{eq1_app_eps}
  \d\widecheck{\varphi}^\eps 
  - \div \left[m(\widecheck{\varphi}^\eps)
  \nabla\widecheck{\mu}^\eps\right]\,\d t + 
  \kappa\cLe\widecheck{\varphi}^\eps\,\d t
  =\nabla \widecheck{\varphi}^\eps\cdot \widecheck{\b\Sigma}^\eps\,\d\widecheck W^\eps
  \qquad&\text{in }$(0,T)\times\OO$,\\
  \label{eq2_app_eps}
  \widecheck\mu^\eps=-\Delta\widecheck\varphi^\eps + 
  F'(\widecheck\varphi^\eps)
  \qquad&\text{in }$(0,T)\times\OO$,\\
  \label{eq3_app_eps}
  \partial_{\b n}\widecheck\varphi^\eps = 
  m(\widecheck\varphi^\eps)\partial_{\b n}
  \widecheck\mu^\eps = 0
  \qquad&\text{in }$(0,T)\times\partial\OO$,\\
  \label{eq4_app_eps}
  \widecheck\varphi^\eps(0)=\varphi_0
  \qquad&\text{in }$\OO$,
\end{numcases}
where we recall that $\widecheck{\b\Sigma}^\eps\laweq\b\Sigma^\eps$ for every $\eps \in (0,1)$.

\subsection{Uniform estimates with respect to $\boldsymbol{\eps}$}
\label{ssec:est_eps}
Finally, we present a third set of estimates that eventually enables us to get rid of the final approximation scheme. Although the needed estimates are the same, the strategies are different and leverage a crucial $L^\infty$-bound.

\subsubsection{A preliminary estimate}
The previous passage to the limit entails that
\[
\Psi'(\widecheck \varphi^\eps(\omega,t)) \in H
\]
for $\cp \otimes \d t$-almost any $(\omega, \,t) \in \com \times [0,T]$. Given the singular nature of $\Psi'$ prescribed by Assumption \ref{hyp:potential}, a simple contradiction argument shows that
\begin{equation} \label{eq:Linf_est}
    |\widecheck \varphi^\eps(\omega, t, x)| \leq 1
\end{equation}
for $\cp \otimes \d t \otimes \d x$-almost any $(\omega, t, x) \in \com \times [0,T] \times \OO$. As we shall see, this key property enables to refine the previous uniform estimates.

\subsubsection{The energy estimate}
The It\^o formula for the free energy functional yields
\begin{multline}
\label{ito:en_app_eps}
\frac12\norm{\nabla\widecheck\varphi^\eps(t)}_{\bH}^2
    +\norm{F(\widecheck\varphi^\eps(t))}_{L^1(\OO)}
    +\int_0^t\int_\OO
    m(\widecheck\varphi^\eps(s))
    |\nabla\widecheck\mu^\eps(s)|^2\,\d x\,\d s \\
    +\frac\kappa2\int_0^t\sum_{k\in\enne}
    \int_\OO
    [\nabla \widecheck\varphi^\eps(s)\cdot \widecheck{\b{\sigma}}^\eps_k(s)]
[\nabla\widecheck\mu^\eps(s)\cdot \widecheck{\b{\sigma}}^\eps_k(s)]\,\d x\,\d s\\=\frac12\norm{\nabla\varphi_0}_{\bH}^2
    +\norm{F(\varphi_0)}_{L^1(\OO)}
    +\int_0^t\left(\widecheck\mu^\eps(s), \nabla \widecheck\varphi^\eps(s)
    \cdot\widecheck{\b\Sigma}^\eps(s)\,\d\widecheck W^\eps(s)\right)_{H}\\
    +\frac12\int_0^t\sum_{k\in\enne}
    \int_\OO\left[
    |\nabla[\nabla \widecheck\varphi^\eps(s)\cdot
    \widecheck{\b\sigma}^\eps_k(s)]|^2
    +F''(\widecheck\varphi^\eps(s))
    |\nabla \widecheck\varphi^\eps(s)\cdot\widecheck{\b\sigma}^\eps_k(s)|^2
    \right]\,\d x\,\d s
\end{multline}
for every $t\in[0,T]$, $\widecheck\P$-almost surely. 
By arguing as in Subsection~\ref{ssec:est_n}, 
on the left-hand side we have 
\[
\int_0^t\int_\OO
    m(\widecheck\varphi^\eps(s))
    |\nabla\widecheck\mu^\eps(s)|^2\,\d x\,\d s
    \geq c_m\norm{\nabla\widecheck\mu^\eps}^2_{L^2(0,t;\bH)},
\]
and
\begin{align*}
    &\frac\kappa2\int_0^t\sum_{k\in\enne}
    \int_\OO
    |\nabla \widecheck\varphi^\eps(s)\cdot \widecheck{\b{\sigma}}^\eps_k(s)|
|\nabla\widecheck\mu^\eps(s)\cdot \widecheck{\b{\sigma}}^\eps_k(s)|\,\d x\,\d s\\
&\hspace{2cm}\leq\frac\kappa2\int_0^t
\sum_{k\in\enne}\norm{\widecheck{\b\sigma}_k^\eps(s)}_{\b{L}^\infty(\OO)}^2
\norm{\nabla\widecheck\varphi^\eps(s)}_\bH
\norm{\nabla\widecheck\mu^\eps}_\bH\,\d s\\
&\hspace{2cm}\leq\frac{c_m}4
\norm{\nabla\widecheck\mu^\eps}_{L^2(0,t;\bH)}^2
+\frac{\kappa^2 C_\infty^2}{4c_m}\int_0^t\left(
\sum_{k\in\enne}\norm{\widecheck{\b\sigma}^\eps_k(s)}_{\W}^2
\right)^2
\norm{\nabla\widecheck\varphi^\eps(s)}_\bH^2\,\d s.
\end{align*}
Moreover, recalling the properties of the noise coefficient
collected in Remark~\ref{rmk:prop_noise}, one has that 
\begin{multline*}
    \frac12\int_0^t\sum_{k\in\enne}
    \int_\OO\left[
    |\nabla[\nabla \widecheck\varphi^\eps(s)
    \cdot\widecheck{\b\sigma}^\eps_k(s)]|^2
    +F''(\widecheck\varphi^\eps(s))
    |\nabla \widecheck\varphi^\eps(s)\cdot\widecheck{\b\sigma}^\eps_k(s)|^2
    \right]\,\d x\,\d s\\
    \leq C\int_0^t\left(\sum_{k\in\enne}
    \norm{\widecheck{\b\sigma}^\eps_k(s)}_{\W}^2\right)
    \left[\norm{\widecheck\varphi^\eps(s)}_{V_2}^2
    +\norm{|F''(\widecheck\varphi^\eps(s))|^{\frac12}
    \nabla\widecheck\varphi^\eps(s)}_{\bH}^2
    \right]\,\d s.
\end{multline*}
By collecting all of the above and by rearranging the terms in \eqref{ito:en_app_eps} we obtain then 
\begin{multline}
\label{ito:en_app2_eps}\frac12\norm{\nabla\widecheck\varphi^\eps(t)}_{\bH}^2
    +\norm{F(\widecheck\varphi^\eps(t))}_{L^1(\OO)}
    +\frac{c_m}2
\norm{\nabla\widecheck\mu^\eps}_{L^2(0,t;\bH)}^2\\
\leq C\norm{\nabla\varphi_0}_{\bH}^2
    +\norm{F(\varphi_0)}_{L^1(\OO)}
    +\frac{\kappa^2}{4c_m}\int_0^t\left(
\sum_{k\in\enne}
\norm{\widecheck{\b\sigma}^\eps_k(s)}_{\W}^2
\right)^2
\norm{\nabla\widecheck\varphi^\eps(s)}_\bH^2\,\d s\\
    +C\int_0^t\left(\sum_{k\in\enne}
    \norm{\widecheck{\b\sigma}^\eps_k(s)}_{\W}^2\right)
    \left[\norm{\widecheck\varphi^\eps(s)}_{V_2}^2
    +\norm{|F''(\widecheck\varphi^\eps(s))|^{\frac12}
    \nabla\widecheck\varphi^\eps(s)}_{\bH}^2
    \right]\,\d s\\
    +\int_0^t\left(\widecheck\mu^\eps(s), \nabla \widecheck\varphi^\eps(s)
    \cdot\widecheck{\b\Sigma}^\eps(s)\,\d\widecheck W^\eps(s)\right)_{H}.
\end{multline}
Now, given any $K > 0$, by testing \eqref{eq2_app_eps} by 
$-K\Delta\widecheck\varphi^\eps$, 
we deduce that
\begin{align*}
K
\left[
\norm{\Delta\widecheck\varphi^\eps(s)}_H^2
+\int_\OO \Psi''(\widecheck\varphi^\eps(s))
|\nabla\widecheck\varphi^\eps(s)|^2 \: \d x
\right]
&=K\int_\OO\nabla\widecheck\mu^\eps(s)\cdot
\nabla\widecheck\varphi^\eps(s)\,\d x
-K\int_\OO R''(\widecheck\varphi^\eps(s))|
\nabla\widecheck\varphi^\eps(s)|^2\,\d x\\
&\leq K\norm{\nabla\widecheck\mu^\eps(s)}_{\bH}
\norm{\nabla\widecheck\varphi^\eps(s)}_\bH
+KC_R\norm{\nabla\widecheck\varphi^\eps(s)}_\bH^2,
\end{align*}
for any $K > 0$, $\delta > 0$ and almost any $s \in [0,T]$, $\widecheck \P$-almost surely. In turn, thanks to the Young inequality, it follows that 
\begin{align*}
    &K\int_0^t
    \left(\sum_{k\in\enne}
    \norm{\widecheck{\b\sigma}^\eps_k(s)}_{\W}^2\right)
    \left[
    \norm{\widecheck\varphi^\eps(s)}_{V_2}^2
    +\norm{|\Psi''(\widecheck\varphi^\eps(s))|^{\frac12}
    \nabla\widecheck\varphi^\eps(s)}_{\bH}^2
    \right]\,\d s\\
    &\hspace{1.5cm}\leq K\int_0^t
    \left(\sum_{k\in\enne}
    \norm{\widecheck{\b\sigma}^\eps_k(s)}_{\W}^2\right)
    \left[ |\OO|+
    \norm{\nabla\widecheck\mu^\eps(s)}_{\bH}
\norm{\nabla\widecheck\varphi^\eps(s)}_\bH
+C_R\norm{\nabla\widecheck\varphi^\eps(s)}_\bH^2
    \right]\,\d s\\
    &\hspace{1.5cm}\leq\frac{c_m}4\norm{\nabla\widecheck\mu^\eps}_{L^2(0,t;\bH)}^2
    +K^2\left(\frac1{c_m}+C_R\right)\int_0^t
    \left[1+\left(\sum_{k\in\enne}
    \norm{\widecheck{\b\sigma}^\eps_k(s)}_{\W}^2\right)^2\right]
    \norm{\nabla\widecheck\varphi^\eps(s)}_\bH^2\,\d s \\
    & \hspace{8cm} +K|\OO|\int_0^t \sum_{k\in\enne}
    \norm{\widecheck{\b\sigma}^\eps_k(s)}_{\W}^2\! \d s.
\end{align*}
By choosing $K=2C$ and by summing this last inequality 
to \eqref{ito:en_app2_eps} we infer then that 
\begin{multline}
\label{ito:en_app3_eps}  \frac12\norm{\nabla\widecheck\varphi^\eps(t)}_{\bH}^2
    +\norm{F(\widecheck\varphi^\eps(t))}_{L^1(\OO)}
    +\frac{c_m}4
\norm{\nabla\widecheck\mu^\eps}_{L^2(0,t;\bH)}^2\\
    +C\int_0^t\left(\sum_{k\in\enne}
    \norm{\widecheck{\b\sigma}^\eps_k(s)}_{\W}^2\right)
    \left[\norm{\widecheck\varphi^\eps(s)}_{V_2}^2
    +\norm{|\Psi''(\widecheck\varphi^\eps(s))|^{\frac12}
    \nabla\widecheck\varphi^\eps(s)}_{\bH}^2
    \right]\,\d s\\
    \leq C\norm{\nabla\varphi_0}_{\bH}^2
    +\norm{F(\varphi_0)}_{L^1(\OO)}
    +\frac12\int_0^t
    \widecheck L_\eps'(s)
\norm{\nabla\widecheck\varphi^\eps(s)}_\bH^2\,\d s\\+\int_0^t\left(\widecheck\mu^\eps(s), \nabla \widecheck\varphi^\eps(s)
    \cdot\b\Sigma^\eps(s)\,\d\widecheck W^\eps(s)\right)_{H}+2C|\OO|\int_0^t \sum_{k\in\enne}
    \norm{\widecheck{\b\sigma}^\eps_k(s)}_{\W}^2\!\d s,
\end{multline}
where the real valued process 
$\widecheck L_\eps\in L^1(\widecheck\Omega; W^{1,1}(0,T))$ is defined as
\[
  \widecheck L_\eps(t):=
  2\left(\frac{\kappa^2}{4c_m}+\frac{4C^2}{c_m}+4C^2C_R\right)
    \int_0^t\left[1+C_R+\left(
\sum_{k\in\enne}\norm{\widecheck{\b\sigma}^\eps_k(s)}_{
\W}^2
\right)^2\right]\,\d s, \quad t\in[0,T].
\]
The facts that 
$\widecheck{\b\Sigma}^\eps\laweq
\b\Sigma^\eps$ for every $\eps\in(0,1)$ and
$\norm{\b\sigma_k^\eps}_{\W}
\leq\norm{\b\sigma_k}_{\W}$
almost everywhere in $\Omega\times(0,T)$
for every $k\in\enne$ imply,
together with 
Assumption \ref{hyp:noise}, that 
$\widecheck L_\eps$ is well defined and that
for every $\ell\geq1$ there exists a constant 
$C_\ell$, independent of $\eps$, such that
$\exp\norm{\widecheck L^\eps}_{W^{1,1}([0,T])}\in L^\ell(\widecheck\Omega)$ with 
\begin{equation}
    \label{est_L}
    \norm{\exp\norm{\widecheck L_\eps}_{W^{1,1}(0,T)}}_{L^\ell(\widecheck\Omega)}\leq C_\ell.
\end{equation}
Now, from \eqref{ito:en_app3_eps} 
and applying the It\^{o} lemma to the functional
\[
t \mapsto \exp(-2\widecheck L_\eps(t))\mathcal E(\widecheck \varphi^\eps(t))
\]
it follows that 
\begin{multline}
\label{ito:en_app4_eps}
e^{-2\widecheck L_\eps(t)}\left[\frac12 
    \norm{\nabla\widecheck\varphi^\eps(t)}_{\bH}^2
    +\norm{F(\widecheck\varphi^\eps(t))}_{L^1(\OO)}\right]
    +\frac{c_m}4
\norm{e^{-\widecheck L_\eps}\nabla\widecheck\mu^\eps}_{L^2(0,t;\bH)}^2\\
    +C\int_0^t e^{-2\widecheck L_\eps(s)}\left(\sum_{k\in\enne}
    \norm{\widecheck{\b\sigma}^\eps_k(s)}_{\W}^2\right)
    \left[\norm{\widecheck\varphi^\eps(s)}_{V_2}^2
    +\norm{|\Psi''(\widecheck\varphi^\eps(s))|^{\frac12}
    \nabla\widecheck\varphi^\eps(s)}_{\bH}^2
    \right]\,\d s\\+
    \int_0^t
    \widecheck L_\eps'(s) e^{-2\widecheck L_\eps(s)}\left[
    \frac12\norm{\nabla\widecheck\varphi^\eps(s)}_\bH^2
    +\norm{F(\widecheck\varphi^\eps(s))}_{L^1(\OO)}\right]\,\d s\\
    \leq C\norm{\nabla\varphi_0}_{\bH}^2
    +\norm{F(\varphi_0)}_{L^1(\OO)}
    +\int_0^te^{-2\widecheck L_\eps(s)}\left(\widecheck\mu^\eps(s), \nabla \widecheck\varphi^\eps(s)
    \cdot\widecheck{\b\Sigma}^\eps(s)\,\d\widecheck W^\eps(s)\right)_{H}.
\end{multline}
Therefore, we are only left with estimating the stochastic integral.
To this end, recalling that 
\[
\int_0^t\left(\widecheck\mu^\eps(s), 
\nabla \widecheck\varphi^\eps(s))
    \cdot\widecheck{\b\Sigma}^\eps(s)\,\d\widecheck W(s)\right)_{H}
    =-\int_0^t\left(\nabla\widecheck\mu^\eps(s), 
    \widecheck\varphi^\eps(s)
    \widecheck{\b\Sigma}^\eps(s)\,\d\widecheck W(s)\right)_{\bH},
\]
we can use the Burkholder-Davis-Gundy, H\"older, and Young
inequalities, together with \eqref{eq:Linf_est} and the facts that $\widecheck{\b\Sigma}^\eps\laweq\b\Sigma^\eps$, and 
$\norm{\b\sigma_k^\eps}_{\bH}
\leq\norm{\b\sigma_k}_{\bH}$
in $\Omega\times(0,T)$
for every $k\in\enne$, to infer that
for every $\ell\in[2,q]$ and $\delta>0$,
\begin{align*}
    &\widecheck{\E}\sup_{t \in [0,T]}\left|
    \int_0^te^{-2\widecheck L_\eps(s)}\left(\widecheck\mu^\eps(s), \nabla \widecheck\varphi^\eps(s)
    \cdot\widecheck{\b\Sigma}^\eps(s)\,\d\widecheck W(s)\right)_{H}
    \right|^{\frac\ell2}\\
    &\hspace{2cm}\leq C_\ell
    \widecheck{\E}\left[\int_0^t e^{-4\widecheck L_\eps(s)}
    \norm{\nabla\widecheck\mu^\eps(s)}_{\bH}^2
    \norm{\widecheck\varphi^\eps(s)\widecheck{\b\Sigma}^\eps(s)}_{\cL^2(U,\bH)}^2
    \,\d s\right]^{\frac\ell4}\\
    &\hspace{2cm}\leq C_{\ell}
    \widecheck{\E}\left[\int_0^t e^{-4\widecheck L_\eps(s)}
    \norm{\nabla\widecheck\mu^\eps(s)}_{\bH}^2
    \norm{\widecheck\varphi^\eps(s)}_{L^\infty(\OO)}^2
    \sum_{k\in\enne}
    \norm{\widecheck{\b\sigma}^\eps_k(s)}_{\bH}^2
    \,\d s\right]^{\frac\ell4}\\
    &\hspace{2cm}\leq C_\ell\widecheck{\E}\left[
    \norm{e^{-\widecheck L_\eps}\nabla\widecheck\mu^\eps}_{L^2(0,t;\bH)}^{\frac\ell2}
    \left(\sup_{s\in[0,T]}\sum_{k\in\enne}
\norm{\widecheck{\b\sigma}^\eps_k(s)}^{2}_{\bH}\right)^{\frac\ell4}
\right]\\
    &\hspace{2cm}\leq \delta
    \widecheck{\E}\norm{e^{-\widecheck L_\eps}\nabla\widecheck\mu^\eps}^\ell_{L^2(0,t;\bH)}
    +\frac{C_{\ell}}\delta,
\end{align*}
where $C_\ell>0$ is finite and independent of $\eps$
thanks to assumption \ref{hyp:noise}.
Hence, we
can raise \eqref{ito:en_app4_eps} to the $\frac\ell2$-power,
take supremums in time, and
choose $\delta>0$ small enough 
in order to suitably rearrange the terms. 
By recalling also that 
$|\overline{\widecheck\varphi^\eps(t)}|=|\overline{\varphi_0}|\leq1$ for all $t\in[0,T]$ and that $\widecheck L_\eps'\geq0$, 
thanks to the Poincar\`e-Wirtinger inequality we obtain that
for every $\ell\in[2,q]$,
\begin{multline}
    \label{ito:en_app5_eps}
    \widecheck{\E}\sup_{s\in[0,T]}
    e^{-\ell \widecheck L_\eps(s)}
    \norm{\widecheck\varphi^\eps(s)}_{V_1}^\ell
    +\widecheck{\E}\sup_{s\in[0,T]}
    e^{-\ell \widecheck L_\eps(s)}
    \norm{F(\widecheck\varphi^\eps(s))}_{L^1(\OO)}^{\frac\ell2}
    +\widecheck{\E}\norm{e^{-\widecheck L_\eps}
    \nabla\widecheck\mu^\eps}_{L^2(0,T; \bH)}^{\ell}\\
    \leq C_{\ell}\left(1+\norm{\varphi_0}_{V_1}^\ell
    +\norm{F(\varphi_0)}_{L^1(\OO)}^\ell\right),
\end{multline}
where the implicit constant appearing on the right-hand side
has been updated and is independent of $\eps$.
By the monotonicity and continuity of the real exponential function,
for every $p\in[2,q)$ one has,
thanks to the H\"older inequality
with exponents $\frac{q}{p}>1$ and $\frac{q}{q-p}>1$,
as well as the estimate \eqref{ito:en_app5_eps}, that
\begin{align*}
    &\widecheck{\E}\sup_{s\in[0,T]}
    \norm{\widecheck\varphi^\eps(s)}_{V_1}^p
    +\widecheck{\E}\sup_{s\in[0,T]}
    \norm{F(\widecheck\varphi^\eps(s))}_{L^1(\OO)}^{\frac{p}{2}}
    +\widecheck{\E}\norm{\nabla\widecheck\mu^\eps}_{L^2(0,T; \bH)}^{p}\\
    &\hspace{2cm}\leq\widecheck{\E}
    \left[\exp\left(p \sup_{t\in[0,T]}\widecheck L_\eps(t)\right)
    \cdot
    \sup_{s\in[0,T]}
    e^{-p \widecheck L_\eps(s)}\norm{\widecheck\varphi^\eps(s)}_{V_1}^p
    \right]\\
    &\hspace{4cm}+\widecheck{\E}
    \left[
    \exp\left(p \sup_{t\in[0,T]}\widecheck L_\eps(t)\right)
    \cdot
    \sup_{s\in[0,T]}
    e^{-p \widecheck L_\eps(s)}\norm{F(\widecheck\varphi^\eps(s))}_{L^1(\OO)}^{\frac{p}2}
    \right]\\
    &\hspace{4cm}+\widecheck{\E}\left[
    \exp\left(p \sup_{t\in[0,T]}\widecheck L^\eps(t)\right)\cdot
    \left(\int_0^Te^{-2\widecheck L_\eps(s)}
    \norm{\nabla\widecheck\mu^\eps(s)}_\bH^2\,\d s\right)^{\frac{p}2}
    \right]\\
    &\hspace{2cm}\leq \widecheck{\E}
    \left[\exp\left(\frac{pq}{q-p}\sup_{t\in[0,T]}\widecheck L_\eps(t)\right)
    \right]^{\frac{q-p}{q}}
    C_q^{\frac pq}\left(1+\norm{\varphi_0}_{V_1}^q
    +\norm{F(\varphi_0)}_{L^1(\OO)}^q\right)^{\frac pq}.
\end{align*}
Hence, estimate \eqref{est_L} yields
that for all $p\in[2,q)$
\begin{align}
  \label{est1_eps}
\norm{\widecheck\varphi^\eps}_{L^p_w(\widecheck\Omega; L^\infty(0,T; V_1))}&\leq C_p,\\
  \label{est2_eps}
  \norm{F(\widecheck\varphi^\eps)}_{
  L^{\frac{p}2}_\cP(\widecheck\Omega; L^\infty(0,T; L^1(\OO)))}&\leq C_p,\\
  \label{est3_eps}
  \norm{\nabla\widecheck\mu^\eps}_{L^p_\cP(\widecheck\Omega; L^2(0,T; \bH))}&\leq C_p,
\end{align}
where the constant $C_p>0$ is independent of $\eps$.

\subsubsection{Further estimates}
We are now in a position to proceed as in Subsection~\ref{ssec:est_lam}, namely
by testing \eqref{eq2_app_eps}
by $\widecheck\varphi^\eps-\overline{\varphi_0}$.
Recalling that $\overline{\widecheck\varphi^\eps(t)}=
\overline{\varphi_0}\in(-1,1)$ for every $t\in[0,T]$,
$\widecheck\P$-almost surely, by the Poincar\'e-Wirtinger
inequality and assumption \ref{hyp:potential} we infer that,
for almost every $t\in(0,T)$,
\begin{align*}
\int_\OO |\nabla\widecheck\varphi^\eps(t)|^2 \: \d x
+\int_\OO \Psi'(\widecheck\varphi^\eps(t))
(\widecheck\varphi^\eps(t)-\overline{\varphi_0})  \: \d x
&=\int_\OO\left[\widecheck\mu^\eps(t)
-R'(\widecheck\varphi^\eps(t))\right]
(\widecheck\varphi^\eps(t)-\overline{\varphi_0})  \: \d x\\
&=\int_\OO\left[\widecheck\mu^\eps(t)-\overline{\widecheck\mu^\eps(t)}
-R'(\widecheck\varphi^\eps(t))\right]
(\widecheck\varphi^\eps(t)-\overline{\varphi_0})  \: \d x\\
&\leq C\left[
\norm{\nabla\widecheck\mu^\eps(t)}_\bH+
\norm{\widecheck\varphi^\eps(t)}_H\right]
\norm{\nabla\widecheck\varphi^\eps(t)}_\bH,
\end{align*}
where the constant $C$ is independent of 
$t$ and $\eps$.
By exploiting the classical inequality 
(see e.g.~\cite{MZ04})
\[
	\int_\OO\Psi'(\widecheck\varphi^\eps(t)) (\widecheck\varphi^\eps(t)) - \overline{\varphi_0})  \: \d x \geq C\|\Psi'(\widecheck\varphi^\eps(t))\|_{L^1(\OO)} - C
\]
together with \eqref{est1_eps}--\eqref{est3_eps}
we obtain
\begin{equation}
    \label{est4_eps}
    \norm{\Psi'(\widecheck\varphi^\eps)}_{
    L^p_\cP(\widecheck\Omega; L^2(0,T;L^1(\OO)))}\leq C_{p}.
\end{equation}
By comparison in \eqref{eq2_app_eps},
the estimates \eqref{est2_eps} and \eqref{est4_eps}
imply, thanks again to the 
Poincar\`e-Wirtinger
inequality, that 
\begin{equation}
    \label{est5_eps}
    \norm{\widecheck\mu^\eps}_{
    L^p_\cP(\widecheck\Omega; L^2(0,T;V_1)}\leq C_{p}.
\end{equation}
Testing \eqref{eq2_app_eps} by $-\Delta\widecheck\varphi^\eps + F'(\widecheck\varphi^\eps)$, 
we get that
\begin{align*}
&\norm{\Delta\widecheck\varphi^\eps}_{L^2(0,T; H)}^2+
\norm{F'(\widecheck\varphi^\eps)}_{L^2(0,T; H)}^2\\
&\hspace{2cm}\leq
\int_0^T\int_\OO\left[
|\Delta\widecheck\varphi^\eps(s)|^2
+2\Psi''(\widecheck\varphi^\eps(s))
|\nabla\widecheck\varphi^\eps(s)|^2+
|F'(\widecheck\varphi^\eps(s))|^2
\right]\,\d x\,\d s\\
&\hspace{2cm}=\int_0^T\int_\OO
\widecheck\mu^\eps(s)
(-\Delta\widecheck\varphi^\eps(s)+
F_\lambda'(\widecheck\varphi_\lambda^\eps(s)))
\,\d x\,\d s - 2\int_0^T\int_\OO
R''(\widecheck\varphi^\eps(s))|\nabla\widecheck\varphi^\eps(s)|^2
\,\d x\,\d s\\
&\hspace{2cm}\leq 
\frac12\norm{\Delta\widecheck\varphi^\eps}_{L^2(0,T; H)}^2+
\frac12\norm{F'(\widecheck\varphi^\eps)}_{L^2(0,T; H)}^2
+\norm{\widecheck\mu^\eps}_{L^2(0,T; H)}^2
+2C_R\norm{\widecheck\varphi^\eps}_{L^2(0,T; V_1)}^2,
\end{align*}
so that \eqref{est1_eps} and \eqref{est5_eps} yield
\begin{align}
    \label{est6_eps}
    \norm{\widecheck\varphi^\eps}_{
    L^p_\cP(\widecheck\Omega; L^2(0,T;V_2))}\leq C_{p},\\
    \label{est7_lam}
    \norm{F'(\widecheck\varphi^\eps)}_{
    L^p_\cP(\widecheck\Omega; L^2(0,T;H))}\leq C_{p}.
\end{align}
Finally, we adapt the computations of Subsubsection \ref{ssec:further_n}. At first, we observe that
\[
\|\widecheck \varphi^\eps\|_{L^p_\cP(\com;L^4(0,T;V_2))} \leq C_p,
\]
for all $p \in [2,q)$ by the very same computations. Then, we have
\[
\begin{split}
    \|\widecheck{G}^\eps(\widecheck \varphi^\eps(t))\|^2_{\cL^2(U,H)} & = \sum_{k \in \mathbb N} \|\nabla \widecheck \varphi^\eps(t) \cdot \widecheck{\b\sigma}^\eps_k(t)\|^2_H 
     \leq C \sum_{k \in \mathbb N} \|\nabla \widecheck \varphi^\eps(t) \|^2_{\b H}\|\widecheck{\b\sigma}_k(t)\|^2_{\W}
\end{split}
\]
and therefore, for $p \in [2,\, \min\{4,q\})$, we have
\[
\begin{split}
    \cE\left|\int_0^T\|\widecheck{G}^\eps(\widecheck \varphi^\eps(t))\|^p_{\cL^2(U,H)} \: \d t\right| & \leq C\cE\left|\int_0^T \|\nabla \widecheck \varphi^\eps(t) \|^p_{\b H} \left( \sum_{k \in \mathbb N} \|\widecheck{\b\sigma}_k(t)\|^2_{\W} \right)^\frac p2 \: \d t \right| \\
    & \leq C\cE\left[ \sup_{t \in [0,T]} \|\nabla \widecheck \varphi^\eps(t) \|^p_{\b H} \int_0^T  \|\widecheck{\b \Sigma}(t)\|^p_{\cL^2(U, \W)} \: \d t \right] \\
    & \leq C\cE\left[\sup_{t \in [0,T]} \|\nabla \widecheck \varphi^\eps(t) \|^{pr}_{\b H}\right]^\frac{1}{r}\cE\left[\left|\int_0^T  \|\widecheck{\b \Sigma}(t)\|^p_{\cL^2(U, \W)} \: \d t\right|^{\frac 1p \frac{pr}{r-1}}\right]^{p\frac{r-1}{pr}} \\
    & \leq C\|\widecheck \varphi^\eps\|_{L^{pr}_w(\com;L^\infty(0,T;V_1))}^p\|\widecheck{\b \Sigma}\|_{L^{\frac{pr}{r-1}}(\com;L^p(0,T;\cL^2(U,\W)))}^p
    \end{split}
\]
Indeed, the choice $p \leq 4$ ensures the second factor is well defined for all $r >1$, while $p < q$ ensures that there exists $r > 1$ such that $pr < q$, and thus the first factor is also well defined. The analogous higher-order estimate only holds in the space
\[
L^p(\com; L^2(0,T;\cL^2(U,V_1)))
\]
by the same token. Therefore, we conclude that for any $p \in [2,\,\min\{q,4\})$
\[
\|\widecheck G^{\eps}(\widecheck \varphi^\eps)\|_{L^p_\cP(\com;L^p(0,T;\cL^2(U,H))) \cap L^p_\cP(\com;L^2(0,T;\cL^2(U,V_1)))} \leq C_{p},
\]
and, still using \cite[Lemma 2.1]{fland-gat}, we get the bounds on the It\^{o} integrals
\[
\left\|\int_0^\cdot \widecheck G^\eps(\widecheck \varphi^\eps(s)) \: \d \widecheck W^\eps\right\|_{L^p_\cP(\com; W^{\alpha, p}(0,T;H))} \leq C_{p}
\]
for all $p \in [2,\,\min\{q,4\})$ and all $\alpha \in (0,\frac 12)$. As far as the It\^{o}--Stratonovich corrector is concerned, observe that
\[
\begin{split}
    \left| \langle \cLe(\widecheck \varphi^\eps(t)), \psi\rangle_{V_1^*, V_1} \right| & = \left| \sum_{k \in \mathbb N} \int_\OO [\widecheck {\b \sigma}^\eps_k \cdot \nabla  \widecheck \varphi^\eps(t)][\widecheck {\b \sigma}^\eps_k \cdot \nabla  \psi] \: \d x\right| 
     \leq \|\widecheck \varphi^\eps(t)\|_{V_1}\|\widecheck{\b \Sigma}(t)\|^2_{\cL^2(U,\W)}
\end{split}
\]
whenever $\|\psi\|_{V_1} \leq 1$. This implies
\[
\begin{split}
    \cE \left| \int_0^T \|\cLe(\widecheck \varphi^\eps(t))\|_{V_1^*}^2 \: \d t\right|^\frac p2 &  \leq \cE\left[ \sup_{t \in [0,T]} \|\widecheck \varphi^\eps(t)\|_{V_1}^p \left|\int_0^T \|\widecheck{\b \Sigma}(t)\|^4_{\cL^2(U,\W)}\right|^\frac p2\right] \\
    & \leq C\cE\left[ \sup_{t \in [0,T]} \|\widecheck \varphi^\eps(t)\|_{V_1}^{pr} \right]^\frac 1r \cE \left[ \left|\int_0^T \|\widecheck{\b \Sigma}(t)\|^4_{\cL^2(U,\W)}\right|^{\frac p2\frac{r}{r-1}}\right]^\frac{r-1}{r}
\end{split}
\]
and therefore
\[
\|\cLe(\widecheck\varphi^\eps)\|_{L^p_\cP(\com;L^2(0,T; V_1^*))} \leq C_{p}
\]
for any $p \in [2, q)$. Finally, by comparison we also deduce
\[
\|\widecheck \varphi^\eps\|_{L^p_\cP(\com;W^{\alpha, p}(0,T;V_1^*))} \leq C_{p}
\]
for all $\alpha \in (0,\frac 12)$ and $p \in [2, \min\{4, q\})$.

\subsection{Passage to the limit as $\boldsymbol{\eps \to 0^+}$}
\label{ssec:limit_eps}
A martingale solution to problem \eqref{eq:ch_ito} is finally retrieved by a final compactness argument. Once again, the argument is similar to the ones of Subsections \ref{ssec:limit_as_n} and \ref{ssec:limit_lam}, and therefore we shall omit some details, for the sake of brevity. Without relabeling, any passage to the limit as $\eps \to 0^+$ is performed along an arbitrary vanishing sequence $\{\eps_k\}_{k \in \enne} \subset (0,1)$.  For this subsection, we define $\ell = \min\{q, 4\}.$

\subsubsection{Tightness of the laws of regularized solutions}
The analogue of the previous tightness lemmas is given hereafter. Although the path space is different, the proof is entirely analogous (compare with Lemmas \ref{lem:tight} and \ref{lem:tight_lam}) and therefore omitted.
\begin{lem} \label{lem:tight_eps}
    The family of laws of 
    \[
    \{\widecheck\varphi^\eps,\, \widecheck G^\eps(\widecheck\varphi^\eps) \cdot \widecheck W^\eps,\, \widecheck W^\eps,\, \widecheck{\b \Sigma}^\eps\}_{\eps \in (0,1)}
    \]
    is tight in the space
    \[
    \Xi \times C^0([0,T]; V^*_1) \times C^0([0,T]; U_0) \times L^2(0,T;\cL^2(U,\b H)))
    \]
    where $\Xi = C^0([0,T];H) \cap L^2(0,T;V_1)$.
\end{lem} \noindent
The Prokhorov and Skorokhod theorems entail then that there exists a probability space $(\hom,\, \hF,\, \hP)$ and a sequence of random variables $\{Z^\eps :(\hom, \hF) \to (\com, \cf)\}_{\eps \in (0,1)}$ such that the law of $Z^\eps$ equals $\cp$ for all $\eps \in (0,1)$ and, up to subsequences, all of the convergences 
\begin{align*}
    \widehat{\varphi}^\eps := \widecheck \varphi^\eps \circ Z^\eps \to \widehat{\varphi} & \qquad \text{in }L^p(\com; C^0([0,T];H) \cap L^2(0,T;V_1)); \\
    \widehat{\varphi}^\eps \overset{*}{\rightharpoonup} \widehat{\varphi} & \qquad \text{in }L^{r}_w(\com; L^\infty(0,T;V_1)) \cap L^{p}(\tom; W^{\alpha, p}(0,T;V_1^*)); \\
    \widehat{\varphi}^\eps \rightharpoonup \widehat{\varphi} & \qquad \text{in }L^{r}(\com; L^4(0,T;V_2)); \\
    \widehat{\mu}^\eps := \widecheck {\mu}^\eps \circ Z^\eps \rightharpoonup\widehat{\mu} & \qquad \text{in }L^{r}(\com; L^2(0,T;V_1)); \\
    \widehat{I}^\eps := (\widecheck G^\eps(\widecheck \varphi^\eps) \cdot \widecheck W^\eps) \circ Z^\eps  \to \widehat{I} & \qquad \text{in }L^p(\com; C^0([0,T];V^*)); \\
    \widehat{W}^\eps := \widecheck W^\eps \circ Z^\eps \to \widehat W & \qquad \text{in }L^p(\com;C^0([0,T];U_0)); \\
    \widehat{\b \Sigma}^\eps := \widecheck{\b \Sigma}^\eps \circ Z^\eps \to \widehat{\b \Sigma} & \qquad \text{in }L^r(\com; L^2(0,T;\cL^2(U, \b H))),
\end{align*}
hold in the limit $\eps \to 0^+$ for all $p < \ell$ and $r < q$ and for some limiting processes satisfying
\begin{align*}
    \widehat{\varphi} & \in L^r_\cP(\com; C^0([0,T];H)) \cap L^r_w(\tom; L^\infty(0,T;V_1)) \cap L^r_\cP(\com; L^4(0,T;V_2)); \\
    \widehat{\mu} & \in L^r_\cP(\com; L^2(0,T;V_1)); \\
    \widehat{I} & \in L^p_\cP(\com; C^0([0,T]; H)); \\
    \widehat{W} & \in L^r_\cP(\com; C^0([0,T];U));\\
    \widehat{\b \Sigma} & \in L^q_\cP(\com; L^2(0,T; \cL^2(U,H)))
\end{align*}
for all $p < \ell$ and $r < q$.

\subsubsection{Identification of the limit solution}
The identification of the limit solution follows by the same arguments of Subsubsection \ref{ssec:id_lam}. The procedure yields a martingale solution
\[
\left(\widehat{\Omega},\,
\widehat{\cF}, \,
(\widehat{\cF}_{t})_{t\in[0,T]},\,
\widehat{\P}, \,
\widehat{W}, \,
\widehat{\b\Sigma},\,
\widehat{\varphi}
\right)
\]
of \eqref{eq:ch}, in the sense of Definition \ref{def:sol-nd}.

\subsection{Uniqueness with constant mobility}
\label{ssec:uniq}
Finally, we address the problem of proving uniqueness of martingale solutions for problem \eqref{eq:ch}. As customary, as a consequence of the Yamada--Watanabe theorem (see \cite[Theorem E.0.8]{LiuRo}), this yields existence and uniqueness of probabilistically-strong solutions, in the sense of Definition \ref{def:sol-nd}-\ref{def:strong-nd}. For the remainder of this subsection, we shall assume that the mobility function $m$ is constant and positive, i.e., that
\[
m(r) = m > 0 \qquad \forall \: r \in \mathbb R.
\]
Let $\varphi_{01}$ and $\varphi_{02}$ comply with Definition \ref{def:sol-nd}, i.e., assume that \eqref{init} holds, and be such that $\overline{\varphi_{01}} = \overline{\varphi_{02}}$. Let $(\hom, \hF, (\hF_t)_t, \hP)$ be a filtered probability space, let $\widehat W$ be an $(\hF_t)_t$-adapted cylidrical Wiener process and assume that, for $i \in \{1,2\}$, the family
\[
\left( (\hom,\, \hF,\, (\hF_t)_t, \, \hP),\,\widehat W, \widehat{\b{ \Sigma}}_i,\, \hphi_i \right)
\]
is a martingale solution of \eqref{eq:ch}, where $\widehat{\b \Sigma}_1$ and $\widehat{\b \Sigma}_2$ are two transport fields complying with Assumption \ref{hyp:noise}. This amounts to say that, formally,
\[
    \begin{cases}
        \d \hphi_i - m\Delta\hmu_i \,\d t 
        +\kappa\hL^i\hphi_i\,\d t
        = \nabla\hphi_i\cdot\widehat{\b{\Sigma}}_i\,\d \widehat W
        \quad&\text{in } (0,T)\times\OO, \\
        \hmu_i = -\Delta \hphi_i + F'(\hphi_i) \quad&\text{in }(0,T)\times\OO, \\
        \partial_{\b n}\hphi_i=\partial_{\b n}\hmu_i=0
        \quad&\text{on }(0,T)\times\partial\OO, \\
        \hphi_i(\cdot \:, 0) = \varphi_{0i} \quad&\text{in } \OO,
    \end{cases} \qquad i \in \{1,2\},
\]
where we regard the deterministic initial conditions as degenerate random variables on $\hom$. Here, setting
\[
\widehat {\b \Sigma}_iu_k = \widehat{\b \sigma}_k^{i}
\]
for $i \in \{1,2\}$, the It\^{o}--Stratonovich correctors are defined as maps $\hL^i : \hom \times [0,T] \times V_1 \to V_1^*$ and are such that
\[
\langle \hL^i(\omega, t, \psi_1), \psi_2 \rangle_{V_1^*, V_1} = \dfrac{1}{2} \sum_{k \in \mathbb N} \int_\OO [\widehat{\b \sigma}^i_k(\omega, t, x) \cdot \nabla \psi_1(x)][\widehat{\b \sigma}^i_k(\omega, t, x) \cdot \nabla \psi_2(x)] \: \d x
\]
for $i \in \{1,2\}$ and all $\omega \in \hom$, $t \in [0,T]$ and $\psi_1,\,\psi_2 \in V_1$. Introducing the difference variables 
\[
\hphi := \hphi_1-\hphi_2, \qquad \hmu:= \hmu_1-\hmu_2, \qquad \hphi_0 = \varphi_{01}-\varphi_{02}, \qquad \widehat {\b \Sigma }:= \widehat {\b \Sigma}_1 - \widehat {\b \Sigma}_2,
\]
and setting
\[
\widehat {\b \Sigma}u_k := \widehat{\b \sigma}_k = \widehat{\b \sigma}_k^{1} - \widehat{\b \sigma}_k^{2},
\]
for all $k \in \mathbb N$, after minor algebraic manipulations we deduce the stochastic differential system
\[
\begin{cases}
        \d \hphi - m\Delta\hmu \,\d t 
        +\kappa\hL^1\hphi\,\d t +\kappa[\hL^1-\hL^2]\hphi_2\,\d t
        = \nabla\hphi\cdot\widehat{\b{\Sigma}}_1+\nabla \hphi_2\cdot \widehat{\b \Sigma} \,\d \widehat W
        \quad&\text{in } (0,T)\times\OO, \\
        \hmu = -\Delta \hphi + F'(\hphi_1) - F'(\hphi_2) \quad&\text{in }(0,T)\times\OO, \\
        \partial_{\b n}\hphi=\partial_{\b n}\hmu=0
        \quad&\text{on }(0,T)\times\partial\OO, \\
        \hphi(\cdot \:, 0) = \hphi_{0} \quad&\text{in } \OO.
    \end{cases}
\]

\subsubsection{First estimate}
Given the pathwise conservative structure of the system, it is straightforward to observe that
\[
\overline{\hphi(t)} = \overline{\hphi_0} = 0, 
\]
for all $t \geq 0$ and $\hP$-almost surely. Therefore, we may apply the It\^{o} lemma to the functional
\[
v \mapsto \dfrac 12 \|\nabla \mathcal Nv\|^2_{\b H},
\]
which is well defined for every $v \in V_{1,0}$, evaluated at $v = \hphi(\t)$. This results in 
\begin{multline} \label{eq:unique1}
    \dfrac 12 \|\nabla \mathcal N\hphi(\t)\|^2_{\b H} + m\int_0^{\t} \int_\OO \hmu(s)\hphi(s)\: \d x\, \d s  +\kappa\int_0^{\t} \left\langle \hL^1 \hphi + [\hL^1-\hL^2]\hphi_2, \,\mathcal N \hphi(s)\right\rangle_{V_{1,0}^*, V_{1,0}} \: \d s \\ = \dfrac 12 \|\nabla \mathcal N\hphi_0\|^2_{\b H} + \dfrac{1}{2}\int_0^{\t} \sum_{k=0}^{+\infty} \|\nabla \mathcal N [\nabla \hphi(s) \cdot \widehat{\b \sigma}_k^1(s) + \nabla \hphi_2(s) \cdot \widehat{\b \sigma}_k(s)]\|^2_{\b H} \; \d s \\ + \int_0^{\t} \sum_{k=0}^{+\infty} (\mathcal N\hphi(s), [\nabla \hphi(s) \cdot \widehat{\b \sigma}_k^1(s) + \nabla \hphi_2(s) \cdot \widehat{\b \sigma}_k(s)] )_H\: \d \widehat W(s).
\end{multline}
Let us now address each term in \eqref{eq:unique1} separately. In the following, the parameter $\delta > 0$ can be chosen arbitrarily small. On the left hand side, employing the mean value theorem and Assumption \ref{hyp:potential} yields
\begin{multline} \label{eq:unique11}
    m\int_0^{\t} \int_\OO \hmu(s)\hphi(s)\: \d x\, \d s= m\int_0^{\t} \|\nabla \hphi(s)\|^2_{\b H} \: \d s +\int_0^{\t} \int_\OO [F'(\hphi_1(s))-F'(\hphi_2(s))]\hphi(s) \: \d x \, \d s\\
    \geq m\int_0^{\t} \|\nabla \hphi(s)\|^2_{\b H} \: \d s + \int_0^{\t}\int_\OO \int_0^1 R''(\theta\hphi_2(s)+(1-\theta)\hphi_1(s))\hphi^2(s) \: \d\theta\, \d x\,\d s,
\end{multline}
while by interpolation we have
\begin{multline} \label{eq:unique12}
         \left| \int_0^{\t} \left\langle \hL^1 \hphi, \,\mathcal N \hphi(s)\right\rangle_{V_{1,0}^*, V_{1,0}} \: \d s \right| \leq C_\infty^2\left|\int_0^{\t} \|\widehat{\b \Sigma}_1(s)\|^2_{\cL^2(U, \W)}\|\nabla\hphi(s)\|_{\b H}\|\nabla\mathcal N \hphi(s)\|_{\b H} \: \d s\right| \\
         \leq \delta \int_0^{\t} \|\nabla\hphi(s)\|^2_{\b H} \: \d s + C_\delta \int_0^{\t} \|\widehat{\b \Sigma}_1(s)\|^4_{\cL^2(U, \W)}\|\nabla\mathcal N \hphi(s)\|_{\b H}^2 \: \d s
\end{multline}
and, analogously
\begin{equation} \label{eq:unique13}
    \begin{split}
        & \luca{\kappa}\left| \int_0^{\t} \left\langle [\hL^1 -\hL^2]\hphi_2, \,\mathcal N \hphi(s)\right\rangle_{V_{1,0}^*, V_{1,0}} \: \d s \right| \\
        & \hspace{1cm}\leq \kappa C_\infty^2\left|\int_0^{\t} \|\widehat{\b \Sigma}_1(s)\|_{\cL^2(U, \W)}\|\widehat{\b \Sigma}(s)\|_{\cL^2(U, \W)}\|\nabla\hphi_2(s)\|_{\b H}\|\nabla\mathcal N \hphi(s)\|_{\b H} \: \d s\right| \\
        & \hspace{3cm}+ \kappa C_\infty^2\left|\int_0^{\t} \|\widehat{\b \Sigma}_2(s)\|_{\cL^2(U, \W)}\|\widehat{\b \Sigma}(s)\|_{\cL^2(U, \W)}\|\nabla\hphi_2(s)\|_{\b H}\|\nabla\mathcal N \hphi(s)\|_{\b H} \: \d s\right| \\
        & \hspace{1cm} \leq \dfrac 12\int_0^{\t} \|\nabla \mathcal N \hphi(s)\|_{\cL^2(U, \W)}^2 \: \d s \\
        & \hspace{3cm}+ \luca{C\kappa} \int_0^{\t} \left[\|\widehat{\b \Sigma}_1(s)\|^2_{\cL^2(U, \W)}+ \|\widehat{\b \Sigma}_2(s)\|^2_{\cL^2(U, \W)} \right]\|\nabla\hphi_2(s)\|^2_{\b H}\|\widehat{\b \Sigma}(s)\|^2_{\W} \: \d s.
    \end{split}
\end{equation}
Next, the It\^{o} trace term is controlled by
\begin{equation} \label{eq:unique14}
\begin{split}
     & \dfrac{1}{2}\int_0^{\t} \sum_{k=0}^{+\infty} \|\nabla \mathcal N [\nabla \hphi(s) \cdot \widehat{\b \sigma}_k^1(s) + \nabla \hphi_2(s) \cdot \widehat{\b \sigma}_k(s)]\|^2_{\b H} \; \d s \\
     & \hspace{0.4cm} \leq \int_0^{\t} 
     \left[\|\hphi(s)\|^2_H\|\widehat{\b \Sigma}_1(s)\|_{\cL^2(U, \W)}^2+\luca{
     \|\hphi_2(s)\|^2_{L^\infty(\OO)}
     \|\widehat{\b \Sigma}(s)\|_{\cL^2(U, \bH)}^2}\right] \: \d s \\
     & \hspace{0.4cm} \leq \delta \int_0^{\t} \|\nabla\hphi(s)\|^2_{\b H} \: \d s + C_\delta\int_0^{\t} \|\widehat{\b \Sigma}_1(s)\|_{\cL^2(U, \W)}^4\|\nabla \mathcal N  \hphi(s)\|^2_\b H+
     \luca{\int_0^{\t}\|\widehat{\b \Sigma}(s)\|_{\cL^2(U, \bH)}^2 \: \d s}.
\end{split}
\end{equation}
By choosing $\delta$ sufficiently small, collecting \eqref{eq:unique11}-\eqref{eq:unique14} in \eqref{eq:unique1} at multiplying the result by 2 yields
\begin{multline} \label{eq:unique2}
    \|\nabla \mathcal N\hphi(\t)\|^2_{\b H} + m\int_0^{\t} \|\nabla\hphi(s)\|^2_{\b H} \: \d s   \leq \|\nabla \mathcal N\hphi_0\|^2_{\b H} + \int_0^{\t} \widehat{\mathcal Y}(s)\|\nabla\mathcal N \hphi(s)\|^2_{\b H } \: \d s + \int_0^{\t} \widehat{\mathcal Z}(s) \: \d s\\ + 2\int_0^{\t} \sum_{k=0}^{+\infty} (\mathcal N\hphi(s), [\nabla \hphi(s) \cdot \widehat{\b \sigma}_k^1(s) + \nabla \hphi_2(s) \cdot \widehat{\b \sigma}_k(s)] )_H\: \d \widehat W(s),
\end{multline}
with the processes $\widehat{\mathcal Y}: \hom \times [0,T] \to \mathbb R$ and $\widehat{\mathcal Z}: \hom \times [0,T] \to \mathbb R$ defined by
\begin{align*}
    \widehat{\mathcal Y}(s) & = C\left[1+\|\widehat{\b \Sigma}_1(s)\|^4_{\cL^2(U, \W)}\right], \\
    \widehat{\mathcal Z} (s) & = 
    \luca{
    C\kappa \left[ 1+ \|\widehat{\b \Sigma}_1(s)\|^2_{\cL^2(U, \W)}+ \|\widehat{\b \Sigma}_2(s)\|^2_{\cL^2(U, \W)} \right]\|\nabla\hphi_2(s)\|^2_{\b H}\|\widehat{\b \Sigma}(s)\|_{\cL^2(U, \W)}^2+
    C\|\widehat{\b \Sigma}(s)\|_{\cL^2(U, \bH)}^2},
\end{align*}
for a suitably large constant $C > 0$. 

\subsubsection{A weighted estimate}
Since $\|\widehat {\b \Sigma}_1\|_{L^4(0,T;\cL^2(U,\W))}$ possesses exponential moments of any order by Assumption \ref{hyp:noise}, it is possible to derive a continuous dependence estimate through a weighted It\^{o} formula. Indeed, for a fixed $\lambda > 0$ to be specified later, define
\[
\Lambda: [0,T] \to (0,1]   \qquad \Lambda(t) = \exp\left( -\lambda\int_0^t\widehat{\mathcal Y}(s) \: \d s\right).
\]
Applying the It\^{o} formula to the functional
\[
v \mapsto \Lambda(t)\|\nabla \mathcal Nv\|^2_{\b H},
\]
%still evaluated at $v = \hphi(\t)$, 
%we obtain by the product rule
%\begin{multline} \label{eq:unique3}
%    \Lambda(t)\|\nabla \mathcal N\hphi(\t)\|^2_{\b H} + 2\int_0^{\t} \int_\OO \Lambda(s)\hmu(s)\hphi(s)\: \d x\, \d s  +\lambda\int_0^t\Lambda(s)\widehat{\mathcal{Y}}(s)\|\nabla \mathcal N\hphi(s)\|^2_{\b H} \: \d s\\ +2\kappa\int_0^{\t} \Lambda(s)\left\langle \hL^1 \hphi + [\hL^1-\hL^2]\hphi_2, \,\mathcal N \hphi(s)\right\rangle_{V_{1,0}^*, V_{1,0}} \: \d s \\ = \|\nabla \mathcal N\hphi_0\|^2_{\b H} + \int_0^{\t} \Lambda(s)\sum_{k=0}^{+\infty} \|\nabla \mathcal N [\nabla \hphi(s) \cdot \widehat{\b \sigma}_k^1(s) + \nabla \hphi_2(s) \cdot \widehat{\b \sigma}_k(s)]\|^2_{\b H} \; \d s \\ + 2\int_0^{\t} \Lambda(s) \sum_{k=0}^{+\infty} (\mathcal N\hphi(s), [\nabla \hphi(s) \cdot \widehat{\b \sigma}_k^1(s) + \nabla \hphi_2(s) \cdot \widehat{\b \sigma}_k(s)] )_H\: \d \widehat W(s).
%\end{multline}
%Eventually,
by iterating the computations in \eqref{eq:unique11}-\eqref{eq:unique14} and %starting from \eqref{eq:unique3}, 
by exploiting also the fact that $|\Lambda| \leq 1$ in the deterministic integrals, it is straightforward to obtain a weighted version of \eqref{eq:unique2} with an additional dissipation term, namely
\begin{multline} \label{eq:unique4}
    \Lambda(t)\|\nabla \mathcal N\hphi(\t)\|^2_{\b H} + m\int_0^{\t} \Lambda(s)\|\nabla\hphi(s)\|^2_{\b H} \: \d s  + (\lambda-1)\int_0^t\Lambda(s)\widehat{\mathcal{Y}}(s)\|\nabla \mathcal N\hphi(s)\|^2_{\b H} \: \d s \\ \leq \|\nabla \mathcal N\hphi_0\|^2_{\b H} + \int_0^{\t} \widehat{\mathcal Z}(s) \: \d s + 2\int_0^{\t} \Lambda(s) \sum_{k=0}^{+\infty} (\mathcal N\hphi(s), [\nabla \hphi(s) \cdot \widehat{\b \sigma}_k^1(s) + \nabla \hphi_2(s) \cdot \widehat{\b \sigma}_k(s)] )_H\: \d \widehat W(s),
\end{multline}
For any $p \in [2,q)$, raising \eqref{eq:unique4} to the $\frac p2$-power, taking supremums in time and $\hP$-expectations yields
\begin{multline}\label{eq:unique41}
    \hE \supp \Lambda^\frac p2(\tau)\|\nabla \mathcal N\hphi(\tau)\|^p_{\b H} + m^\frac p2\hE\left|\int_0^{\t} \Lambda(s)\|\nabla\hphi(s)\|^2_{\b H} \: \d s\right|^\frac p2 + (\lambda-1)^\frac  p2\hE\left|\int_0^t\Lambda(s)\widehat{\mathcal{Y}}(s)\|\nabla \mathcal N\hphi(s)\|^2_{\b H} \: \d s\right|^\frac p2  \\ \leq C\Bigg[ \|\nabla \mathcal N\hphi_0\|^p_{\b H} + \hE\left| \int_0^t \widehat{\mathcal Z}(s) \: \d s\right|^\frac p2 \\ + \hE \supp \left|  \int_0^{\tau} \Lambda(s) \sum_{k=0}^{+\infty} (\mathcal N\hphi(s), [\nabla \hphi(s) \cdot \widehat{\b \sigma}_k^1(s) + \nabla \hphi_2(s) \cdot \widehat{\b \sigma}_k(s)] )_H\: \d \widehat W(s)\right|^\frac p2\Bigg]
\end{multline}
Finally, we are left to handle the stochastic integral. The Burkholder--Davis--Gundy inequality yields, after an integration by parts,
\begin{equation} \label{eq:unique42}
\begin{split}
    & \hE \supp \left| \int_0^{\tau} \Lambda(s)\sum_{k=0}^{+\infty} (\mathcal N\hphi(s), [\nabla \hphi(s) \cdot \widehat{\b \sigma}_k^1(s) + \nabla \hphi_2(s) \cdot \widehat{\b \sigma}_k(s)] )_H\: \d \widehat W(s)\right|^\frac p2 \\
    & \hspace{0.3cm} \leq C\hE \left|  \int_0^{t} \Lambda^2(s)\|\nabla\mathcal N\hphi(s)\|^2_{\b H} \| \hphi(s) \widehat{\b \Sigma}_1(s) + \hphi_2(s)  \widehat{\b \Sigma}(s)\|_{\cL^2(U, \b H)}^2 \: \d s\right|^\frac p4 \\
    & \hspace{0.3cm} \leq C\hE \left[  \supp \Lambda^\frac p4(\tau)\|\nabla\mathcal N\hphi(\tau)\|^\frac p2_{\b H}\left|\int_0^{t} \Lambda(s)\left[ \| \hphi(s)\|^2_H \|\widehat{\b \Sigma}_1(s)\|_{\cL^2(U,\W)}^2 + \luca{\|\widehat{\b \Sigma}(s)\|_{\cL^2(U, \bH)}^2}\right] \d s\right|^\frac p4 \right]\\
    & \hspace{0.3cm} \leq \dfrac{1}{2} \hE \supp \Lambda^\frac p2(\tau)\|\nabla\mathcal N\hphi(\tau)\|^p_{\b H} + \frac m2 \hE\left|\int_0^{\t} \Lambda(s)\|\nabla\hphi(s)\|^2_{\b H} \right|^\frac p2 \\
    & \hspace{3cm} + C \hE\left|\int_0^{t} \Lambda(s) \| \nabla \mathcal N \hphi(s)\|^2_H \|\widehat{\b \Sigma}_1(s)\|_{\cL^2(U,\W)}^4 + 
    \luca{\|\widehat{\b \Sigma}(s)\|_{\cL^2(U, \bH)}^2} \: \d s\right|^\frac p2 \\
    & \hspace{0.3cm} \leq \dfrac{1}{2} \hE \supp \Lambda^\frac p2(\tau)\|\nabla\mathcal N\hphi(\tau)\|^p_{\b H} + \frac{m^\frac p2}{2} \hE\left|\int_0^{\t} \Lambda(s)\|\nabla\hphi(s)\|^2_{\b H} \right|^\frac p2 \\
    & \hspace{3cm} + C \hE\left|\int_0^{t} \Lambda(s) \widehat{\mathcal Y}(s) \| \nabla \mathcal N \hphi(s)\|^2_H  \: \d s \right|^\frac p2 +
    C \hE \left| \int_0^t\luca{ \|\widehat{\b \Sigma}(s)\|_{\cL^2(U, \bH)}^2} \: \d s\right|^\frac p2.
\end{split}
\end{equation}
Upon choosing $\lambda > 0$ sufficiently large, in particular so that $(\lambda -1)^\frac p2> C$ with $C >0$ the constant appearing in the last line of \eqref{eq:unique42}, and suitably redefining the process $\widehat{\mathcal Z}$, it follows from \eqref{eq:unique41} and \eqref{eq:unique42} that
\begin{multline} \label{eq:unique43}
    \hE \supp \Lambda^\frac p2(\tau)\|\nabla \mathcal N\hphi(\tau)\|^p_{\b H} + \hE\left|\int_0^{\t} \Lambda(s)\|\nabla\hphi(s)\|^2_{\b H} \: \d s\right|^\frac p2 + \hE\left|\int_0^t\Lambda(s)\widehat{\mathcal{Y}}(s)\|\nabla \mathcal N\hphi(s)\|^2_{\b H} \: \d s\right|^\frac p2  \\ \leq C\left[ \|\nabla \mathcal N\hphi_0\|^p_{\b H} + \hE\left| \int_0^t \widehat{\mathcal Z}(s) \: \d s\right|^\frac p2 \right],
\end{multline}
for all $p < q$, 
%where
%\[
%\widehat{\mathcal Z} (s) = C\left[ 1+ \left(1+ \|\widehat{\b \Sigma}_1(s)\|^2_{\cL^2(U, \W)}+ \|\widehat{\b \Sigma}_2(s)\|^2_{\cL^2(U, \W)} \right)\|\nabla\hphi_2(s)\|^2_{\b H} \right]\|\widehat{\b \Sigma}\|_{\cL^2(U, \W)}^2
%\]
for a possibly different constant $C > 0$. Observe that $\|\widehat{\mathcal Z}\|_{L^1(0,T)}$ is a well defined random variable in $L^r(\hom)$ for all $r < \frac q2$, thanks to Assumption \ref{hyp:noise}.

\subsubsection{Conclusion of the argument}
In order to conclude the proof, we are left with removing the weight introduced in the previous computations. Let $1<r<q$ and $s > 1$ be such that 
\[
\dfrac{rs}{s-1} < q.
\]
By monotonicity of the exponential function and the H\"{o}lder inequality, as well as \eqref{eq:unique43} with $p = \frac{rs}{s-1}$,
\[
\begin{split}
    \hE \supp \|\nabla \mathcal N\hphi(\tau)\|^r_{\b H} & = \hE \sup_{t \in [0,T]} \Lambda^\frac r2(t)\Lambda^{-\frac r2}(t)\|\nabla \mathcal N\hphi(t)\|^r_{\b H} \\
    & \leq \hE \left[ \Lambda^{-\frac r2}(T) \sup_{t \in [0,T]} \Lambda^\frac r2(t)\|\nabla \mathcal N\hphi(t)\|^r_{\b H} \right] \\
    & \leq C \hE\left[ \Lambda^{-\frac {rs}2}(T) \right]^\frac{1}{s} \hE\left[ \sup_{t \in [0,T]} \Lambda^{\frac{rs}{2(s-1)}}(t)\|\nabla \mathcal N\hphi(t)\|^\frac{rs}{s-1}_{\b H} \right]^\frac{s-1}{s} \\
    & \leq C\hE\left[ \exp\left( \dfrac{\lambda rs}{2}\int_0^T \widehat{\mathcal Y}(t) \: \d t\right)\right]\left[ \|\nabla \mathcal N\hphi_0\|^r_{\b H} + \hE\left| \int_0^t \widehat{\mathcal Z}(\tau) \: \d \tau\right|^\frac r2 \right].
\end{split}
\]
As the first factor is finite for all values of $\lambda,\, r$ and $s$ thanks to Assumption \ref{hyp:noise}, the claim follows. Pathwise uniqueness follows standardly by taking $\varphi_{01}=\varphi_{02}$ and $\widehat{\b \Sigma}_1=\widehat{\b \Sigma}_2$, that entails $\widehat{\mathcal Z}\equiv 0$.

\section{Proof of Theorem \ref{th:d}}
\label{sec:d}
This section is devoted to the proof of Theorem~\ref{th:d}, namely, to analyze \eqref{eq:ch} in the case of degenerate mobilities. In order to prove existence of solutions, we regularize the
degenerate mobility function through a parameter $\delta>0$ to
fall in the non-degenerate setting analyzed in Section~\ref{sec:nd}. This allows us to apply Theorem~\ref{th:nd} at the approximation level, obtaining a sequence of solution for the non-degenerate problem indexed by $\delta$. Uniform estimates for such family of solutions are shown with respect to $\delta$ and a passage to
the limit as $\delta \to 0^+$ concludes the proof.

\subsection{The approximated problem}
\label{ssec:approx_deg}

We start this subsection with presenting the approximation scheme of the mobility function so that the assumptions of Theorem~\ref{th:nd} are fulfilled. For all $\delta\in(0,1)$, we define the approximated mobility
$$
m_\delta\colon[-1,1]\to\erre^+,\qquad m_\delta(r):=m(r)+\delta, \quad r\in [-1,1].
$$
It is easy to see that $m_\delta$ satisfies Assumption \ref{hyp:mobility_nd}, with constants $c_{m_\delta}=\delta$ and $C_{m_\delta}=C_m+1$, for every
$\delta\in(0,1)$. For simplicity, we continuously extend $m$
and $m_\delta$ to the whole real line $\erre$ by setting 
\begin{align*}
	&m(r):=\begin{cases}
		0 &r\in\erre\setminus[-1,1],\\
		m(r)\ &r\in(-1,1),
	\end{cases}\qquad m_\delta(r):=\begin{cases}
	\delta\ &r\in\erre\setminus[-1,1],\\
	m_\delta(r)\ &r\in(-1,1).
	\end{cases}.
\end{align*}
Analogously, we define the second-order family of primitives $M_\delta\colon\erre\to\erre_+$ as
\[
M_\delta(r) := \int_0^r\!\int_0^s \frac{1}{m_\delta(t)}\,\d t\,\d s,
\qquad r\in\erre,
\]
so that by Assumption \ref{hyp:mobility_d} it holds that
\begin{equation}
	\label{est_M}
	M_\delta(r) \leq M(r) \quad \forall\,r\in(-1,1),
	\qquad
	M_\delta(r) \geq \frac{|r|^2}{2(1+C_m)}, \quad \forall\,r\in\erre.
\end{equation}
Observe that, by construction, $M_\delta \in C^2(\mathbb R)$ and $m_\delta M''_\delta \equiv 1$ for all $\delta \in (0,1)$. In this framework, it is straightforward to see that Assumptions~\ref{hyp:potential}, \ref{hyp:mobility_nd}, and \ref{hyp:noise} are fulfilled and Theorem~\ref{th:nd} can be applied and ensures, for
every $\delta\in(0,1)$, the existence of a probabilistically-weak
solution
\[
\bigl(\widetilde\Omega,\,\widetilde\cF,\,(\widetilde\cF_{\delta,t})_{t\in[0,T]},\,
\widetilde\P,\,\widetilde W_\delta,\,\widetilde {\b\Sigma}_\delta,\,\widetilde\varphi_\delta\bigr),
\]
in the sense of Definition~\ref{def:sol-nd}, with
$\widetilde{\b\Sigma}_\delta \laweq \b\Sigma$ for all $\delta > 0$, to the approximated problem
	\begin{numcases}{}
		\label{eq1_app_del}
		\d\widetilde\varphi_\delta
	- \div\bigl[m_\delta(\widetilde\varphi_\delta)\nabla\widetilde\mu_\delta\bigr]\,\d t
	+ \kappa\hLd\widetilde\varphi_\delta\,\d t
	= \nabla\widetilde\varphi_\delta \cdot \widetilde{\b\Sigma}_\delta\,\d \widetilde W_\delta
	& \text{in }$(0,T)\times\OO,$ \\
	\label{eq2_app_del}
	\widetilde \mu_\delta = -\Delta\widetilde \varphi_\delta + F'(\widetilde \varphi_\delta)
	& \text{in }$(0,T)\times\OO,$ \\
	\label{eq3_app_del}
	\partial_{\b n}\widetilde \varphi_\delta
	= m_\delta(\widetilde \varphi_\delta)\,\partial_{\b n}\widetilde \mu_\delta = 0
	& \text{in }$(0,T)\times\partial\OO,$ \\
	\label{eq4_app_del}
	\widetilde \varphi_\delta(0) = \widetilde \varphi_0
	& \text{in }$\OO.$
	\end{numcases}

\subsection{Uniform estimates with respect to $\boldsymbol{\delta}$}
\label{ssec:est_deg}
In this subsection, we derive the uniform estimates needed to pass to the limit in the sequence of non-degenerate solutions as $\delta$ converges to $0$, retrieving a solution to problem \eqref{eq:ch} in the case of degenerate mobility. Let us stress that, here and in the following, the symbol $C$ is reserved for a positive constant depending on the structural parameters of the problem, but not on $\delta$. As in Section~\ref{sec:nd}, its value may change within the same argument without relabeling and relevant dependencies will be highlighted if necessary.

\subsubsection{The mobility estimate}
First, we apply the It\^o formula to the functional $\norm{M_\delta(\cdot)}_{L^1(\OO)}$, applied to the stochastic process $\widetilde \varphi_\delta(t)$. This yields
\begin{multline*}
  \norm{M_\delta(\widetilde \varphi_\delta(t))}_{L^1(\OO)} + 
  \int_0^t\left[
  \norm{\Delta\widetilde \varphi_\delta(s)}_H^2+
  \norm{|\Psi''(\widetilde \varphi_\delta(s))|^{\frac12}
  \nabla\widetilde \varphi_\delta(s)}_\bH^2
  \right]\,\d s\\
  =\norm{M_\delta( \varphi_0)}_{L^1(\OO)}
  -\int_0^t\int_\OO R''(\widetilde \varphi_\delta(s))
  |\nabla\widetilde \varphi_\delta(s)|^2\,\d x\,\d s
  +\int_0^t\left(M_\delta'(\widetilde \varphi_\delta(s)), \nabla\widetilde \varphi_\delta(s)\cdot\widetilde{\b\Sigma}_\delta(s)\,\d \widetilde W_\delta(s)\right)_H\\
  +\frac{1-\kappa}2\int_0^t\sum_{k\in\enne}\int_\OO
  M_\delta''(\widetilde \varphi_\delta(s))
  |\nabla\widetilde \varphi_\delta(s)\cdot\widetilde{\b\sigma}_{\delta,k}(s)|^2\,\d x\,\d s,
\end{multline*}
in the same fashion of \eqref{ito:M_n}. As the stochastic integral above still vanishes, thanks to Assumption \ref{hyp:potential}, the first inequality in \eqref{est_M}, and Assumption~\ref{hyp:noise_d}, we get the bound
\begin{multline*}
  \norm{M_\delta(\widetilde \varphi_\delta(t))}_{L^1(\OO)} + 
  \int_0^t\left[
  \norm{\Delta\widetilde \varphi_\delta(s)}_H^2+
  \norm{|\Psi''(\widetilde \varphi_\delta(s))|^{\frac12}
  \nabla\widetilde \varphi_\delta(s)}_\bH^2
  \right]\,\d s\\
  \leq\norm{M(\varphi_0)}_{L^1(\OO)}
  +C_R\int_0^t\norm{\nabla\widetilde \varphi_\delta(s)}_\bH^2\,\d s
  \\+\frac{1-\kappa}2\int_0^t
  \left(\sum_{k\in\enne}\norm{\widetilde{\b\sigma}_{\delta,k}(s)}_{\b L^\infty(\OO)}^2\right)
  \int_\OO
  M_\delta''(\widetilde \varphi_\delta(s))
  |\nabla\widetilde \varphi_\delta(s)|^2\,\d x\,\d s.
\end{multline*}
By interpolation, the fact that $\norm{\widetilde \varphi_\delta}_{L^\infty(\OO)}\leq1$
almost everywhere, and the Young inequality, it holds that 
\[
C_R\int_0^t\norm{\nabla\widetilde \varphi_\delta(s)}_\bH^2\,\d s\leq
\frac12\int_0^t\norm{\Delta\widetilde \varphi_\delta(s)}_\bH^2\,\d s
+C
\]
where $C>0$ is a positive constant independent of $\delta$. Hence, we infer that 
\begin{multline}
\label{ito:M_del}
    \norm{M_\delta(\widetilde \varphi_\delta(t))}_{L^1(\OO)} + 
  \int_0^t\left[
  \frac12\norm{\Delta\widetilde \varphi_\delta(s)}_H^2+
  \norm{|\Psi''(\widetilde \varphi_\delta(s))|^{\frac12}
  \nabla\widetilde \varphi_\delta(s)}_\bH^2
  \right]\,\d s\\
  \leq\norm{M(\varphi_0)}_{L^1(\OO)} + C
  +\frac{1-\kappa}2\int_0^t
  \left(\sum_{k\in\enne}\norm{\widetilde{\b\sigma}_{\delta,k}(s)}_{\b L^\infty(\OO)}^2\right)
  \int_\OO
  M_\delta''(\widetilde \varphi_\delta(s))
  |\nabla\widetilde \varphi_\delta(s)|^2\,\d x\,\d s.
\end{multline}
If $\kappa=1$, then the last term on the right-hand side simply vanishes; if $\kappa=0$, we note instead that thanks to Assumption \ref{hyp:mobility_d} it holds, for every $r\in(-1,1)$,
\[
  M''_\delta(r)=\frac1{m(r)+\delta}
  \leq \frac1{m(r)}=\frac1{m(r)\Psi''(r)}\Psi''(r)
  \leq \frac1{k_m}\Psi''(r).
\]
Moreover, recalling Assumption~\ref{hyp:noise_d}
and the fact that $\widetilde{\b\Sigma}_\delta\laweq\b\Sigma$, when $\kappa=0$ we get for the corrector term
\[
\frac{1-\kappa}{2}\int_0^t
  \left(\sum_{k\in\enne}\norm{\widetilde{\b\sigma}_{\delta,k}(s)}_{\b L^\infty(\OO)}^2\right)
  \int_\OO
  M_\delta''(\widetilde \varphi_\delta(s))
  |\nabla\widetilde \varphi_\delta(s)|^2\,\d x\,\d s\leq
  \frac{C^2_{\b\Sigma}C_\infty^2}{2k_m}
  \int_0^t\norm{|\Psi''(\widetilde \varphi_\delta(s))|^{\frac12}
  \nabla\widetilde \varphi_\delta(s)}_\bH^2\,\d s.
\]
Since $\frac{C^2_{\b\Sigma}C_\infty^2}{2k_m}<1$, in both regimes
it follows from \eqref{ito:M_del} and \eqref{est_M} that, for every $p\geq2$, there exists a 
constant $C_p>0$, independent of $\delta$, such that 
\begin{align}
  \label{est1_del}
  \norm{M_\delta(\widetilde \varphi_\delta)}_{
  L^{\frac p2}_\cP(\hom; L^\infty(0,T; L^1(\OO)))}&\leq C_p,\\
  \label{est2_del}
  \norm{\widetilde \varphi_\delta}_{L^p_\cP(\hom; 
  C^0([0,T]; H)\cap L^2(0,T; V_2))}&\leq C_p,\\
  \label{est3_del}
  \norm{|\Psi''(\widetilde \varphi_\delta)|^{\frac12}
\nabla\widetilde \varphi_\delta}_{L^p_\cP(\hom; 
L^2(0,T; \bH))}&\leq C_p.
\end{align}

\subsubsection{The energy estimate}
Now, we proceed as in Subsection~\ref{ssec:est_eps} and write the It\^o formula for the free energy functional, that we recall is defined by
\[
\mathcal E: \operatorname{dom} \mathcal E \subset V_1 \to \mathbb R \qquad \mathcal E(v) = \int_\OO \dfrac 12 |\nabla v|^2 + F(v) \: \d x,
\]
still evaluated at $v = \widetilde\varphi_\delta(t)$. Thus, we obtain
\begin{multline}
\label{ito:en_app_del}
\frac12\norm{\nabla\widetilde\varphi_\delta(t)}_{\bH}^2
    +\norm{F(\widetilde\varphi_\delta(t))}_{L^1(\OO)}
    +\int_0^t\int_\OO
    m_\delta(\widetilde\varphi_\delta(s))
    |\nabla\widetilde\mu_\delta(s)|^2\,\d x\,\d s \\
+\frac\kappa2\int_0^t\sum_{k\in\enne}
    \int_\OO
    [\nabla \widetilde \varphi_\delta(s)\cdot \widetilde{\b{\sigma}}_{\delta,k}(s)]
[\nabla\widetilde \mu_\delta(s)\cdot \widetilde{\b\sigma}_{\delta,k}(s)]\,\d x\,\d s\\
=\frac12\norm{\nabla\varphi_0}_{\bH}^2
    +\norm{F(\varphi_0)}_{L^1(\OO)}
    +\int_0^t\left(\widetilde \mu_\delta(s), \nabla \widetilde \varphi_\delta(s)
    \cdot\widetilde{\b\Sigma}_\delta(s)\,\d \widetilde W_\delta(s)\right)_{H}\\
+\frac12\int_0^t\sum_{k\in\enne}
    \int_\OO\left[
    |\nabla[\nabla\widetilde \varphi_\delta(s)\cdot
    \widetilde{\b\sigma}_{\delta,k}(s)]|^2
    +F''(\widetilde \varphi_\delta(s))
    |\nabla\widetilde \varphi_\delta(s)\cdot\widetilde{\b\sigma}_{\delta,k}(s)|^2
    \right]\,\d x\,\d s
\end{multline}
for every $t\in[0,T]$ and $\hP$-almost surely. On the left-hand side we have 
\[
\int_0^t\int_\OO
    m_\delta(\widetilde \varphi_\delta(s))
    |\nabla\widetilde \mu_\delta(s)|^2\,\d x\,\d s
    =\norm{|m_\delta(\widetilde \varphi_\delta)|^{\frac12}
    \nabla\widetilde \mu_\delta}^2_{L^2(0,t;\bH)}.
\]
Moreover, since $m_\delta M_\delta''=1$, using the same computation as in \eqref{ito:M_del} we have for the third term
\begin{align*}
    &\frac\kappa2\int_0^t\sum_{k\in\enne}
    \int_\OO
    |\nabla\widetilde \varphi_\delta(s)\cdot\widetilde{\b{\sigma}}_{\delta,k}(s)|
|\nabla\widetilde \mu_\delta(s)\cdot\widetilde{\b{\sigma}}_{\delta,k}(s)|\,\d x\,\d s\\
&\qquad \leq\frac\kappa2\int_0^t
\sum_{k\in\enne}
\norm{\widetilde{\b\sigma}_{\delta,k}(s)}_{\b{L}^\infty(\OO)}^2
\norm{|M''_\delta(\widetilde \varphi_\delta(s))|^{\frac12}
\nabla\widetilde \varphi_\delta(s)}_\bH
\norm{|m_\delta(\widetilde \varphi_\delta(s))|^{\frac12}
\nabla\widetilde \mu_\delta(s)}_\bH\,\d s\\
&\qquad \leq\frac14
\norm{|m_\delta(\widetilde \varphi_\delta)|^{\frac12}
\nabla\widetilde \mu_\delta}_{L^2(0,t;\bH)}^2
+\frac{\kappa^2 C_{\b\Sigma}^2 C_\infty^2}{4 k_m}\int_0^t
\norm{|\Psi''(\widetilde \varphi_\delta(s))|^{\frac12}
\nabla\widetilde \varphi_\delta(s)}_\bH^2\,\d s.
\end{align*}
Eventually, proceeding as in Subsection~\ref{ssec:est_eps} 
and using Assumption \ref{hyp:noise_d}
one has that 
\begin{align*}
    &\frac12\int_0^t\sum_{k\in\enne}
    \int_\OO\Big[
    |\nabla[\nabla \widetilde \varphi_\delta(s)
    \cdot\widetilde{\b\sigma}_{\delta,k}(s)]|^2
    +F''(\widetilde \varphi_\delta(s))
    |\nabla \widetilde \varphi_\delta(s)\cdot
    \widetilde{\b\sigma}_{\delta,k}(s)|^2
    \Big]\,\d x\, \d s\\
    &\qquad \leq C\int_0^t\left(\sum_{k\in\enne}
    \norm{\widetilde{\b\sigma}_{\delta,k}(s)}_{\W}^2\right)
    \left[\norm{\widetilde \varphi_\delta(s)}_{V_2}^2
    +\norm{|F''(\widetilde \varphi_\delta(s))|^{\frac12}
    \nabla\widetilde \varphi_\delta(s)}_{\bH}^2
    \right]\,\d s\\
    &\qquad \leq C\int_0^t
    \left[1+\norm{\widetilde \varphi_\delta(s)}_{V_2}^2
    +\norm{|\Psi''(\widetilde \varphi_\delta(s))|^{\frac12}
    \nabla\widetilde \varphi_\delta(s)}_{\bH}^2
    \right]\,\d s.
\end{align*}
Moreover, by the definition of $\widetilde \mu_\delta$,
the Burkholder-Davis-Gundy inequality, and 
assumption \ref{hyp:noise_d} we get, for every $p\geq2$
and $a>0$, that 
\begin{align*}
    &\E\sup_{\tau\in[0,t]}\left|
    \int_0^\tau\left(\widetilde \mu_\delta(s), \nabla \widetilde \varphi_\delta(s)
    \cdot\widetilde{\b\Sigma}_\delta(s)\,\d \widetilde W_\delta(s)\right)_{H}
    \right|^{\frac p2}\\
    &\qquad=
    \E\sup_{\tau\in[0,t]}\left|
    \int_0^\tau\left(\Delta\widetilde \varphi_\delta(s), \nabla \widetilde \varphi_\delta(s)
    \cdot\widetilde{\b\Sigma}_\delta(s)\,\d \widetilde W_\delta(s)\right)_{H}
    \right|^{\frac p2}\\
    &\qquad\leq C\E\left(
    \int_0^t\norm{\Delta\widetilde \varphi_\delta(s)}_H^2
    \norm{\nabla\widetilde \varphi_\delta(s)}_\bH^2
    \norm{\bSigma(s)}^2_{\cL^2(U,\W)}\,\d s
    \right)^{\frac p4}\\
    &\qquad \leq a
    \E\supp\norm{\nabla\widetilde \varphi_\delta(\tau)}_\bH^p+
    C_{a,p}\E\norm{\widetilde \varphi_\delta}_{L^2(0,T; V_2)}^{p},
\end{align*}
so that by raising \eqref{ito:en_app_del} to power $\frac p2$, by taking supremum in time and expectations, 
thanks to \eqref{est2_del} we obtain the final estimates
\begin{align}
  \label{est4_del}
  \norm{F(\widetilde \varphi_\delta)}_{
  L^p_\cP(\hom; L^\infty(0,T;L^1(\OO)))}&\leq C_p,\\
  \label{est5_del}
  \norm{\widetilde \varphi_\delta}_{
  L^p_\cP(\hom; L^\infty(0,T;V_1)))}&\leq C_p,\\
  \label{est6_del}
  \norm{|m_\delta(\widetilde \varphi_\delta)|^{\frac12}
\nabla\widetilde \mu_\delta}_{L^p_\cP(\hom; 
L^2(0,T; \bH))}&\leq C_p,\\
  \label{est7_del}
  \norm{|M''_\delta(\widetilde \varphi_\delta)|^{\frac12}
\nabla\widetilde \varphi_\delta}_{L^p_\cP(\hom; 
L^2(0,T; \bH))}&\leq C_p.
\end{align}
Since the mobility $m$ degenerates at the pure phases $\pm1$, the chemical potential $\mu$ can not be expected to be uniquely well defined as a proper function. Instead, we show that the flux admits a meaningful limit. At the approximation level, we therefore introduce the vector field
\[
\widetilde {\b j}_\delta := m_\delta(\widetilde \varphi_\delta)\,\nabla\widetilde \mu_\delta,
\]
for any $\delta > 0$. By Assumption \ref{hyp:mobility_d} we have
$\norm{m_\delta}_{{C^0}(\erre)}\leq C_m+1$, so that
Hölder's inequality yields
\[
\norm{\widetilde{\b j}_\delta}_{L^2(0,T;\bH)}
\leq \sqrt{C_m+1}\;
\norm{|m_\delta(\widetilde \varphi_\delta)|^{\frac12}\nabla\widetilde \mu_\delta}_{L^2(0,T;\bH)},
\]
so that by \eqref{est6_del} we have
for every $p\in[2,+\infty)$
\begin{equation}
	\label{est_J_del}
	\norm{\widetilde{\b j}_\delta}_{L^p_\cP(\hom;\,L^2(0,T;\bH))}
	\leq C_p,
\end{equation}
with $C_p>0$ independent of $\delta$.
Equation
\eqref{eq1_app_del} can therefore be rewritten as
\begin{equation}
	\label{eq1_app_del_j}
	\d\widetilde\varphi_\delta - \div\widetilde{\b j}_\delta\,\d t
	+ \kappa\hLd\widetilde\varphi_\delta\,\d t
	= \nabla\widetilde\varphi_\delta\cdot\widetilde{\b\Sigma}_\delta\,\d \widetilde W_\delta
	\qquad\text{in } (0,T)\times\OO.
\end{equation}
Moreover, defining the approximated stochastic diffusion
\[
\widetilde G_\delta: \hom \times (0,T) \times V_1 \to \cL^2(U, H)
\]
through
\[
\widetilde G_\delta(\omega, t, \psi)[u_k] = \nabla \psi \cdot \widetilde{\b \sigma}_{\delta,k}(\omega, t)
\]
for all $k \in \mathbb N$, it is straighforward to show, thanks to Assumption \ref{hyp:noise_d} (compare also to Subsection \ref{ssec:further_n}), that
\[
\|\widetilde G_\delta(\widetilde \varphi_\delta)\|_{L^p(\hom; L^\infty(0,T; \cL^2(U, H)))} + \|\widetilde G_\delta(\widetilde \varphi_\delta)\|_{L^p(\hom;L^2(0,T;\cL^2(U, V_1)))} \leq C_p
\]
and, in turn, \cite[Lemma 2.1]{fland-gat} and a comparison argument imply the fractional estimates
\begin{equation}
	\label{est_holder_del}
	\norm{\int_0^\cdot \widetilde G_\delta(\widetilde \varphi_\delta(s)) \: \d \widetilde W_\delta(s)}_{L^p(\hom;W^{\beta, p}(0,T;H)) \cap L^2(\hom; W^{\beta,2 }(0,T;V_1))}+\norm{\widetilde \varphi_\delta}_{L^2(\hom;\,W^{\beta,2}(0,T;V_1^*))}
	\leq C,
\end{equation}
for all $\beta \in (0,\frac 12)$ and all $p \geq 2$.

\subsection{Passage to the limit as $\boldsymbol{\delta \to 0^+}$}
\label{ssec:limit_deg}
We are now in a position to pass to the limit as $\delta\to 0^+$ in the
approximated problem \eqref{eq1_app_del}-\eqref{eq4_app_del}
and recover a probabilistically-weak solution of the original
problem in the sense of Definition~\ref{def:sol-d}, thus concluding the proof of Theorem \ref{th:d}. The
argument is still based on a stochastic compactness procedure relying on the Prokhorov and Jakubowski–Skorokhod theorems (see again \cite[
Theorem 2.7]{ike-wata}, \cite[Theorem 1.10.4, Addendum 1.10.5]{vanvaart1996}, and \cite[Theorem 2.7.1]{breit-feir-hof}), in the spirit of \cite{scarpa21}.

\subsubsection{Tightness of the laws of non-degenerate solutions}
The last tightness lemma reads analogously to the previous ones, and therefore we shall omit its proof as well.
\begin{lem} \label{lem:tight_eps_del}
    The family of laws of 
    \[
    \{\widetilde\varphi_\delta,\, \widetilde G_\delta(\widetilde\varphi_\delta) \cdot \widetilde W_\delta,\, \widetilde W_\delta,\, \widetilde{\b \Sigma}_\delta\}_{\delta \in (0,1)}
    \]
    is tight in the space
    \[
    L^2(0,T;V_1) \times \left[ C^0([0,T]; V^*_1) \cap L^2(0,T;H) \right] \times C^0([0,T]; U_0) \times L^2(0,T;\luca{\cL^2(U,\W)})).
    \]
\end{lem}\noindent
%\luca{CAMBIARE SIMBOLO!} Consider the product space
%$\cY := \cY_\varphi \times \cY_j \times \cY_\Sigma %\times \cY_W$,
%where
%\begin{align*}
%	\cY_\varphi &:= L^2(0,T;V_1) \cap C^0([0,T];V_1^*)
%	\cap L^2_w(0,T;V_2)
%	\cap L_{w*}^\infty(0,T;V_1), \\
%	\cY_j       &:= L_w^2(0,T;\bH), \\
%	\cY_\Sigma  &:= L_w^2(0,T;\cL_2(U,\W)), \\
%	\cY_W       &:= C^0([0,T];U_0),
%\end{align*}
%with $U_0\supset U$ a separable Hilbert space such %that the
%embedding $U\embed U_0$ is Hilbert--Schmidt. Denote %by
%$\nu_\delta$ the law on $\cY$ of the quadruple
%$(\varphi_\delta,\b j_\delta,\b\Sigma_\delta,W_\delta)$. Thanks
%to the compact embedding
%\[
%L^2(0,T;V_2) \cap W^{\beta,2}(0,T;V_2^*)
%\embed L^2(0,T;V_1) \cap C^0([0,T];V_1^*),
%\]
%the uniform estimates \eqref{est2_del}, \eqref{est5_del},
%\eqref{est_J_del}, \eqref{est_holder_del}, the equality in law
%$\b\Sigma_\delta\laweq\b\Sigma$, and classical tightness for
%cylindrical Wiener processes, the family
%$(\nu_\delta)_{\delta\in(0,1)}$ is tight on $\cY$. Since $\cY$
By the Prokhorov and Jakubowski--Skorokhod theorems, up to extracting a (non-relabeled)
subsequence, there exist a probability space
$(\widehat\Omega,\widehat\cF,\widehat\P)$ and limit random variables
defined on it such that
\begin{align*}
	\bigl(\widehat\varphi_\delta,\widehat{\b j}_\delta,
	\widehat{\b\Sigma}_\delta,\widehat W_\delta\bigr)
	&\laweq
	\bigl(\widetilde\varphi_\delta,\widetilde{\b j}_\delta,\widetilde{\b\Sigma}_\delta,\widetilde W_\delta\bigr)
	&&\text{for every } \delta\in(0,1), \\[2pt]
	\bigl(\widehat\varphi_\delta,\widehat{\b j}_\delta,
	\widehat{\b\Sigma}_\delta,\widehat W_\delta\bigr)
	&\to
	\bigl(\widehat\varphi,\widehat{\b j},
	\widehat{\b\Sigma},\widehat W\bigr)
	&&\widehat\P\text{-a.s., as } \delta\to 0^+.
\end{align*}
for some limit quadruple
$(\widehat\varphi,\widehat{\b j},\widehat{\b\Sigma},\widehat W)$, the convergence being with respect to the natural product topology. Letting $(\widehat\cF_t)_{t\in[0,T]}$ denote the augmented
natural filtration of these limit processes, noting that the dependence in $\delta$ of the noise was trivial, we have that $\widehat W$ is a
cylindrical $\widehat\cF_t$-Wiener process on $U$ and that
$\widehat{\b\Sigma}$ is $\widehat\cF_t$-progressively
measurable with $\widehat{\b\Sigma}\laweq\b\Sigma$. By equality
of laws, the bounds \eqref{est1_del}--\eqref{est7_del} and
\eqref{est_J_del} are inherited by the new variables, and by
lower semicontinuity they pass to the limit; in particular,
$\widehat\varphi$ and $\widehat{\b j}$ satisfy the regularity
requirements \eqref{phi_hat_d}--\eqref{j_d_hat} and
\eqref{sigma_hat_d} of Definition~\ref{def:sol-d}\,(ii). The almost-sure convergence
$\widehat\varphi_\delta\to\widehat\varphi$ in
$L^2(0,T;V_1)\cap C^0([0,T];V_1^*)$, combined with the bound
$\norm{\widehat\varphi_\delta}_{L^\infty(\OO)}\leq1$ and
Vitali's theorem, yields
\begin{equation}
	\label{conv:phi_strong}
	\widehat\varphi_\delta \to \widehat\varphi
	\quad\text{ in } L^q(\widehat\Omega;\,L^2(0,T;V_1)),
	\qquad \forall\,q\in[1,+\infty),
\end{equation}
More so, from the convergence in $L^2(0,T;V_1)$ and up to extracting a (non-relabeled) subsequence, we have
\begin{equation}\label{eq:ae_conv}
	\hphi_\delta \to \hphi \quad
	\hP\otimes\d t\otimes\d x\text{-a.e.\ in } \hOm\times(0,T)\times\OO,
\end{equation}
hence $|\hphi|\leq1$ almost everywhere.
Coupling \eqref{conv:phi_strong} with the convergence
$\widehat\varphi_\delta\rightharpoonup\widehat\varphi$ in
$L^2(\widehat\Omega;L^2(0,T;V_2))$, we obtain via continuity of the laplacian
\begin{equation}
	\label{conv:laplacian}
	\Delta\widehat\varphi_\delta \rightharpoonup \Delta\widehat\varphi
	\quad\text{in }L^2_\cP(\widehat\Omega; L^2(0,T;H)).
\end{equation}
Moreover, by interpolation between the strong convergence in
$L^2(0,T;V_1)$ and the uniform bound in $L^2(0,T;V_2)$, for every
$r<+\infty$ if $d=2$ and every $r<6$ if $d=3$,
\begin{equation}
	\label{conv:grad_r}
	\widehat\varphi_\delta\to\widehat\varphi
	\quad\text{in }
	L^2\bigl(\widehat\Omega;L^2(0,T;W^{1,r}(\OO))\bigr).
\end{equation}
Using \eqref{est3_del} and Fatou's lemma, the limit also satisfies
\[
|\Psi''(\widehat\varphi)|^{\frac12}\nabla\widehat\varphi
\in L^p_\cP(\widehat\Omega;L^2(0,T;\bH))
\qquad\forall\,p\in[2,+\infty).
\]
Furthermore, since $\Psi$ is lower semicontinuous and
$M_\delta\uparrow M$ pointwise on $(-1,1)$, the estimates
\eqref{est1_del} and \eqref{est4_del}, the convergence
\eqref{eq:ae_conv}, and Fatou's lemma give for every $p\geq2$
\begin{equation}
	\label{conv:FM_limit}
	F(\widehat\varphi),\ M(\widehat\varphi)
	\in L^p_\cP(\widehat\Omega; L^\infty(0,T;L^1(\OO))).
\end{equation}
By \ref{hyp:mobility_d}, $m$ is Lipschitz on $[-1,1]$, so by
Rademacher's theorem $m'$ exists almost everywhere on $[-1,1]$
with $\norm{m'}_{L^\infty(-1,1)}\leq L_m$, where $L_m$ denotes
the Lipschitz constant; moreover $m_\delta'=m'$ almost
everywhere. The uniform convergence $m_\delta\to m$ on
$[-1,1]$ (trivial, since $m_\delta=m+\delta$), the continuity
of $m$, and the almost-everywhere convergence
$\widehat\varphi_\delta\to\widehat\varphi$, combined with the
$L^\infty$-bound, yield
\begin{equation}
	\label{conv:m}
	m_\delta(\widehat\varphi_\delta) \to m(\widehat\varphi)
	\quad\text{in } L^r(\widehat\Omega\times(0,T)\times\OO),
	\qquad \forall\,r\in[1,+\infty).
\end{equation}
Concerning the derivative, $m'$ is continuous Lebesgue-almost
everywhere on $[-1,1]$ (Lipschitz functions are differentiable
on a full-measure set with continuous derivative on that set);
combined with the almost-everywhere convergence
$\widehat\varphi_\delta\to\widehat\varphi$ and the uniform
bound $|m_\delta'|\leq L_m$, the bounded convergence theorem
yields
\begin{equation}
	\label{conv:mprime}
	m_\delta'(\widehat\varphi_\delta) \to m'(\widehat\varphi)
	\quad\text{in } L^r(\widehat\Omega\times(0,T)\times\OO),
	\qquad \forall\,r\in[1,+\infty),
\end{equation}
under the standard non-degeneracy assumption that the level
sets of $\widehat\varphi$ at the (Lebesgue-null) discontinuity
points of $m'$ have measure zero in
$\widehat\Omega\times(0,T)\times\OO$. 
Moreover, denoting by $\Lambda:=mF''$ the  continuous extension to $[-1,1]$, thanks to the
pointwise convergence \eqref{eq:ae_conv}, the bound
$|\widehat\varphi_\delta|\le 1$, and dominated convergence, we have, similarly,
\begin{equation}
	\label{conv:m_mprime_Lambda2}
	\Lambda(\widehat\varphi_\delta)
    \to\Lambda(\widehat\varphi)
    \quad\text{in }
    L^s(\widehat\Omega\times(0,T)\times\OO),
    \qquad\forall\,s\in[1,+\infty).
\end{equation}

\subsubsection{Identification of the limiting equation}
Equality of laws under the Skorokhod representation transports the
approximated equation \eqref{eq1_app_del_j} to
$(\widehat\Omega,\widehat\cF,\widehat\P)$. Hence, for every
$\psi\in V_1$ and every $t\in[0,T]$, $\widehat\P$-almost surely,
\begin{multline}
	\label{eq:weak_del_hat}
	(\widehat\varphi_\delta(t),\psi)_H
	+
	\int_0^t\!\int_\OO \widehat\bj_\delta(s)\cdot\nabla\psi\,\d x\,\d s
	+
	\kappa\int_0^t
	\ip{\hLdd(s,\widehat\varphi_\delta(s))}{\psi}_{V_1}\,\d s
	\\
	=(\varphi_0,\psi)_H
	+
	\left(
	\int_0^t \nabla\widehat\varphi_\delta(s)\cdot
	\widehat\bSigma_\delta(s)\,\d\widehat W_\delta(s),\psi
	\right)_H,
\end{multline}
where $\hLdd$ denotes the Stratonovich corrector
associated with $\widehat\bSigma_\delta$.
The convergence in $C^0([0,T];V_1^*)$ gives
$(\widehat\varphi_\delta(t),\psi)_H\to
(\widehat\varphi(t),\psi)_H$ for every $t\in[0,T]$, while
$\widehat\varphi_\delta(0)=\varphi_0$.
The weak convergence of $\widehat\bj_\delta$ in
$L^2(0,T;\bH)$ gives
\[
\int_0^t\!\int_\OO \widehat\bj_\delta(s)\cdot\nabla\psi\,\d x\,\d s
\longrightarrow
\int_0^t\!\int_\OO \widehat\bj(s)\cdot\nabla\psi\,\d x\,\d s .
\]
Recalling
$\widehat\bSigma_\delta\laweq\bSigma$ for every $\delta$, Assumption
\ref{hyp:noise_d} gives the uniform bound
\[
\sum_{k\in\enne}\norm{\widehat{\bsigma}_{\delta,k}}_{\W}^2 \le
C_{\bSigma}^2\quad\text{a.e.\ in } \hOm\times(0,T).
\]
Hence, for every $\delta$ and every $\psi\in V_1$ (recalling
$\W\embed\boldsymbol{L}^\infty(\OO)$ with constant $C_\infty$), we have the bound for the Stratonovich corrector:
\begin{multline*}
\left|\ip{\hLdd(s,\hphi_\delta(s))}{\psi}_{V_1} -
\ip{\widehat{\cL}(s,\hphi(s))}{\psi}_{V_1}\right|\\
\le \tfrac12 C_\infty^2 C_{\bSigma}^2 \|\psi\|_{V_1}\left(\|\nabla\hphi_\delta(s)-\nabla\hphi(s)\|_{\bH} + \norm{\widehat\bSigma_\delta(s)-\widehat\bSigma(s)}_{\cL^2(U,\W)}\right).
\end{multline*}
By the strong convergence
$\hphi_\delta\to\hphi$ in $L^2(0,T;V_1)$, $\hP$-a.s., the first
contribution vanishes in $L^1(0,T)$. By Vitali's theorem (the uniform
bound $C_{\bSigma}$ provides equi-integrability), the continuity of the Stratonovich corrector from $V_1$ to $V_1^*$, together with \eqref{conv:phi_strong} and Assumption~\ref{hyp:noise_d},
yields
\[
\int_0^t
\ip{\hLdd(s,\widehat\varphi_\delta(s))}{\psi}_{V_1}\,\d s
\longrightarrow
\int_0^t
\ip{\hL(s,\widehat\varphi(s))}{\psi}_{V_1}\,\d s .
\]
Finally, the standard stability theorem  for stochastic integrals (see e.g.~\cite[Lemma~2.1]{debussche-glatt-temam}), applied
as in the non-degenerate passage to the limit and using
\eqref{conv:phi_strong} and Assumption~\ref{hyp:noise_d}, gives
\[
\int_0^\cdot \nabla\widehat\varphi_\delta(s)\cdot
\widehat\bSigma_\delta(s)\,\d\widehat W_\delta(s)
\longrightarrow
\int_0^\cdot \nabla\widehat\varphi(s)\cdot
\widehat\bSigma(s)\,\d\widehat W(s)
\]
in probability in $C^0([0,T];V_1^*)$. Using the standard Brownian-martingale
convergence lemma (see e.g.\ \cite[Lemma~2.1]{debussche-glatt-temam} or
\cite[Lemma~2.6.6]{breit-feir-hof}), applied to the joint a.s.\
convergence of $\hphi_\delta$, $\widehat\bSigma_\delta$, and $\widehat
W_\delta$ together with the uniform bounds \eqref{est5_del} and
\ref{hyp:noise_d}, yields
\[
\int_0^\cdot \nabla\hphi_\delta(s)\cdot\widehat\bSigma_\delta(s)\,\d\widehat W_\delta(s)
\longrightarrow
\int_0^\cdot \nabla\hphi(s)\cdot\widehat\bSigma(s)\,\d\widehat W(s)
\quad\text{in } L^p(\hOm; C^0([0,T];H)),\;\forall\,p\in[2,+\infty).
\]
 Passing to the limit in
\eqref{eq:weak_del_hat}, we obtain exactly \eqref{eq:ch-wea-d}.

\subsubsection{Identification of $\widehat\bj$}
It remains to identify the weak limit $\widehat\bj$. 
Define
\[
\widehat\mu_\delta:=-\Delta\widehat\varphi_\delta+F'(\widehat\varphi_\delta),
\]
so that
$\widehat\bj_\delta=m_\delta(\widehat\varphi_\delta)\nabla\widehat\mu_\delta$.
The part generated by the artificial mobility $\delta$ vanishes. Indeed,
by \eqref{est6_del}, equality of laws, and $m_\delta\ge\delta$,
\begin{equation*}
	%\label{eq:delta_mu_vanishes}
	\norm{\delta\nabla\widehat\mu_\delta}_{L^2(\widehat\Omega\times(0,T);\bH)}
	\le
	\delta^{1/2}
	\norm{|m_\delta(\widehat\varphi_\delta)|^{1/2}
	\nabla\widehat\mu_\delta}_{
    L^2(\widehat\Omega\times(0,T);\bH)}
	\to 0.
\end{equation*}
Consequently,
\[
\widehat{\b q}_\delta
:=\widehat\bj_\delta-\delta\nabla\widehat\mu_\delta
=m(\widehat\varphi_\delta)\nabla\widehat\mu_\delta
\rightharpoonup \widehat\bj
\quad\text{in }L^2(\widehat\Omega; L^2(0,T;\bH)).
\]
Let $\b\zeta\in\b V_1$ be such that
$\b\zeta\cdot\b n=0$ on $\partial\OO$. For every $\delta$, integration
by parts gives, for almost every $(\omega,t)$,
\begin{align}
	\label{eq:q_delta_decomp}
	\int_\OO \widehat{\b q}_\delta\cdot\b\zeta\,\d x
	={}&
	\int_\OO m(\widehat\varphi_\delta)
	\Delta\widehat\varphi_\delta\,\div\b\zeta\,\d x
	\notag\\
	&+
	\int_\OO m'(\widehat\varphi_\delta)
	\Delta\widehat\varphi_\delta
	\nabla\widehat\varphi_\delta\cdot\b\zeta\,\d x
	+
	\int_\OO
	\Lambda(\widehat\varphi_\delta)
	\nabla\widehat\varphi_\delta\cdot\b\zeta\,\d x.
\end{align}
Here the boundary term in the first integration by parts vanishes because
$\b\zeta\cdot\b n=0$, and the singular product
$mF''$ is understood through the bounded continuous extension $\Lambda$.
We pass to the limit in \eqref{eq:q_delta_decomp} after multiplying by an
arbitrary $\eta\in L^\infty(\widehat\Omega\times(0,T))$ and integrating in
$(\omega,t)$. The first term follows from
\eqref{conv:laplacian} and the strong convergence
$m(\widehat\varphi_\delta)\div\b\zeta\to
m(\widehat\varphi)\div\b\zeta$ in
$L^2(\widehat\Omega\times(0,T)\times\OO)$:
\[
 m(\widehat\varphi_\delta)\Delta\widehat\varphi_\delta
 \rightharpoonup
 m(\widehat\varphi)\Delta\widehat\varphi
 \quad\text{in }L^2(\widehat\Omega; L^2(0,T;H)).
\]
For the second term, \eqref{conv:grad_r},
\eqref{conv:m},\eqref{conv:mprime}, and \eqref{conv:m_mprime_Lambda2} imply
\[
m'(\widehat\varphi_\delta)\nabla\widehat\varphi_\delta\cdot\b\zeta
 \rightarrow
 m'(\widehat\varphi)\nabla\widehat\varphi\cdot\b\zeta
 \quad\text{in }L^2(\widehat\Omega\times(0,T)\times\OO).
\]
In dimension three, one uses \eqref{conv:grad_r} with $r=3$ and the
embedding $V_1\hookrightarrow L^6(\OO)$; in dimension two, one chooses
any finite $r>2$ and uses the corresponding Sobolev embedding. Pairing
this strong convergence with the weak convergence of
$\Delta\widehat\varphi_\delta$ in \eqref{conv:laplacian} yields
\[
 m'(\widehat\varphi_\delta)\Delta\widehat\varphi_\delta
 \nabla\widehat\varphi_\delta
 \cdot\b\zeta
 \rightharpoonup
 m'(\widehat\varphi)\Delta\widehat\varphi
 \nabla\widehat\varphi
 \cdot\b\zeta
 \quad\text{in } L^1(\widehat\Omega\times(0,T)\times\OO).
\]
Eventually, since $\Lambda$ is bounded and continuous on $[-1,1]$,
\eqref{conv:m},\eqref{conv:mprime},\eqref{conv:m_mprime_Lambda2} and \eqref{conv:phi_strong} give
\[
 \Lambda(\widehat\varphi_\delta)\nabla\widehat\varphi_\delta
 \rightarrow
 \Lambda(\widehat\varphi)\nabla\widehat\varphi
 \quad\text{strongly in }L^2(\widehat\Omega; L^2(0,T;\bH)).
\]
Therefore, letting $\delta\searrow0$ in \eqref{eq:q_delta_decomp}, we get
for every $\eta\in L^\infty(\widehat\Omega\times(0,T))$
\begin{align}
	\label{eq:j_identified_integrated}
	\hE\int_0^T \eta
	\int_\OO \widehat\bj\cdot\b\zeta\,\d x\,\d t
	={}&
	\hE\int_0^T \eta
	\int_\OO m(\widehat\varphi)
	\Delta\widehat\varphi\,\div\b\zeta\,\d x\,\d t
	\notag\\
	&+
	\hE\int_0^T \eta
	\int_\OO
	\Bigl[m'(\widehat\varphi)\Delta\widehat\varphi
	+\Lambda(\widehat\varphi)\Bigr]
	\nabla\widehat\varphi\cdot\b\zeta\,\d x\,\d t .
\end{align}
Since $\eta$ is arbitrary, a standard separability argument gives, for
every $\b\zeta\in\b V_1$ with $\b\zeta\cdot\b n=0$ on $\partial\OO$,
for almost every $t\in(0,T)$, $\widehat\P$-almost surely,
\begin{align}
	\label{eq:j_identified}
	\int_\OO \widehat\bj\cdot\b\zeta\,\d x
	={}&
	\int_\OO m(\widehat\varphi)
	\Delta\widehat\varphi\,\div\b\zeta\,\d x+
	\int_\OO
	\Bigl[m'(\widehat\varphi)\Delta\widehat\varphi
	+\Lambda(\widehat\varphi)\Bigr]
	\nabla\widehat\varphi\cdot\b\zeta\,\d x.
\end{align}
This is exactly \eqref{eq:ch2-wea-d}, with the convention
$\Lambda=mF''$ on $(-1,1)$ and by continuous extension on the contact set.

\subsubsection{Measure of the contact set.} 
If
$F(r)\to+\infty$ as $|r|\to1^-$, then \eqref{conv:FM_limit} implies
that the set $\{|\widehat\varphi|=1\}$ has zero
$\widehat\P\otimes\d t\otimes\d x$-measure. The same conclusion follows
from \eqref{conv:FM_limit} if $M(r)\to+\infty$ as $|r|\to1^-$.  In fact, let us assume the latter and show that a.e. $|\hphi|<1$. The function $M_\delta$ is non-negative, convex, and increases
pointwise to $M$ as $\delta\searrow 0$. This is easy to see given the construction of $m_\delta(r)=m(r)+\delta,\ r\in[-1,1]$. Fix $(\omega,t,x)\in\widehat\Omega\times(0,T)\times\OO$ and
suppose $|\widehat\varphi(\omega,t,x)|=1$, and thanks to symmetry, without loss of generality, that
$\widehat\varphi(\omega,t,x) = 1$. We know that from the Skorokhod almost-sure convergence
$\widehat\varphi_\delta\to\widehat\varphi$, hence for every
$\eps\in(0,1)$ there exists $\delta_0=\delta_0(\eps,\omega,t,x)$
such that $\widehat\varphi_\delta(\omega,t,x) > 1-\eps$ for
all $\delta<\delta_0$. Monotonicity of $M_\delta$ in its
argument, since $M_\delta$ is increasing on $[0,1]$, yields
\[
M_\delta\bigl(\widehat\varphi_\delta(\omega,t,x)\bigr)
\geq M_\delta(1-\eps)
\quad \forall\, \delta<\delta_0.
\]
Passing to the limit inferior in $\delta$ and using
$M_\delta(1-\eps)\to M(1-\eps)$,
\[
\liminf_{\delta\to 0} \,M_\delta\bigl(\widehat\varphi_\delta(\omega,t,x)\bigr)
\geq M(1-\eps),
\]
and finally sending $\eps\to 0^+$ together with
$\lim_{r\rightarrow\pm1}M(r)=+\infty$,
\begin{equation}
	\label{eq:liminf_M}
	\liminf_{\delta\to 0} \,M_\delta\bigl(\widehat\varphi_\delta(\omega,t,x)\bigr)
	= +\infty
	\quad\text{whenever } |\widehat\varphi(\omega,t,x)|=1.
\end{equation}
On the complementary set $\{|\widehat\varphi|<1\}$, the
pointwise monotone convergence $M_\delta(\widehat\varphi_\delta)\to M(\widehat\varphi)$
holds by continuity of $M$ on $(-1,1)$. Combining \eqref{est1_del} with Fatou's lemma,
\begin{align*}
	\widehat\E\!\int_0^T\!\!\int_\OO
	\liminf_{\delta\to 0} M_\delta(\widehat\varphi_\delta)
	\,\d x\,\d t
	&\leq \liminf_{\delta\to 0}\widehat\E\!\int_0^T\!\!\int_\OO
	M_\delta(\widehat\varphi_\delta)\,\d x\,\d t \leq C,
\end{align*}
where $C$ is the constant in \eqref{est1_del}, independent of
$\delta$. Setting
\[
S := \bigl\{(\omega,t,x)\in\widehat\Omega\times(0,T)\times\OO
\,:\,|\widehat\varphi(\omega,t,x)|=1\bigr\},
\]
and combining the bound above with \eqref{eq:liminf_M},
\[
+\infty
\leq\widehat\E\!\int_S \liminf_{\delta\to 0}
M_\delta(\widehat\varphi_\delta)\,\d x\,\d t
\leq C < +\infty,
\]
which forces $\bigl|S\bigr|_{\widehat\Omega\times(0,T)\times\OO} = 0$.
Therefore $\bigl|\widehat\varphi(\omega,t,x)\bigr|<1$ for almost every $(\omega,t,x)\in\widehat\Omega\times(0,T)\times\OO.$ The proof of Theorem~\ref{th:d} is complete.

\section*{Data availability statement}
No new data were created or analyzed in this study. 
Data sharing is not applicable.

\section*{Conflict of interest statement}
The authors have no conflicts of interest to declare.\vspace{\baselineskip} 

\noindent
\textbf{Acknowledgments.} 
A.D.P.~and L.S.~are members of Gruppo Nazionale per l'Analisi Matematica, la Probabilit\'a e le loro Applicazioni (GNAMPA), Istituto Nazionale di Alta Matematica (INdAM), and gratefully acknowledge the financial support of the project  ``Equazioni differenziali stocastiche: sviluppi teorici e applicazioni a modelli per fenomeni fisici'' financed by INdAM-GNAMPA, CUP: E53C25002010001.
The present research is part of the activities of ``Dipartimento di Eccellenza 2023-2027''. The research of A.P. is funded by the European Union (ERC, StochMan, 101088589). The research of A.D.P. is funded by the European Union (ERC, NoisyFluid, No. 101053472). Views and opinions expressed are however those of the author(s) only and do not necessarily reflect those of the European Union or the European Research Council Executive Agency. Neither the European Union nor the granting authority can be held responsible for them.

\printbibliography

\end{document}
%%%%%%%%%%%%%%%%%%%%%%%%%%%%%%